\documentclass[12pt,a4paper,reqno]{amsart}
\usepackage{amsmath}
\usepackage{amsfonts}
\usepackage{amssymb}
\usepackage{graphicx}
\usepackage{pstricks}
\usepackage{bm} 

\newcommand\eps{\varepsilon}
\newcommand\R{{\mathbf{R}}}
\newcommand\C{{\mathbf{C}}}
\newcommand\Z{{\mathbf{Z}}}

\newcommand\bigO{{\mathcal{O}}}
\theoremstyle{plain}
  \newtheorem{theorem}[subsection]{Theorem}
  \newtheorem{conjecture}[subsection]{Conjecture}
  \newtheorem{proposition}[subsection]{Proposition}
  \newtheorem{lemma}[subsection]{Lemma}
  \newtheorem{corollary}[subsection]{Corollary}

\theoremstyle{remark}
  \newtheorem{remark}[subsection]{Remark}

\theoremstyle{definition}
  \newtheorem{definition}[subsection]{Definition}

\include{psfig}

\begin{document}

\title[Localisation and compactness for Navier-Stokes]{Localisation and compactness properties of the Navier-Stokes global regularity problem}
\author{Terence Tao}
\address{Department of Mathematics, UCLA, Los Angeles CA 90095-1555}
\email{tao@math.ucla.edu}
\subjclass{35Q30, 76D05, 76N10}

\vspace{-0.3in}
\begin{abstract}  In this paper we establish a number of implications between various qualitative and quantitative versions of the global regularity problem for the Navier-Stokes equations, in the periodic, smooth finite energy, smooth $H^1$, Schwartz, or mild $H^1$ categories, and with or without a forcing term.  In particular, we show that if one has global well-posedness in $H^1$ for the periodic Navier-Stokes problem with a forcing term, then one can obtain global regularity both for periodic and for Schwartz initial data (thus yielding a positive answer to both official formulations of the  problem for the Clay Millennium Prize), and can also obtain global almost smooth solutions from smooth $H^1$ data or smooth finite energy data, although we show in this category that fully smooth solutions are not always possible.  Our main new tools are localised energy and enstrophy estimates to the Navier-Stokes equation that are applicable for large data or long times, and which may be of independent interest.
\end{abstract}

\maketitle

\section{Introduction}

The purpose of this paper is to establish some implications between various formulations of the global regularity problem (either with or without a forcing term) for the Navier-Stokes system of equations, including the four formulations appearing in the Clay Millennium Prize formulation \cite{feff} of the problem, and in particular to isolate a single formulation that implies these four formulations, as well as several other natural versions of the problem.  In the course of doing so, we also establish some new local energy and local enstrophy estimates which seem to be of independent interest.

To describe these various formulations, we must first define properly the concept of a solution to the Navier-Stokes problem.  We will need to study a number of different types of solutions, including periodic solutions, finite energy solutions, $H^1$ solutions, and smooth solutions; we will also consider a forcing term $f$ in addition to the initial data $u_0$.  We begin in the classical regime of smooth solutions.  Note that even within the category of smooth solutions, there is some choice in what decay hypotheses to place on the initial data and solution; for instance, one can require that the initial velocity $u_0$ be Schwartz class, or merely smooth with finite energy.  Intermediate between these two will be data which is smooth and in $H^1$.

More precisely, we define:

\begin{definition}[Smooth solutions to the Navier-Stokes system]  A \emph{smooth set of data} for the Navier-Stokes system up to time $T$ is a triplet $(u_0,f,T)$, where $0 < T < \infty$ is a time, the initial velocity vector field $u_0: \R^3 \to \R^3$ and the forcing term $f: [0,T] \times \R^3 \to \R^3$ are assumed to be smooth on $\R^3$ and $[0,T] \times \R^3$ respectively (thus, $u_0$ is infinitely differentiable in space, and $f$ is infinitely differentiable in spacetime), and $u_0$ is furthermore required to be divergence-free:
\begin{equation}\label{div-free-data}
\nabla \cdot u_0 = 0.
\end{equation}
If $f=0$, we say that the data is \emph{homogeneous}.

The \emph{total energy} $E(u_0,f,T)$ of a smooth set of data $(u_0,f,T)$ is defined by the quantity\footnote{We will review our notation for spacetime norms such as $L^p_t L^q_x$, together with sundry other notation, in Section \ref{notation-sec}.}
\begin{equation}\label{energy-def}
E(u_0,f,T) := \frac{1}{2} (\|u_0 \|_{L^2_x(\R^3)} + \|f\|_{L^1_t L^2_x([0,T] \times \R^3)})^2
\end{equation}
and $(u_0,f,T)$ is said to have \emph{finite energy} if $E(u_0,f,T) < \infty$.  We define the \emph{$H^1$ norm} ${\mathcal H^1}(u_0,f,T)$ of the data to be the quantity
$$ {\mathcal H^1}(u_0,f,T) := \| u_0 \|_{H^1_x(\R^3)} + \| f \|_{L^\infty_t H^1_x(\R^3)} < \infty$$
and say that $(u_0,f,T)$ is $H^1$ if ${\mathcal H}^1(u_0,f,T) < \infty$; note that the $H^1$ regularity is essentially one derivative higher than the energy regularity, which is at the level of $L^2$, and instead matches the regularity of the \emph{initial enstrophy}
$$ \frac{1}{2} \int_{\R^3} |\omega_0(t,x)|^2\ dx,$$
where $\omega_0 := \nabla \times u_0$ is the initial vorticity. We say that a smooth set of data $(u_0,f,T)$ is \emph{Schwartz} if, for all integers $\alpha, m, k \geq 0$, one has
$$ \sup_{x \in \R^3} (1+|x|)^k |\nabla_x^\alpha u_0(x)| < \infty$$
and
$$ \sup_{(t,x) \in [0,T] \times \R^3} (1+|x|)^k |\nabla_x^\alpha \partial_t^m f(x)| < \infty.$$
Thus, for instance, the Schwartz property implies $H^1$, which in turn implies finite energy.  We also say that $(u_0,f,T)$ is \emph{periodic} with some period $L>0$ if one has $u_0(x+Lk) = u_0(x)$ and $f(t,x+Lk) = f(t,x)$ for all $t \in [0,T]$, $x \in \R^3$, and $k \in \Z^3$.  Of course, periodicity is incompatible with the Schwartz, $H^1$, or finite energy properties, unless the data is zero.  To emphasise the periodicity, we will sometimes write a periodic set of data $(u_0,f,T)$ as $(u_0,f,T,L)$.  

A \emph{smooth solution to the Navier-Stokes system}, or a \emph{smooth solution} for short, is a quintuplet $(u,p,u_0,f,T)$, where $(u_0,f,T)$ is a smooth set of data, and the velocity vector field $u: [0,T] \times \R^3 \to \R^3$ and pressure field $p: [0,T] \times \R^3 \to \R$ are smooth functions on $[0,T] \times \R^3$ that obey the Navier-Stokes equation
\begin{equation}\label{ns}
\partial_t u + (u \cdot \nabla) u = \Delta u - \nabla p + f
\end{equation}
and the incompressibility property
\begin{equation}\label{div-free}
\nabla \cdot u = 0
\end{equation}
on all of $[0,T] \times \R^3$, and also the initial condition
\begin{equation}\label{initial}
u(0,x) = u_0(x)
\end{equation}
for all $x \in \R^3$.  We say that a smooth solution $(u,p,u_0,f,T)$ has \emph{finite energy} if the associated data $(u_0,f,T)$ has finite energy, and in addition one has\footnote{Following \cite{feff}, we omit the finite energy dissipation condition $\nabla u \in L^2_t L^2_x([0,T] \times\R^3)$ that often appears in the literature, particularly when discussing Leray-Hopf weak solutions.  However, it turns out that this condition is actually automatic from \eqref{finite-energy} and smoothness; see Lemma \ref{energy-est}.  Similarly, from Corollary \ref{ens-bound} we shall see that the $L^2_t H^2_x$ condition in \eqref{finite-enstrophy} is in fact redundant.}
\begin{equation}\label{finite-energy}
 \| u \|_{L^\infty_t L^2_x( [0,T] \times \R^3) } < \infty.
\end{equation}
Similarly, we say that $(u,p,u_0,f,T)$ is $H^1$ if the associated data $(u_0,f,T)$ is $H^1$, and in addition one has
\begin{equation}\label{finite-enstrophy}
 \| u \|_{L^\infty_t H^1_x( [0,T] \times \R^3) } 
 + \| u \|_{L^2_t H^2_x([0,T] \times \R^3) } < \infty.
\end{equation}
We say instead that a smooth solution $(u,p,u_0,f,T)$ is \emph{periodic} with period $L>0$ if the associated data $(u_0,f,T) = (u_0,f,T,L)$ is periodic with period $L$, and if $u(t,x+Lk)=u(t,x)$ for all $t \in [0,T]$, $x \in \R^3$, and $k \in \Z^3$.  (Following \cite{feff}, however, we will not initially directly require any periodicity properties on the pressure.)  As before, we will sometimes write a periodic solution $(u,p,u_0,f,T)$ as $(u,p,u_0,f,T,L)$ to emphasise the periodicity.

We will sometimes abuse notation and refer to a solution $(u,p,u_0,f,T)$ simply as $(u,p)$ or even $u$.  Similarly, we will sometimes abbreviate a set of data $(u_0,f,T)$ as $(u_0,f)$ or even $u_0$ (in the homogeneous case $f=0$).
\end{definition}

\begin{remark} In \cite{feff}, one considered\footnote{The viscosity parameter $\nu$ was not normalised in \cite{feff} to equal $1$, as we are doing here, but one can easily reduce to the $\nu=1$ case by a simple rescaling.} smooth finite energy solutions associated to Schwartz data, as well as periodic smooth solutions associated to periodic smooth data.  In the latter case, one can of course normalise the period $L$ to equal $1$ by a simple scaling argument.  In this paper we will be focused on the case when the data $(u_0,f,T)$ is large, although we will not study the asymptotic regime when $T \to \infty$.  
\end{remark}

We recall the two standard \emph{global regularity} conjectures for the Navier-Stokes equation, using the formulation in \cite{feff}:

\begin{conjecture}[Global regularity for homogeneous Schwartz data]\label{global-schwartz-homog}  Let $(u_0,0,T)$ be a homogeneous Schwartz set of data.  Then there exists a smooth finite energy solution $(u,p,u_0,0,T)$ with the indicated data.
\end{conjecture}

\begin{conjecture}[Global regularity for homogeneous periodic data]\label{global-periodic-homog}  Let $(u_0,0,T)$ be a smooth homogeneous periodic set of data.  Then there exists a smooth periodic solution $(u,p,u_0,0,T)$ with the indicated data.
\end{conjecture}

In view of these conjectures, one can naturally try to extend them to the inhomogeneous case as follows:

\begin{conjecture}[Global regularity for Schwartz data]\label{global-schwartz}  Let $(u_0,f,T)$ be a Schwartz set of data.  Then there exists a smooth finite energy solution $(u,p,u_0,f,T)$ with the indicated data.
\end{conjecture}

\begin{conjecture}[Global regularity for periodic data]\label{global-periodic}  Let $(u_0,f,T)$ be a smooth periodic set of data.  Then there exists a smooth periodic solution $(u,p,u_0,f,T)$ with the indicated data.
\end{conjecture}

As described in \cite{feff}, a positive answer to either Conjecture \ref{global-schwartz-homog} or Conjecture \ref{global-periodic-homog}, or a negative answer to Conjecture \ref{global-schwartz} or Conjecture \ref{global-periodic}, would qualify for the Clay Millennium Prize.  

However, Conjecture \ref{global-periodic} is not quite the ``right'' extension of Conjecture \ref{global-periodic-homog} to the inhomogeneous setting, and needs to be corrected slightly.  This is because there is a technical quirk in the inhomogeneous periodic problem as formulated in Conjecture \ref{global-periodic}, due to the fact that the pressure $p$ is not required to be periodic.  This opens up a Galilean invariance in the problem which allows one to homogenise away the role of the forcing term.  More precisely, we have

\begin{proposition}[Elimination of forcing term]\label{force}  Conjecture \ref{global-periodic} is equivalent to Conjecture \ref{global-periodic-homog}.
\end{proposition}

We establish this fact in Section \ref{homog-sec}.  We remark that this is the only implication we know of that can deduce a global regularity result for the inhomogeneous Navier-Stokes problem from a global regularity result for the homogeneous Navier-Stokes problem.

Proposition \ref{force} exploits the technical loophole of non-periodic pressure.  The same loophole can also be used to easily demonstrate failure of uniqueness for the periodic Navier-Stokes problem (although this can also be done by the much simpler expedient of noting that one can adjust the pressure by an arbitrary constant without affecting \eqref{ns}).  This suggests that in the non-homogeneous case $f \neq 0$, one needs an additional normalisation to ``fix'' the periodic Navier-Stokes problem to avoid such loopholes.  This can be done in a standard way, as follows.
If one takes the divergence of \eqref{ns} and use the incompressibility \eqref{div-free}, one sees that that
\begin{equation}\label{deltap}
 \Delta p = -\partial_i \partial_j (u_i u_j) + \nabla \cdot f
\end{equation}
where we use the usual summation conventions.  If $(u,p,u_0,f,T)$ is a smooth periodic solution, then the right-hand side of \eqref{deltap} is smooth, periodic, and has mean zero.  From Fourier analysis, we see that given any smooth periodic mean zero function $F$, there is a unique smooth periodic mean zero function $\Delta^{-1} F$ with Laplacian equal to $F$.  We then say that the periodic smooth solution $(u,p,u_0,f,T)$ has \emph{normalised pressure} if one has\footnote{Up to the harmless freedom to add a constant to $p$, this normalisation is equivalent to requiring that the pressure be periodic with the same period as the solution $u$.}
\begin{equation}\label{pressure-point}
p = -\Delta^{-1} \partial_i \partial_j (u_i u_j) + \Delta^{-1} \nabla \cdot f.
\end{equation}
We remark that this normalised pressure condition can also be imposed for smooth finite energy solutions (because $\partial_i \partial_j(u_i u_j)$ is a second derivative of an $L^1_x(\R^3)$ function, and $\nabla \cdot f$ is the first derivative of an $L^2_x(\R^3)$ function), but it will turn out that normalised pressure is essentially automatic in that setting anyway; see Proposition \ref{reduction}.

It is well-known that once one imposes the normalised pressure condition, then the periodic Navier-Stokes problem becomes locally well-posed in the smooth category (in particular, smooth solutions are now unique, and exist for sufficiently short times from any given smooth data); see Theorem \ref{lwp-h1}.  Related to this, the Galilean invariance trick that allows one to artificially homogenise the forcing term $f$ is no longer available.  We can then pose a ``repaired'' version of Conjecture \ref{global-periodic}:

\begin{conjecture}[Global regularity for periodic data with normalised pressure]\label{global-periodic-normalised}  Let $(u_0,f,T)$ be a smooth periodic set of data.  Then there exists a smooth periodic solution $(u,p,u_0,f,T)$ with the indicated data and with normalised pressure.
\end{conjecture}

It is easy to see that the homogeneous case $f=0$ of Conjecture \ref{global-periodic-normalised} is equivalent to Conjecture \ref{global-periodic-homog}; see e.g. Lemma \ref{reduction} below.

We now leave the category of classical (smooth) solutions for now, and turn instead to the category of \emph{periodic $H^1$ mild solutions} $(u,p,u_0,f,T,L)$.  By definition, these are functions $u,f: [0,T] \times \R^3/L\Z^3 \to \R^3$, $p: [0,T] \times \R^3/L\Z^3 \to \R$, $u_0: \R^3/L\Z^3 \to \R^3$ with $0 < T, L < \infty$, obeying the regularity hypotheses
\begin{align*}
u_0 &\in H^1_x(\R^3/L\Z^3) \\
f &\in L^\infty_t H^1_x([0,T] \times (\R^3/L\Z^3))\\
u &\in L^\infty_t H^1_x \cap L^2_t H^2_x([0,T] \times (\R^3/L\Z^3))
\end{align*}
with $p$ being given by \eqref{pressure-point}, which obey the divergence-free conditions \eqref{div-free}, \eqref{div-free-data}, and obey the integral form
\begin{equation}\label{duhamel-1}
u(t) = e^{t\Delta} u_0 + \int_0^t e^{(t-t')\Delta} (-(u \cdot \nabla) u - \nabla p + f)(t')\ dt'
\end{equation}
of the Navier-Stokes equation \eqref{ns} with initial condition \eqref{initial}; using the Leray projection $P$ onto divergence-free vector fields, we may also express \eqref{duhamel} equivalently as
\begin{equation}\label{duhamel-2}
u(t) = e^{t\Delta} u_0 + \int_0^t e^{(t-t')\Delta} (P B(u,u) + Pf)(t')\ dt'
\end{equation}
where $B(u,v)$ is the symmetric bilinear form
\begin{equation}\label{buj}
 B(u,v)_i := -\frac{1}{2} \partial_j( u_i v_j + u_j v_i ).
\end{equation}
Similarly, we define \emph{periodic $H^1$ data} to be a quadruplet $(u_0,f,T,L)$ whose $H^1$ norm
$$ {\mathcal H}^1(u_0,f,T,L) := \|u_0\|_{H^1_x((\R^3/L\Z^3))} + \|f\|_{L^\infty_t H^1_x((\R^3/L\Z^3))}$$
is finite, with $u_0$ divergence-free.

Note from Duhamel's formula \eqref{Duhamel-formula} that every smooth periodic solution with normalised pressure is automatically a periodic $H^1$ mild solution.

As we will recall in Theorem \ref{lwp-h1} below, the Navier-Stokes equation is locally well-posed in the periodic $H^1$ category. We can then formulate a global well-posedness conjecture in this category:

\begin{conjecture}[Global well-posedness in periodic $H^1$]\label{global-h1}  Let $(u_0,f,T,L)$ be a periodic $H^1$ set of data.  Then there exists a periodic $H^1$ mild solution $(u,p,u_0,f,T,L)$ with the indicated data.
\end{conjecture}

We may also phrase a quantitative variant of this conjecture:

\begin{conjecture}[\emph{A priori} periodic $H^1$ bound]\label{global-h1-quant}  There exists a function $F: \R^+ \times \R^+ \times \R^+ \to \R^+$ with the property that whenever $(u,p,u_0,f,T,L)$ is a smooth periodic, normalised-pressure solution with $0 < T < T_0 < \infty$ and
$$ {\mathcal H}^1(u_0,f,T,L) \leq A < \infty$$
then
$$ \| u \|_{L^\infty_t H^1_x([0,T] \times \R^3/L\Z^3)} \leq F( A, L, T_0 ).$$
\end{conjecture}

\begin{remark} By rescaling, one may set $L=1$ in this conjecture without any loss of generality; by partitioning the time interval $[0,T_0]$ into smaller sub-intervals we may also simultaneously set $T_0=1$ if desired.  Thus, the key point is that the size of the data $A$ is allowed to be large (for small $A$ the conjecture follows from the local well-posedness theory, see Theorem \ref{lwp-h1}).
\end{remark}

As we shall soon see, Conjecture \ref{global-h1} and Conjecture \ref{global-h1-quant} are actually equivalent.

We now turn to the non-periodic setting.  In Conjecture \ref{global-schwartz}, the hypothesis that the initial data be Schwartz may seem unnecessarily restrictive, given that the incompressible nature of the fluid implies that the Schwartz property need not be preserved over time; also, there are many interesting examples of initial data that are smooth and finite energy (or $H^1$) but not Schwartz.  In particular, one can consider generalising Conjecture \ref{global-schwartz} to data that is merely smooth and $H^1$, or even smooth and finite energy, rather than Schwartz\footnote{We are indebted to Andrea Bertozzi for suggesting these formulations of the Navier-Stokes global regularity problem.} of Conjecture \ref{global-schwartz}.  Unfortunately, the naive generalisation of Conjecture \ref{global-schwartz} (or even Conjecture \ref{global-schwartz-homog}) fails instantaneously in this case:

\begin{theorem}[No smooth solutions from smooth $H^1$ data]\label{counter}  There exists smooth $u_0 \in H^1_x(\R^3)$ such that there does not exist any smooth finite energy solution $(u,p,u_0,0,T)$ with the indicated data for any $T>0$.
\end{theorem}

We prove this proposition in Section \ref{counter-sec}.  At first glance, this proposition looks close to being a negative answer to either Conjecture \ref{global-schwartz} or Conjecture \ref{global-schwartz-homog}, but it relies on a technicality; for smooth $H^1$ data, the second derivatives of $u_0$ need not be square-integrable, and this can cause enough oscillation in the pressure to prevent the pressure from being $C^2_t$ (or the velocity field from being $C^3_t$) at the initial time\footnote{For most evolutionary PDEs, one can gain unlimited time differentiability at $t=0$ assuming smooth initial data by differentiating the PDE in time (cf. the proof of the Cauchy-Kowalesky theorem).  However, the problem here is that the pressure $p$ in the Navier-Stokes equation does not obey an evolutionary PDE, but is instead determined in a non-local fashion from the initial data $u$ (see \eqref{pressure-point}), which prevents one from obtaining much time regularity of the pressure initially.} $t=0$.  This theorem should be compared with the classical local existence theorem of Heywood \cite{heywood}, which obtains smooth solutions for small \emph{positive} times from smooth data with finite enstrophy, but merely obtains continuity at the initial time $t=0$.

The situation is even worse in the inhomogeneous setting; the argument in Theorem \ref{counter} can be used to construct inhomogeneous smooth $H^1$ data whose solutions will now be non-smooth in time at all times, not just at the initial time $t=0$.  Because of this, we will not attempt to formulate a global regularity problem in the inhomogeneous smooth $H^1$ or inhomogeneous smooth finite energy categories.

In the homogeneous setting, though, we can get around this technical obstruction by introducing the notion of an \emph{almost smooth finite energy solution} $(u,p,u_0,f,T)$, which is the same concept as a smooth finite energy solution, but instead of requiring $u, p$ to be smooth on $[0,T] \times \R^3$, we instead require that $u,p$ are smooth on $(0,T] \times \R^3$, and for each $k \geq 0$, the functions $\nabla_x^k u, \partial_t \nabla_x^k u, \nabla_x^k p$ exist and are continuous on $[0,T] \times \R^3$.  Thus, the only thing that almost smooth solutions lack when compared to smooth solutions is a limited amount of time differentiability at the starting time $t=0$; informally, $u$ is only $C^1_t C^\infty_x$ at $t=0$, and $p$ is only $C^0_t C^\infty_x$ at $t=0$.  This is still enough regularity to interpret the Navier-Stokes equation \eqref{ns} in the classical sense, but is not a completely smooth solution.

The ``corrected'' conjectures for global regularity in the homogeneous smooth $H^1$ and smooth finite energy categories are then

\begin{conjecture}[Global almost regularity for homogeneous $H^1$]\label{global-enstrophy-homog}  Let $(u_0,0,T)$ be a smooth homogeneous $H^1$ set of data.  Then there exists an almost smooth finite energy solution $(u,p,u_0,0,T)$ with the indicated data.
\end{conjecture}

\begin{conjecture}[Global almost regularity for homogeneous finite energy data]\label{global-energy-homog}  Let $(u_0,0,T)$ be a smooth homogeneous finite energy set of data.  Then there exists an almost smooth finite energy solution $(u,p,u_0,0,T)$ with the indicated data.
\end{conjecture}

We carefully note that these conjectures only concern \emph{existence} of smooth solutions, and not uniqueness; we will comment on some of the uniqueness issues later in this paper.

Another way to repair the global regularity conjectures in these settings is to abandon smoothness altogether, and work instead with the notion of mild solutions.  More precisely, define a \emph{$H^1$ mild solution} $(u,p,u_0,f,T)$ 
to be fields $u,f: [0,T] \times \R^3 \to \R^3$, $p: [0,T] \times \R^3 \to \R$, $u_0: \R^3 \to \R^3$ with $0 < T< \infty$, obeying the regularity hypotheses
\begin{align*}
u_0 &\in H^1_x(\R^3) \\
f &\in L^\infty_t H^1_x([0,T] \times \R^3)\\
u &\in L^\infty_t H^1_x \cap L^2_t H^2_x([0,T] \times \R^3)
\end{align*}
with $p$ being given by \eqref{pressure-point}, which obey \eqref{div-free}, \eqref{div-free-data}, and \eqref{duhamel-1} (and thus \eqref{duhamel-2}).  Similarly define the concept of $H^1$ data $(u_0,f,T)$.

We then have the following conjectures in the homogeneous setting:

\begin{conjecture}[Global well-posedness in homogeneous $H^1$]\label{global-h1-r3} Let $(u_0,0,T)$ be a homogeneous $H^1$ set of data.  Then there exists a $H^1$ mild solution $(u,p,u_0,0,T)$ with the indicated data.
\end{conjecture}

\begin{conjecture}[\emph{A priori} homogeneous $H^1$ bound]\label{global-h1-quant-r3}  There exists a function $F: \R^+ \times \R^+ \to \R^+$ with the property that whenever $(u,p,u_0,0,T)$ is a smooth $H^1$ solution with $0 < T < T_0 < \infty$ and
$$ \|u_0\|_{H^1_x(\R^3)} \leq A < \infty$$
then
$$ \| u \|_{L^\infty_t H^1_x([0,T] \times \R^3)} \leq F( A, T_0 ).$$
\end{conjecture}

We also phrase a global-in-time variant:

\begin{conjecture}[\emph{A priori} global homogeneous $H^1$ bound]\label{global-h1-quant-global}  There exists a function $F: \R^+ \to \R^+$ with the property that whenever $(u,p,u_0,0,T)$ is a smooth $H^1$ solution with 
$$ \|u_0\|_{H^1_x(\R^3)} \leq A < \infty$$
then
$$ \| u \|_{L^\infty_t H^1_x([0,T] \times \R^3)} \leq F( A ).$$
\end{conjecture}

In the inhomogeneous setting, we will state two slightly technical conjectures:

\begin{conjecture}[Global well-posedness from spatially smooth Schwartz data]\label{global-schwartz-spatial} Let $(u_0,f,T)$ be data obeying the bounds
$$ \sup_{x \in \R^3} (1+|x|)^k |\nabla_x^\alpha u_0(x)| < \infty$$
and
$$ \sup_{(t,x) \in [0,T] \times \R^3} (1+|x|)^k |\nabla_x^\alpha f(x)| < \infty$$
for all $k, \alpha \ge 0$.
Then there exists a $H^1$ mild solution $(u,p,u_0,f,T)$ with the indicated data.
\end{conjecture}

\begin{conjecture}[Global well-posedness from spatially smooth $H^1$ data]\label{global-h1-spatial} Let $(u_0,f,T)$ be an $H^1$ set of data, such that 
$$ \sup_{x \in K} |\nabla_x^\alpha u_0(x)| < \infty$$
and
$$ \sup_{(t,x) \in [0,T] \times K} |\nabla_x^\alpha f(x)| < \infty$$
for all $\alpha \geq 0$ and all compact $K$.  Then there exists a $H^1$ mild solution $(u,p,u_0,f,T)$ with the indicated data.
\end{conjecture}

Needless to say, we do not establish\footnote{Indeed, the arguments here do not begin to address the main issue in any of these conjectures, namely the analysis of fine-scale (and turbulent) behaviour.  The results in this paper do not prevent singularities from occuring in the Navier-Stokes flow; but they can largely localise the impact of such singularities to a bounded region of space.} any of these conjectures unconditionally in this paper.  However, as the main result of this paper, we are able to establish the following implications:

\begin{theorem}[Implications]\label{main}\ 
\begin{itemize}
\item[(i)] Conjecture \ref{global-h1} and Conjecture \ref{global-h1-quant} are equivalent.
\item[(ii)] Conjecture \ref{global-h1} implies Conjecture \ref{global-periodic-normalised} (and hence also Conjecture \ref{global-periodic} and Conjecture \ref{global-periodic-homog}).
\item[(iii)] Conjecture \ref{global-h1} implies Conjecture \ref{global-h1-spatial}, which is equivalent to Conjecture \ref{global-schwartz-spatial}.
\item[(iv)] Conjecture \ref{global-h1-spatial} implies Conjecture \ref{global-enstrophy-homog} and Conjecture \ref{global-schwartz} (and hence also Conjectures \ref{global-schwartz-homog}).
\item[(v)] Conjecture \ref{global-enstrophy-homog} is equivalent to Conjecture \ref{global-energy-homog}.
\item[(vi)] Conjecture \ref{global-enstrophy-homog}, Conjecture \ref{global-h1-r3}, Conjecture \ref{global-h1-quant-r3}, and Conjecture \ref{global-h1-quant-global} are all equivalent to each other.
\end{itemize}
\end{theorem}

The logical relationship between these conjectures given by the above implications (as well as some trivial implications, and the equivalences in \cite{tao-quantitative}) are displayed in Figure \ref{imp}.

\begin{figure}[tb]
\includegraphics[scale=0.75]{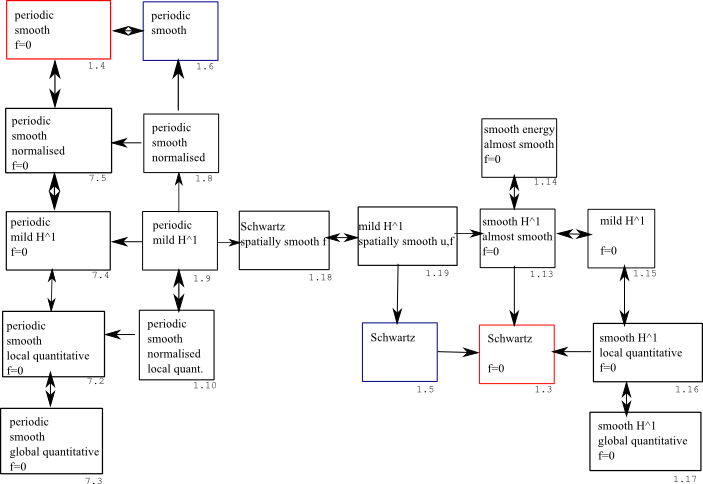}
\caption{Known implications between the various conjectures described here (i.e. existence of smooth or mild solutions, or local or global quantitative bounds in the periodic, Schwartz, $H^1$, or finite energy categories, with or without normalised pressure, and with or without the $f=0$ condition) and also in \cite{tao-quantitative} (the latter conjectures and implications occupy the far left column).  A positive solution to the red problems, or a negative solution to the blue problems, qualify for the Clay Millennium prize as stated in \cite{feff}.}
\label{imp}
\end{figure}

Among other things, these results essentially show that in order to solve the Navier-Stokes global regularity problem, it suffices to study the periodic setting (but with the caveat that one now has to consider forcing terms with the regularity of $L^\infty_t H^1_x$).

Theorem \ref{main}(i) is a variant of the compactness arguments used in \cite{tao-quantitative} (see also \cite{gallagher}, \cite{rusin}), and is proven in Section \ref{compact-sec}.  
Part (ii) of this theorem is a standard consequence of the periodic $H^1$ local well-posedness theory, which we review in Section \ref{local-h1}. In the homogeneous $f=0$ case it is possible to reverse this implication by the compactness arguments mentioned previously; see \cite{tao-quantitative}.  However, we were unable to obtain this converse implication in the inhomogeneous case.  Part (iv) is similarly a consequence of the non-periodic $H^1$ local well-posedness theory, and is also proven in Section \ref{local-h1}.

Part (vi) is also a variant of the results in \cite{tao-quantitative}, with the main new ingredient being a use of concentration compactness instead of compactness in order to deal with the unboundedness of the spatial domain $\R^3$, using the methods from \cite{bahouri}, \cite{gerard}, \cite{gallagher}.  We establish these results in Section \ref{quant-sec}.

The more novel aspects of this theorem are parts (iii) and (v), which we establish in Section \ref{enstrophy-sec} and Section \ref{energy-sol-sec} respectively.  These results rely primarily on a new localised enstrophy inequality (Theorem \ref{enstrophy-loc}) which can be viewed as a weak version of finite speed of propagation\footnote{Actually, in our setting, ``finite distance of propagation'' would be more accurate; we obtain an $L^1_t$ bound for the propagation velocity (see Proposition \ref{bounded-speed}) rather than an $L^\infty_t$ bound.} for the enstrophy $\frac{1}{2} \int_{\R^3} |\omega(t,x)|^2\ dx$, where $\omega := \nabla \times u$ is the vorticity.  We will also obtain a similar localised energy inequality for the energy $\frac{1}{2} \int_{\R^3} |u(t,x)|^2\ dx$, but it will be the enstrophy inequality that is of primary importance to us, as the enstrophy is a subcritical quantity and can be used to obtain regularity (and local control on enstrophy can similarly be used to obtain local regularity).  Remarkably, one is able to obtain local enstrophy inequalities even though the only \emph{a priori} controlled quantity, namely the energy, is supercritical; the main difficulty is a harmonic analysis one, namely to control nonlinear effects primarily in terms of the local enstrophy and only secondarily in terms of the energy.  

\begin{remark} As one can see from Figure \ref{imp}, the precise relationship between all the conjectures discussed here is rather complicated.  However, if one is willing to ignore the distinction between homogeneous and inhomogeneous data, as well as the (rather technical) distinction between smooth and almost smooth solutions, then the main implications can then be informally summarised as follows:
\begin{itemize}
\item (Homogenisation)  Without pressure normalisation, the inhomogeneity in the periodic global regularity conjecture is irrelevant: the inhomogeneous regularity conjecture is equivalent to the homogeneous one.
\item (Localisation) The global regularity problem in the Schwartz, $H^1$, and finite energy categories are ``essentially'' equivalent to each other;
\item (More localisation) The global regularity problem in any of the above three categories is ``essentially'' a consequence of the global regularity problem in the periodic category; and
\item (Concentration compactness) Quantitative and qualitative versions of the global regularity problem (in a variety of categories) are ``essentially'' equivalent to each other.
\end{itemize}
The qualifier ``essentially'' here though needs to be taken with a grain of salt; again, one should consult Figure \ref{imp} for an accurate depiction of the implications.
\end{remark}

The local enstrophy inequality has a number of other consequences, for instance allowing one to construct Leray-Hopf weak solutions whose (spatial) singularities are compactly supported in space; see Proposition \ref{partial}.

\begin{remark} Since the submission of this manuscript, the referee pointed out that the partial regularity theory of Cafarelli, Kohn, and Nirenberg \cite{ckn} also allows one to partially reverse the implication in Theorem \ref{main}(iii), and more specifically to deduce Conjecture \ref{global-periodic-normalised} from Conjecture \ref{global-h1-spatial}.  We sketch the referee's argument in Remark \ref{ref}.
\end{remark}

\subsection{Acknowledgements}

The author is supported by NSF Research Award CCF-0649473, the NSF Waterman Award and a grant from the MacArthur Foundation.  The author is also indebted to Andrea Bertozzi for valuable discussions, and in particular for raising the question of whether the answer to the global regularity problem of Navier-Stokes equation is sensitive to the decay hypotheses on the initial velocity field.  We thank Xiaoyi Zhang and the anonymous referee for corrections and suggestions.

\section{Notation and basic estimates}\label{notation-sec}

We use $X \lesssim Y$, $Y \gtrsim X$, or $X=O(Y)$ to denote the estimate $X \leq CY$ for an absolute constant $C$.  If we need $C$ to depend on a parameter, we shall indicate this by subscripts, thus for instance $X \lesssim_s Y$ denotes the estimate $X \leq C_s Y$ for some $C_s$ depending on $s$.  We use $X \sim Y$ as shorthand for $X \lesssim Y \lesssim X$.

We will occasionally use the Einstein summation conventions, using Roman indices $i,j$ to range over the three spatial dimensions $1,2,3$, though we will not bother to raise and lower these indices; for instance, the components of a vector field $u$ will be $u_i$.  We use $\partial_i$ to denote the derivative with respect to the $i^{\operatorname{th}}$ spatial coordinate $x_i$.  Unless otherwise specified, the Laplacian $\Delta = \partial_i \partial_i$ will denote the spatial Laplacian.  (In Lemma \ref{divloc}, though, we will briefly need to deal with the Laplace-Beltrami operator $\Delta_{S^2}$ on the sphere $S^2$.)  Similarly, $\nabla$ will refer to the spatial gradient $\nabla = \nabla_x$ unless otherwise stated.  We use the usual notations $\nabla f$, $\nabla \cdot u$, $\nabla \times u$, for the gradient, divergence, or curl of a scalar field $f$ or a vector field $u$.

It will be convenient (particularly when dealing with nonlinear error terms) to use \emph{schematic notation}, in which an expression such as $\bigO( u v w )$ involving some vector or tensor-valued quantities $u,v,w$ denotes some constant-coefficient combination of products of the components of $u, v, w$ respectively, and similarly for other expressions of this type.  Thus, for instance, $\nabla \times \nabla \times u$ could be written schematically as $\bigO( \nabla^2 u )$, $|u \times v|^2$ could be written schematically as $\bigO( u u v v )$, and so forth. 

For any centre $x_0 \in \R^3$ and radius $R>0$, we use $B(x_0,R) := \{ x \in \R^3: |x-x_0| \leq R \}$ to denote the (closed) Euclidean ball.  Much of our analysis will be localised to a ball $B(x_0,R)$, an annulus $B(x_0,R) \backslash B(x_0,r)$, or an exterior region $\R^3 \backslash B(x_0,R)$ (and often $x_0$ will be normalised to the origin $0$).

We define the absolute value of a tensor in the usual Euclidean sense.  Thus, for instance, if $u=u_i$ is a vector field, then $|u|^2 = u_i u_i$, $|\nabla u|^2 = (\partial_i u_j) (\partial_i u_j)$, $|\nabla^2 u|^2 = (\partial_i \partial_j u_k) (\partial_i \partial_j u_k)$, and so forth.

If $E$ is a set, we use $1_E$ to denote the associated indicator function, thus $1_E(x)=1$ when $x \in E$ and $1_E(x)=0$ otherwise.  We sometimes also use a statement in place of $E$; thus for instance $1_{k \neq 0}$ would equal $1$ if $k \neq 0$ and $0$ when $k=0$.

We use the usual Lebesgue spaces $L^p(\Omega)$ for various domains $\Omega$ (usually subsets of Euclidean space $\R^3$ or a torus $\R^3/L\Z^3$) and various exponents $1 \leq p \leq \infty$, which will always be equipped with an obvious Lebesgue measure.  We often write $L^p(\Omega)$ as $L^p_x(\Omega)$ to emphasise the spatial nature of the domain $\Omega$.
Given an absolutely integrable function $f \in L^1_x(\R^3)$, we define the Fourier transform $\hat f: \R^3 \to \C$ by the formula
$$ \hat f(\xi) := \int_{\R^3} e^{-2\pi i x \cdot \xi} f(x)\ dx;$$
we then extend this Fourier transform to tempered distributions in the usual manner.  For a function $f$ which is periodic with period $1$, and thus representable as a function on the torus $\R^3/\Z^3$, we define the discrete Fourier transform $\hat f: \Z^3 \to \C$ by the formula
$$ \hat f(k) := \int_{\R^3/\Z^3} e^{-2\pi i k \cdot x} f(x)\ dx$$
when $f$ is absolutely integrable on $\R^3/\Z^3$, and extend this to more general distributions on $\R^3/\Z^3$ in the usual fashion.
Strictly speaking, these two notations are not compatible with each other, but it will always be clear in context whether we are using the non-periodic or the periodic Fourier transform.

For any spatial domain $\Omega$ (contained in either $\R^3$ or $\R^3/L\Z^3$) and any natural number $k \geq 0$, we define the classical Sobolev norms $\|u\|_{H^k_x(\Omega)}$ of a smooth function $u: \Omega \to \R$ by the formula
$$ \|u\|_{H^k_x(\Omega)} := \left(\sum_{j=0}^k \| \nabla^j u \|_{L^2_x(\Omega)}^2\right)^{1/2},$$
and say that $u \in H^k_x(\Omega)$ when $\|u\|_{H^k_x(\Omega)}$ is finite.  Note that we do not impose any vanishing conditions at the boundary of $\Omega$, and to avoid technical issues we will not attempt to define these norms for non-smooth functions $u$ in the event that $\Omega$ has a non-trivial boundary.  In the domain $\R^3$ and for $s \in \R$, we define the Sobolev norm $\|u\|_{H^s_x(\R^3)}$ of a tempered distribution $u: \R^3 \to \R$ by the formula
$$ \|u\|_{H^s_x(\R^3)} := \left(\int_{\R^3} (1+|\xi|^2)^s |\hat u(\xi)|^2\ d\xi\right)^{1/2}.$$
Strictly speaking, this conflicts slightly with the previous notation when $k$ is a non-negative integer, but the two norms are equivalent up to constants (and both norms define a Hilbert space structure), so the distinction will not be relevant for our purposes.  For $s>-3/2$, we also define the homogeneous Sobolev norm
$$ \|u\|_{\dot H^s_x(\R^3)} := \left(\int_{\R^3} |\xi|^{2s} |\hat u(\xi)|^2\ d\xi\right)^{1/2},$$
and let $H^s_x(\R^3), \dot H^s_x(\R^3)$ be the space of tempered distributions with finite $H^s_x(\R^3)$ or $\dot H^s_x(\R^3)$ norm respectively.
Similarly, on the torus $\R^3/\Z^3$ and $s \in \R$, we define the Sobolev norm $\|u\|_{H^s_x(\R^3/\Z^3)}$ of a distribution $u: \R^3/\Z^3 \to \R$ by the formula
$$ \|u\|_{H^s_x(\R^3/\Z^3)} := \left(\sum_{k \in \Z^3} (1+|k|^2)^s |\hat u(k)|^2\right)^{1/2};$$
again, this conflicts slightly with the classical Sobolev norms $H^k_x(\R^3/\Z^3)$, but this will not be a serious issue in this paper.  We define $H^s_x(\R^3/\Z^3)$ to be the space of all distributions $u$ with finite $H^s_x(\R^3/\Z^3)$ norm, and $H^s_x(\R^3/\Z^3)_0$ to be the codimension one subspace
of functions or distributions $u$ which are mean zero in the sense that $\hat u(0)=0$.

In a similar vein, given a spatial domain $\Omega$ and a natural number $k \geq 0$, we define $C^k_x(\Omega)$ to be the space of all $k$ times continuously differentiable functions $u: \Omega \to \R$ whose norm
$$ \|u\|_{C^k_x(\Omega)} := \sum_{j=0}^k \|\nabla^j u\|_{L^\infty_x(\Omega)}$$
is finite\footnote{Note that if $\Omega$ is non-compact, then it is possible for a smooth function to fail to lie in $C^k(\Omega)$ if it becomes unbounded or excessively oscillatory at infinity.  One could use a notation such as $C^k_{x,\operatorname{loc}}(\Omega)$ to describe the space of functions that are $k$ times continuously differentiable with no bounds on derivatives, but we will not need such notation here.}.

Given any spatial norm $\| \|_{X_x(\Omega)}$ associated to a function space $X_x$ defined on a spatial domain $\Omega$, and a time interval $I$, we can define mixed-norms $\|u\|_{L^p_t X_x(I \times \Omega)}$ on functions $u: I \times \Omega \to \R$ by the formula
$$ \|u\|_{L^p_t X_x(I \times \Omega)} := (\int_I \|u(t)\|_{X_x(\Omega)}^p \ dt)^{1/p}$$
when $1 \leq p < \infty$, and
$$ \|u\|_{L^\infty_t X_x(I \times \Omega)} := \operatorname{ess} \sup_{t \in I} \|u(t)\|_{X_x(\Omega)},$$
assuming in both cases that $u(t)$ lies in $X(\Omega)$ for almost every $\Omega$, and then let $L^p_t X_x(I \times \Omega)$ be the space of functions (or, in some cases, distributions) whose $L^p_t X_x(I \times \Omega)$ is finite.  Thus, for instance, $L^\infty_t C^2_x(I \times \Omega)$ would be the space of functions $u: I \times \Omega \to \R$ such that for almost every $x \in I$, $u(t): \Omega \to \R$ is in $C^2_x(\Omega)$, and the norm
$$ \|u\|_{L^\infty_t C^2_x(I \times \Omega)} := \operatorname{ess} \sup_{t \in I} \|u(t)\|_{C^2_x(\Omega)}$$
is finite.

Similarly, for any natural number $k \geq 0$, we define $C^k_t X_x(I \times \Omega)$ to be the space of all functions $u: I \times \Omega \to \R$ such that the curve $t \mapsto u(t)$ from $I$ to $X_x(\Omega)$ is $k$ times continuously differentiable, and that the norm
$$ \| u \|_{C^k_t X_x(I \times \Omega)} := \sum_{j=0}^k \| \nabla^j u \|_{L^\infty_t X_x(I \times \Omega)}$$
is finite.

Given two normed function spaces $X, Y$ on the same domain (in either space or spacetime), we can endow their intersection $X \cap Y$ with the norm
$$ \| u \|_{X \cap Y} := \|u\|_X + \|u\|_Y.$$
For us, the most common example of such hybrid norms will be the spaces
\begin{equation}\label{x-def}
 X^s(I \times \Omega) := L^\infty_t H^s_x(I \times \Omega) \cap L^2_x H^{s+1}_x(I \times \Omega),
\end{equation}
defined whenever $I$ is a time interval, $s$ is a natural number, and $\Omega$ is a spatial domain, or whenever $I$ is a time interval, $s$ is real, and $\Omega$ is either $\R^3$ or $\R^3/\Z^3$.  The $X^s$ spaces (particularly $X^1$) will play a prominent role in the (subcritical) local well-posedness theory for the Navier-Stokes equations; see Section \ref{local-h1}.  The space $X^0$ will also be naturally associated with energy estimates, and the space $X^1$ with enstrophy estimates.

All of these above function spaces can of course be extended to functions that are vector or tensor-valued without difficulty (there are multiple ways to define the norms in these cases, but all such definitions will be equivalent up to constants).

We use the Fourier transform to define a number of useful multipliers on $\R^3$ or $\R^3/\Z^3$.  On $\R^3$, we formally define the inverse Laplacian operator $\Delta^{-1}$ by the formula
\begin{equation}\label{fax}
 \widehat{\Delta^{-1} f}(\xi) := \frac{-1}{4\pi^2 |\xi|^2} \hat f(\xi),
\end{equation}
which is well-defined for any tempered distribution $f: \R^3 \to \R$ for which the right-hand side of \eqref{fax} is locally integrable.  This is for instance the case if $f$ lies in $k^{\operatorname{th}}$ derivative of a function in $L^1_x(\R^3)$ for some $k \geq 0$, or the $k^{\operatorname{th}}$ derivative of a function in $L^2_x(\R^3)$ for some $k \geq 1$.  If $f \in L^1_x(\R^3)$, then as is well known one has the Newton potential representation
\begin{equation}\label{newton}
 \Delta^{-1} f(x) = \frac{-1}{4\pi} \int_{\R^3} \frac{f(y)}{|x-y|}\ dy.
\end{equation}
Note in particular that \eqref{newton} implies that if $f \in L^1_x(\R^3)$ is supported on some closed set $K$, then $\Delta^{-1} f$ will be smooth away from $K$.  Also observe from Fourier analysis (and decomposition into local and global components) that if $f$ is smooth and is either the $k^{\operatorname{th}}$ derivative of a function in $L^1_x(\R^3)$ for some $k \geq 0$, or the $k^{\operatorname{th}}$ derivative of a function in $L^2_x(\R^3)$ for some $k \geq 1$, then $\Delta^{-1} f$ will be smooth also.

We also note that the Newton potential $-\frac{1}{4\pi|x-y|}$ is smooth away from the diagonal $x=y$.  Because of this, we will often be able to obtain large amounts of regularity in space in the ``far field'' region when $|x|$ is large, for fields such as the velocity field $u$.  However, it will often be significantly more challenging to gain significant amounts of regularity in \emph{time}, because the inverse Laplacian $\Delta^{-1}$ has no smoothing properties in the time variable. 

On $\R^3/\Z^3$, we similarly define the inverse Laplacian operator $\Delta^{-1}$ for distributions $f: \R^3/\Z^3 \to \R$ with $\hat f(0)=0$ by the formula
\begin{equation}\label{fax-k}
 \widehat{\Delta^{-1} f}(k) := \frac{-1_{k \neq 0}}{4\pi^2 |k|^2} \hat f(k).
\end{equation}

We define the \emph{Leray projection} $Pu$ of a (tempered distributional) vector field $u: \R^3 \to \R^3$ by the formula
$$ Pu := \Delta^{-1}( \nabla \times \nabla \times u).$$
If $u$ is square-integrable, then $Pu$ is the orthogonal projection of $u$ onto the space of square-integrable divergence-free vector fields; from Calder\'on-Zygmund theory we know that the projection $P$ is bounded on $L^p_x(\R^3)$ for every $1 < p < \infty$, and from Fourier analysis we see that $P$ is also $H^s_x(\R^3)$ for every $s \in \R$.  Note that if $u$ is square-integrable and divergence-free, then $Pu=u$, and we thus have the \emph{Biot-Savart law}
\begin{equation}\label{biot}
u = \Delta^{-1}( \nabla \times \omega )
\end{equation}
where $\omega := \nabla \times u$.

In either $\R^3$ or $\R^3/L\Z^3$, we let $e^{t\Delta}$ for $t>0$ be the usual heat semigroup associated to the heat equation $u_t = \Delta u$.  On $\R^3$, this takes the explicit form
$$ e^{t\Delta} f(x) = \frac{1}{(4\pi t)^{3/2}} \int_{\R^3} e^{-|x-y|^2/4t} f(y)\ dy$$
for $f \in L^p_x(\R^3)$ for some $1 \leq p \leq \infty$. From Young's inequality we thus record the dispersive inequality
\begin{equation}\label{dispersive}
\| e^{t\Delta} f\|_{L^q(\R^3)} \lesssim t^{\frac{3}{2q} - \frac{3}{2p}} \|f\|_{L^p(\R^3)}
\end{equation}
whenever $1 \leq p \leq q \leq \infty$ and $t>0$.

We recall \emph{Duhamel's formula}
\begin{equation}\label{duhamel}
u(t) = e^{(t-t_0)\Delta} u(t_0) + \int_{t_0}^t e^{(t-t')\Delta} (\partial_t u - \Delta u)(t')\ dt'
\end{equation}
whenever $u: [t_0,t] \times \Omega \to \R$ is a smooth tempered distribution, with $\Omega$ equal to either $\R^3$ or $\R^3/\Z^3$.

We record some linear and bilinear estimates involving Duhamel-type integrals and the spaces $X^s$ defined in \eqref{x-def}, which are useful in the local $H^1$ theory for the Navier-Stokes equation:

\begin{lemma}[Linear and bilinear estimates]  Let $[t_0,t_1]$ be a time interval, let $\Omega$ be either $\R^3$ or $\R^3/\Z^3$, and suppose that $u: [t_0,t_1] \times \Omega \to \R$ and $F: [t_0,t_1] \times \Omega \to \R$ are tempered distributions such that
\begin{equation}\label{Duhamel-formula}
u(t) = e^{(t-t_0)\Delta} u(t_0) + \int_{t_0}^t e^{(t-t')\Delta} F(t')\ dt'.
\end{equation}
Then we have the standard \emph{energy estimate}\footnote{We adopt the convention that an estimate is vacuously true if the right-hand side is infinite or undefined.}
\begin{equation}\label{energy-duh}
\| u \|_{X^s([t_0,t_1] \times \Omega)} \lesssim_s \|u(t_0)\|_{H^s_x(\Omega)} + \| F \|_{L^1_t H^s_x([t_0,t_1] \times \Omega)}
\end{equation}
for any $s \geq 0$, as well as the variant
\begin{equation}\label{energy-duh2}
\| u \|_{X^{s}([t_0,t_1] \times \Omega)} \lesssim_s \|u(t_0)\|_{H^{s}_x(\Omega)} + \| F \|_{L^2_t H^{s-1}_x([t_0,t_1] \times \Omega)}
\end{equation}
for any $s \geq 1$.  We also note the further variant
\begin{equation}\label{energy-duh3}
\| u \|_{X^{s}([t_0,t_1] \times \Omega)} \lesssim_s \|u(t_0)\|_{H^{s}_x(\Omega)} + \| F \|_{L^4_t L^2_x([t_0,t_1] \times \Omega)}
\end{equation}
for any $s<3/2$.

We also have the bilinear estimate
\begin{equation}\label{bilinear}
 \| \nabla(u v) \|_{L^4_t L^2_x([t_0,t_1] \times \Omega)} \lesssim \|u\|_{X^1([t_0,t_1] \times \Omega)} \|v\|_{X^1([t_0,t_1] \times \Omega)}
\end{equation}
for any $u,v: [t_0,t_1] \times \R^3 \to \R$,
which in particular implies (by a H\"older in time) that
\begin{equation}\label{bilinear-2}
 \| \nabla(u v) \|_{L^2_t L^2_x([t_0,t_1] \times \R^3)} \lesssim (t_1-t_0)^{1/4} \|u\|_{X^1([t_0,t_1] \times \R^3)} \|v\|_{X^1([t_0,t_1] \times \R^3)}.
\end{equation}
\end{lemma}

\begin{proof}  The estimates\footnote{Strictly speaking, the result in \cite{tao-quantitative} was stated for the torus rather than $\R^3$, but the argument works without modification in either domain, after first truncating $u(t_0), F$ to be Schwartz to avoid technicalities at infinity, and usig a standard density argument.} \eqref{energy-duh2}, \eqref{energy-duh3}, \eqref{bilinear} are established in \cite[Lemma 2.1, Proposition 2.2]{tao-quantitative}.  The estimate \eqref{energy-duh} follows from the $F=0$ case of \eqref{energy-duh} and Minkowski's inequality.
\end{proof}

Finally, we define the Littlewood-Paley projection operators on $\R^3$.  Let $\varphi(\xi)$ be a fixed bump function supported in the ball $\{ \xi \in \R^3: |\xi| \leq 2 \}$ and equal to $1$ on the ball $\{ \xi \in \R^3: |\xi| \leq 1 \}$.  Define a \emph{dyadic number} to be a number $N$ of the form $N=2^k$ for some integer $k$.  For each dyadic number $N$, we define the Fourier multipliers
\begin{align*}
\widehat{P_{\leq N} f}(\xi) &:= \varphi(\xi/N) \hat f(\xi)\\
\widehat{P_{> N} f}(\xi) &:= (1 - \varphi(\xi/N)) \hat f(\xi)\\
\widehat{P_N f}(\xi) &:= \psi(\xi/N)\hat f(\xi) := (\varphi(\xi/N) - \varphi(2\xi/N)) \hat f(\xi),
\end{align*}
We similarly define $P_{<N}$ and $P_{\geq N}$.  Thus for any tempered distribution we have $f = \sum_N P_N f$ in a weakly convergent sense at least, where the sum ranges over dyadic numbers.  We recall the usual \emph{Bernstein estimates}
\begin{equation}\label{bernstein}
\begin{split}
\| D^s P_N f\|_{L^p_x(\R^3)} &\lesssim_{p,s,D^s} N^s \| P_N f \|_{L^p_x(\R^3)},\\
\| \nabla^k P_N f\|_{L^p_x(\R^3)} &\sim_{k,s} N^k \| P_N f \|_{L^p_x(\R^3)},\\
\|P_{\leq N} f\|_{L^q_x(\R^3)} &\lesssim_{p,q} N^{\frac{3}{p}-\frac{3}{q}} \|P_{\leq N} f\|_{L^p_x(\R^3)},\\
\|P_N f\|_{L^q_x(\R^3)} &\lesssim_{p,q} N^{\frac{3}{p}-\frac{3}{q}} \| P_N f\|_{L^p_x(\R^3)}
\end{split}
\end{equation}
for all $1 \leq p \leq q \leq \infty$, $s \in \R$, $k \geq 0$, and pseudo-differential operators $D^s$ of order $s$; see e.g. \cite[Appendix A]{tao-dispersive}.

We recall the \emph{Littlewood-Paley trichotomy}: an expression of the form $P_N( (P_{N_1} f_1) (P_{N_2} f_2))$ vanishes unless one of the following three scenarios holds:
\begin{itemize}
\item (Low-high interaction) $N_2 \lesssim N_1 \sim N$.
\item (High-low interaction) $N_1 \lesssim N_2 \sim N$.
\item (High-high interaction) $N \lesssim N_1 \sim N_2$.
\end{itemize}
This trichotomy is useful for obtaining estimates on bilinear expressions, as we shall see in Section \ref{speed-sec}.

We have the following frequency-localised variant of \eqref{dispersive}:

\begin{lemma} If $N$ is a dyadic number and $f: \R^3 \to \R$ has Fourier transform supported on an annulus $\{ \xi: |\xi| \sim N \}$, then we have
\begin{equation}\label{energy-decay}
\| e^{t\Delta} f\|_{L^q(\R^3)} \lesssim t^{\frac{3}{2q} - \frac{3}{2p}} \exp(-ct N^2) \|f\|_{L^p(\R^3)}
\end{equation}
for some absolute constant $c>0$ and all $1 \leq p \leq q \leq \infty$.
\end{lemma}

\begin{proof}  By Littlewood-Paley projection it suffices to show that
$$ \| e^{t\Delta} P_N f\|_{L^q(\R^3)} \lesssim t^{\frac{3}{2q} - \frac{3}{2p}} \exp(-ct N^2) \|f\|_{L^p(\R^3)}$$
for all test functions $f$.  By rescaling we may set $t=1$; in view of \eqref{dispersive} we may then set $N \geq 1$.  One then verifies from Fourier analysis that $e^{t\Delta} P_N$ is a convolution operator whose kernel has an $L^\infty_x(\R^3)$ and $L^1_x(\R^3)$ norm that are both $O( \exp(-cN^2) )$ for some absolute constant $c>0$, and the claim follows from Young's inequality.
\end{proof}

From the uniform smoothness of the heat kernel we also observe the estimate
\begin{equation}\label{kop}
\| e^{t\Delta} f \|_{C^k_x(K)} \lesssim_{k,K,T,p} \exp(-c_T r^2) \|f\|_{L^p_x(\R^3)}
\end{equation}
whenever $0 \leq t \leq T$, $1 \leq p \leq \infty$, $k \geq 0$, $K$ is a compact subset of $\R^3$, $r \geq 1$, and $f$ is supported on the set $\{ x \in \R^3: \operatorname{dist}(x,K) \geq r \}$, and some quantity $c_T>0$ depending only on $T$.  In practice, this estimate will be an effective substitute for finite speed of propagation for the heat equation.

\section{Symmetries of the equation}

In this section we review some well-known symmetries of the Navier-Stokes flow, which transform a given smooth solution $(u,p,u_0,f,T)$ to another smooth solution $(\tilde u,\tilde p,\tilde u_0,\tilde f, \tilde T)$, as these symmetries will be useful at various points in the paper.

The simplest symmetry is the \emph{spatial translation symmetry}
\begin{equation}\label{space-translate}
\begin{split}
\tilde u(t,x) &:= u(t, x-x_0) \\
\tilde p(t,x) &:= p(t, x-x_0) \\
\tilde u_0(x) &:= u_0(x-x_0)\\
\tilde f(t,x) &:= f(t, x-x_0)\\
\tilde T &:= T,
\end{split}
\end{equation}
valid for any $x_0 \in \R^3$; this transformation clearly maps mild, smooth, or almost smooth solutions to solutions of the same type, and also preserves conditions such as finite energy, $H^1$, periodicity, pressure normalisation, or the Schwartz property.  In a similar vein, we have the \emph{time translation symmetry}
\begin{equation}\label{time-translate}
\begin{split}
\tilde u(t,x) &:= u(t+t_0, x) \\
\tilde p(t,x) &:= p(t+t_0, x) \\
\tilde u_0(x) &:= u(t_0,x)\\
\tilde f(t,x) &:= f(t+t_0, x)\\
\tilde T &:= T-t_0,
\end{split}
\end{equation}
valid for any $t_0 \in [0,T]$.  Again, this maps mild, smooth, or almost smooth solutions to solutions of the same type (and if $t_0>0$, then almost smooth solutions are even upgraded to smooth solutions). If the original solution is finite energy or $H^1$, then the transformed solution will be finite energy or $H^1$ also.  Note however that if it is only the original \emph{data} that is assumed to be finite energy or $H^1$, as opposed to the \emph{solution}, it is not immediately obvious that the time-translated solution remains finite energy or $H^1$, especially in view of the fact that the $H^1$ norm (or the enstrophy) is not a conserved quantity of the Navier-Stokes flow.  (See however Lemma \ref{energy-est} and Corollary \ref{ens-bound} below.)  The situation is particularly dramatic in the case of Schwartz data; as remarked earlier, time translation can instantly convert\footnote{This can be seen for instance by noting that moments such as $\int_{\R^3} \omega_1(t,x) (x_2^2-x_3^2)\ dx$ are not conserved in time, but must equal zero whenever $u(t)$ is Schwartz.}  Schwartz data to non-Schwartz data, due to the slow decay of the Newton potential appearing in \eqref{pressure-point} (or of its derivatives, such as the Biot-Savart kernel in \eqref{biot}).

Next, we record the \emph{scaling symmetry}
\begin{equation}\label{scaling}
\begin{split}
\tilde u(t,x) &:= \frac{1}{\lambda} u(\frac{t}{\lambda^2}, \frac{x}{\lambda}) \\
\tilde p(t,x) &:= \frac{1}{\lambda^2} p(\frac{t}{\lambda^2}, \frac{x}{\lambda}) \\
\tilde u_0(x) &:= \frac{1}{\lambda} u(\frac{x}{\lambda})\\
\tilde f(t,x) &:= \frac{1}{\lambda^3} f(\frac{t}{\lambda^2}, \frac{x}{\lambda})\\
\tilde T &:= T \lambda^2,
\end{split}
\end{equation}
valid for any $\lambda>0$; it also maps mild, smooth, or almost smooth solutions to solutions of the same type, and preserves properties such as finite energy, finite enstrophy, pressure normalisation, periodicity, or the Schwartz property, though note in the case of periodicity that a solution of period $L$ will map to a solution of period $\lambda L$.   We will only use scaling symmetry occasionally in this paper, mainly because most of the quantities we will be manipulating will be supercritical with respect to this symmetry.  Nevertheless, this scaling symmetry serves a fundamentally important conceptual purpose, by making the key distinction between subcritical, critical (or dimensionless), and supercritical quantities, which can help illuminate many of the results in this paper (and was also crucial in allowing the author to discover\footnote{The author also found dimensional analysis to be invaluable in checking the calculations for errors.  One \emph{could}, if one wished, exploit the scaling symmetry to normalise a key parameter (e.g. the energy $E$, or a radius parameter $r$) to equal one, which would simplify the numerology slightly, but then one would lose the use of dimensional analysis to check for errors, and so we have elected to largely avoid the use of scaling normalisations in this paper.} these results in the first place).

We record three further symmetries that impact upon the issue of pressure normalisation.  The first is the \emph{pressure shifting symmetry}
\begin{equation}\label{const}
\begin{split}
\tilde u(t,x) &:= u(t, x) \\
\tilde p(t,x) &:= p(t, x) + C(t) \\
\tilde u_0(x) &:= u_0(x)\\
\tilde f(t,x) &:= f(t, x)\\
\tilde T &:= T,
\end{split}
\end{equation}
valid for any smooth function $C: \R \to \R$.  This clearly maps smooth, or almost smooth solutions to solutions of the same type, and preserves properties ssuch as finite energy, $H^1$, periodicity, and the Schwartz property; however, it destroys pressure normalisation (and thus the notion of a mild solution).  A slightly more sophisticated symmetry in the same spirit is the \emph{Galilean symmetry}
\begin{equation}\label{galilean}
\begin{split}
\tilde u(t,x) &:= u(t, x - \int_0^t v(s)\ ds) + v(t)\\
\tilde p(t,x) &:= p(t, x - \int_0^t v(s)\ ds) - x \cdot v'(t) \\
\tilde u_0(x) &:= u_0(x) + v(0)\\
\tilde f(t,x) &:= f(t, x - \int_0^t v(s)\ ds)\\
\tilde T &:= T,
\end{split}
\end{equation}
valid for any smooth function $v: \R \to \R^3$.  One can carefully check that this symmetry indeed maps mild, smooth solutions to smooth solutions, and preserves periodicity (recall here that in our definition of a periodic solution, the pressure was \emph{not} required to be periodic).  On the other hand, this symmetry does not preserve finite energy, $H^1$, or the Schwartz property.  It also clearly destroys the pressure normalisation property.

Finally, we observe that one can absorb divergences into the forcing term via the forcing symmetry
\begin{equation}\label{forcing}
\begin{split}
\tilde u(t,x) &:= u(t, x)\\
\tilde p(t,x) &:= p(t, x) + q(t,x) \\
\tilde u_0(x) &:= u_0(x)\\
\tilde f(t,x) &:= f(t, x) + \nabla \cdot q(t,x),\\
\tilde T &:= T,
\end{split}
\end{equation}
valid for any smooth function $P: [0,T] \times \R^3 \to \R^3$.  If the new forcing term $\tilde f$ still has finite energy or is still periodic, then the normalisation of pressure is preserved.  In the periodic setting, we will apply \eqref{forcing} with a linear term $q(t,x) := x \cdot a(t)$, allowing one to alter $f$ by an arbitrary constant $a(t)$.  In the finite energy or $H^1$ setting, one can use \eqref{forcing} and the Leray projection $P$ to reduce to the divergence-free case $\nabla \cdot f = 0$; note though that this projection can destroy the Schwartz nature of $f$.  This divergence-free reduction is particularly useful in the case of normalised pressure, since \eqref{pressure-point} then simplifies to
\begin{equation}\label{pressure-point-simplified}
p = -\Delta^{-1} \partial_i \partial_j (u_i u_j).
\end{equation}

One can of course compose these symmetries together to obtain a larger (semi)group of symmetries.  For instance, by combining \eqref{galilean} and \eqref{forcing} we observe the symmetry
\begin{equation}\label{gal-2}
\begin{split}
\tilde u(t,x) &:= u(t, x - \int_0^t v(s)\ ds) + v(t)\\
\tilde p(t,x) &:= p(t, x - \int_0^t v(s)\ ds) \\
\tilde u_0(x) &:= u_0(x) + v(0)\\
\tilde f(t,x) &:= f(t, x - \int_0^t v(s)\ ds) + v'(t)\\
\tilde T &:= T,
\end{split}
\end{equation}
any smooth function $v: \R \to \R^3$.  This symmetry is particularly useful for periodic solutions; note that it preserves both the periodicity property and the normalised pressure property.  By choosing $v(t)$ appropriately, we see that we can use this symmetry to normalise periodic data $(u_0,f,T,L)$ to be \emph{mean zero} in the sense that
\begin{equation}\label{meanzero-1}
\int_{\R^3/L\Z^3} u_0(x)\ dx = 0
\end{equation}
and
\begin{equation}\label{meanzero-2}
\int_{\R^3/L\Z^3} f(t,x)\ dx = 0
\end{equation}
for all $0 \leq t \leq T$.  By integrating \eqref{ns} over the torus $\R^3/L\Z^3$, we then conclude with this normalisation that $u$ remains mean zero for all times $0 \leq t \leq T$:
\begin{equation}\label{meanzero-3}
\int_{\R^3/L\Z^3} u(t,x)\ dx = 0
\end{equation}
The same conclusion also holds for periodic $H^1$ mild solutions.

\section{Pressure normalisation}

The symmetries in \eqref{const}, \eqref{forcing} can alter the velocity field $u$ and pressure $p$ without affecting the data $(u_0,f,T)$, thus leading to a breakdown of uniqueness for the Navier-Stokes equation.  In this section we investigate this loss of uniqueness, and show that (in the smooth category, at least) one can ``quotient out'' these symmetries by reducing to the situation \eqref{pressure-point} of normalised pressure, at which point uniqueness can be recovered (at least in the $H^1$ category).

More precisely, we show

\begin{lemma}[Reduction to normalised pressure]\label{reduction}\ 
\begin{itemize}
\item[(i)] If $(u,p,u_0,f,T)$ is an almost smooth finite energy solution, then for almost every time $t \in [0,T]$ one has
\begin{equation}\label{pressure-point-modified}
p(t,x) = -\Delta^{-1} \partial_i \partial_j (u_i u_j)(t,x) + \Delta^{-1} \nabla \cdot f(t,x) + C(t).
\end{equation}
for some bounded measurable function $C: [0,T] \to \R$.
\item[(ii)] If $(u,p,u_0,f,T)$ is a periodic smooth solution, then there exists smooth functions $C: [0,T] \to \R$ and $a: [0,T] \to \R^3$ such that
\begin{equation}\label{pressure-point-period}
p(t,x) = -\Delta^{-1} \partial_i \partial_j (u_i u_j)(t,x) + \Delta^{-1} \nabla \cdot f(t,x) + x \cdot a(t) + C(t).
\end{equation}
In particular, after applying a Galilean transform \eqref{galilean} followed by a pressure-shifting transformation \eqref{const}, one can transform $(u,p,u_0,f,T)$ into a periodic smooth solution with normalised pressure.
\end{itemize}
\end{lemma}

\begin{remark} Morally, in (i) the function $C$ should be smooth (at least for times $t>0$), which would then imply that one can apply a pressure-shifting transformation \eqref{const} to convert $(u,p,u_0,f,T)$ into a smooth solution with normalised pressure.  However, there is the technical difficulty that in our definition of a finite energy smooth solution, we do not \emph{a priori} have any control time derivatives of $u$ in any $L^p_x(\R^3)$ norms, and as such we do not have time regularity on the component $\Delta^{-1} \partial_i \partial_j(u_i u_j)$ of \eqref{pressure-point-modified}.  In practice, though, this possible irregularity of $C(t)$ will not bother us, as we only need to understand the gradient $\nabla p$ of the pressure, rather than the pressure itself, in order to solve the Navier-Stokes equations \eqref{ns}.
\end{remark}

\begin{proof}
We begin with the periodic case, which is particularly easy due to Liouville's theorem (which, among other things, implies that the only harmonic periodic functions are the constants).  We may normalise the period $L$ to equal $1$.
Fix an almost smooth periodic solution $(u,p,u_0,f,T)$. Define the normalised pressure $p_0: [0,T] \times \R^3 \to \R$ by the formula
\begin{equation}\label{p0-def}
 p_0 := -\Delta^{-1} \partial_i \partial_j (u_i u_j) + \Delta^{-1} \nabla \cdot f.
\end{equation}
As $u$, $f$ are smooth and periodic, $p_0$ is smooth also, and from \eqref{deltap} one has $\Delta p = \Delta p_0$.  Thus one has
$$ p = p_0 + h$$
where $h: [0,T] \times \R^3 \to \R$ is a smooth function with $h(t)$ harmonic in space for each time $t$.  The function $h$ need not be periodic; however, from \eqref{ns} we have
$$ \partial_t u + (u \cdot \nabla) u = \Delta u - \nabla p_0 - \nabla h + f.$$
Every term aside from $\nabla h$ is periodic, and so $\nabla h$ is periodic also.  Since $\nabla h$ is also harmonic, it must therefore be constant in space by Liouville's theorem.  We therefore may write
$$ h(t,x) = x \cdot a(t) + C(t)$$
for some $a(t) \in \R^3$ and $C(t) \in \R$; since $h$ is smooth, $a, C$ are smooth also, and the claim follows.

Now we turn to the finite energy case, thus $(u,p,u_0,f,T)$ is now an almost smooth finite energy solution. By the time translation symmetry \eqref{time-translate} with an arbitrarily small time shift parameter $t_0$, we may assume without loss of generality that $(u,p,u_0,f,T)$ is smooth (and not just almost smooth).  We define the normalised pressure $p_0$ by \eqref{p0-def} as before, then for each time $t \in [0,T]$ one sees from \eqref{deltap} that
$$ p(t) = p_0(t) + h(t)$$
for some harmonic function $h(t): \R^3 \to \R$.  As $u, f$ are smooth and finite energy, one sees from \eqref{p0-def} that $p_0$ is bounded on compact subsets of spacetime; since $p$ is smooth, we conclude that $h$ is bounded on compact subsets of spacetime also.  From harmonicity, this implies that all spatial derivatives $\nabla^k h$ are also bounded on compact subsets of space time.  However, as noted previously, we cannot impose any time regularity on $p_0$ or $h$ because we do not have decay estimates on time derivatives of $u$.

It is easy to see that $h$ is measurable.  To obtain the lemma, it suffices to show that $h(t)$ is a constant function of $x$ for almost every time $t$.

Let $[t_1,t_2]$ be any interval in $[0,T]$.  Integrating \eqref{ns} in time on this interval, we see that
$$ u(t_2,x)-u(t_1,x) + \int_{t_1}^{t_2} (u \cdot \nabla) u(t,x)\ dt = \int_{t_1}^{t_2} \Delta u(t,x) - \nabla p(t,x) + f(t,x)\ dt.$$
Next, let $\chi: \R^3 \to \R$ be a smooth compactly supported spherically symmetric function of total mass $1$.  We integrate the above formula against $\frac{1}{R^3} \chi(\frac{x}{R})$ for some large parameter $R$, and conclude after some integration by parts (which is justified by the compact support of $\chi$ and the smooth (and hence $C^1$) nature of all functions involved) that
\begin{align*}
R^{-3} \int_{\R^3} u(t_2,x) \chi(\frac{x}{R})\ dx - R^{-3} \int_{\R^3} u(t_1,x) \chi(\frac{x}{R})\ dx \quad & \\
- R^{-4} \int_{t_1}^{t_2} \int_{\R^3} u(t,x) (u(t,x) \cdot \nabla \chi)(\frac{x}{R})\ dx dt &
= R^{-5} \int_{t_1}^{t_2} \int_{\R^3} u(t,x) (\Delta \chi)(\frac{x}{R})\ dx dt \\
&\quad + R^{-3} \int_{t_1}^{t_2} \int_{\R^3} \nabla p(t,x) \chi(\frac{x}{R})\ dx dt \\
&\quad + R^{-3} \int_{t_1}^{t_2} \int_{\R^3} f(t,x) \chi(\frac{x}{R})\ dx dt.
\end{align*}
From the finite energy hypothesis and the Cauchy-Schwarz inequality, one easily verifies that
\begin{align*}
\lim_{R \to \infty} R^{-3} \int_{\R^3} u(t_i,x) \chi(\frac{x}{R})\ dx &= 0 \\
\lim_{R \to \infty} R^{-4} \int_{t_1}^{t_2} \int_{\R^3} u(t,x) (u(t,x) \cdot \nabla \chi)(\frac{x}{R})\ dx dt &= 0 \\
\lim_{R \to \infty} R^{-5} \int_{t_1}^{t_2} \int_{\R^3} u(t,x) (\Delta \chi)(\frac{x}{R})\ dx dt &= 0 \\
\lim_{R \to \infty} R^{-3} \int_{t_1}^{t_2} \int_{\R^3} f(t,x) \chi(\frac{x}{R})\ dx dt &= 0
\end{align*}
and thus
\begin{equation}\label{nap}
 \lim_{R \to \infty} R^{-3} \int_{t_1}^{t_2} \int_{\R^3} \nabla p(t,x) \chi(\frac{x}{R})\ dx dt = 0.
\end{equation}
Next, by an integration by parts and \eqref{p0-def}, we can express
$$ R^{-3} \int_{t_1}^{t_2} \int_{\R^3} \nabla p_0(t,x) \chi(\frac{x}{R})\ dx dt $$
as
\begin{align*}
&R^{-4} \int_{t_1}^{t_2} \int_{\R^3} u_i u_j(t,x) (\nabla \Delta^{-1} \partial_i \partial_j \chi)(\frac{x}{R})\ dx dt \\
&\quad + R^{-3} \int_{t_1}^{t_2} \int_{\R^3} f_i(t,x) (\nabla \Delta^{-1} \partial_i \chi)(\frac{x}{R})\ dx dt.
\end{align*}
From the finite energy nature of $(u,p,u_0,f,T)$ we see that this expression goes to zero as $R \to \infty$.  Subtracting this from \eqref{nap}, we conclude that
\begin{equation}\label{rmean}
 \lim_{R \to \infty} R^{-3} \int_{t_1}^{t_2} \int_{\R^3} \nabla h(t,x) \chi(\frac{x}{R})\ dx dt = 0.
\end{equation}
The function $x \mapsto \int_{t_1}^{t_2} \nabla h(t,x)$ is weakly harmonic, hence harmonic.  By the mean-value property of harmonic functions (and our choice of $\chi$) we thus have
$$ R^{-3} \int_{t_1}^{t_2} \int_{\R^3} \nabla h(t,x) \chi(\frac{x}{R})\ dx dt = \int_{t_1}^{t_2} \nabla h(t,0)\ dt$$
and thus
$$ \int_{t_1}^{t_2} \nabla h(t,0)\ dt = 0.$$
Since $t_1,t_2$ were arbitrary, we conclude from the Lebesgue differentiation theorem
that $\nabla h(t,0)=0$ for almost every $t \in [0,T]$.
Using spatial translation invariance \eqref{space-translate} to replace the spatial origin by an element of a countable dense subset of $\R^3$, and using the fact that harmonic functions are continuous, we conclude that $\nabla h(t)$ is identically zero for almost every $t \in [0,T]$, and so $h(t)$ is constant for almost every $t$ as desired.
\end{proof}

We note a useful corollary of Lemma \ref{reduction}(i):

\begin{corollary}[Almost smooth $H^1$ solutions are essentially mild]\label{as-mild}  Let $(u,p,u_0,f,T)$ be an almost smooth $H^1$ solution.  Then $(u,\tilde p,u_0,f,T)$ is a mild $H^1$ solution, where 
$$ \tilde p(t,x) := -\Delta^{-1} \partial_i \partial_j (u_i u_j)(t,x) + \Delta^{-1} \nabla \cdot f(t,x).$$
Furthermore, for almost every $t \in [0,T]$, $p(t)$ and $\tilde p(t)$ differ by a constant (and thus $\nabla p = \nabla \tilde p$).
\end{corollary}

\begin{proof}  By Lemma \ref{reduction}(i), $\nabla p$ is equal to $\nabla \tilde p$ almost everywhere; in particular, $\nabla p = \nabla \tilde p$ is a smooth tempered distribution.  The claim then follows from \eqref{ns} and the Duhamel formula \eqref{duhamel}.
\end{proof}

\section{Local well-posedness theory in $H^1$}\label{local-h1}

In this section we review the (subcritical) local well-posedness theory for both periodic and non-periodic $H^1$ mild solutions.  The material here is largely standard (and in most cases has been superceded by the more powerful critical well-posedness theory); for instance the uniqueness theory already follows from the work of Prodi \cite{prodi} and Serrin \cite{serrin}, the blowup criterion already is present in Leray \cite{leray}, the local existence theory follows from the work of Kato \cite{kato}, regularity of mild solutions follows from the work of Ladyzhenskaya \cite{lady}, the stability results given here follow from the stronger stability results of Chemin and Gallagher \cite{chemin}, and the compactness results were already essentially present in the author's previous paper \cite{tao-quantitative}.  However, for the convenience of the reader (and because we want to use the $X^s$ function spaces defined in \eqref{x-def} as the basis for the theory) we shall present all this theory in a self-contained manner.  There are now a number of advanced local well-posedness results at critical regularity, most notably that of Koch and Tataru \cite{koch}, but we will not need such powerful results here.

We begin with the periodic theory.  By taking advantage of the scaling symmetry \eqref{scaling} we may set the period $L$ equal to $1$.  Using the symmetry \eqref{gal-2} we may also restrict attention to data obeying the mean zero conditions \eqref{meanzero-1}, \eqref{meanzero-2}, thus $u_0 \in H^1_x(\R^3/\Z^3)_0$ and $f \in L^\infty_t H^1_x([0,T] \times \R^3/\Z^3)_0$.

\begin{theorem}[Local well-posedness in periodic $H^1$]\label{lwp-h1}  Let $(u_0,f,T,1)$ be periodic $H^1$ data obeying the mean zero conditions \eqref{meanzero-1}, \eqref{meanzero-2}.
\begin{itemize}
\item[(i)] (Strong solution) If $(u,p,u_0,f,T,1)$ is a periodic $H^1$ mild solution, then
$$ u \in C^0_t H^1_x([0,T] \times \R^3/\Z^3).$$
In particular, one can unambiguously define $u(t)$ in $H^1_x(\R^3/\Z^3)$ for each $t \in [0,T]$.
\item[(ii)] (Local existence)  If
\begin{equation}\label{d4}
 (\|u_0\|_{H^1_x(\R^3/\Z^3)} + \| f \|_{L^1_t H^1_x(\R^3/\Z^3)})^4 T \leq c
\end{equation}
for a sufficiently small absolute constant $c>0$, then there exists a periodic $H^1$ mild solution $(u,p,u_0,f,T,1)$ with the indicated data with
$$ \| u \|_{X^1([0,T] \times \R^3/\Z^3)} \lesssim \|u_0\|_{H^1_x(\R^3/\Z^3)} + \| f \|_{L^1_t H^1_x(\R^3/\Z^3)}$$
and more generally
$$ \| u \|_{X^k([0,T] \times \R^3/\Z^3)} \lesssim_{k,T,\|u_0\|_{H^k_x(\R^3/\Z^3)},\| f \|_{L^1_t H^k_x(\R^3/\Z^3)}} 1$$
for each $k \geq 1$.  In particular, one has local existence whenever $T$ is sufficiently small depending on ${\mathcal H}^1(u_0,f,T,1)$.
\item[(iii)] (Uniqueness)  There is at most one periodic $H^1$ mild solution $(u,p,u_0,f,T,1)$ with the indicated data.
\item[(iv)] (Regularity) If $(u,p,u_0,f,T,1)$ is a periodic $H^1$ mild solution, and $(u_0,f,T,1)$ is smooth, then $(u,p,u_0,f,T,1)$ is smooth.
\item[(v)] (Lipschitz stability)  Let $(u,p,u_0,f,T,1)$ be a periodic $H^1$ mild solutions with the bounds $0 < T \leq T_0$ and
$$ \| u \|_{X^1([0,T] \times \R^3/\Z^3)} \leq M.$$
Let $(u'_0,f',T,1)$ be another set of periodic $H^1$ data, and define the function
$$ F(t) := e^{t\Delta} (u'_0-u_0) + \int_0^t e^{(t-t')\Delta} (f'(t')-f(t'))\ dt'.$$
If the quantity $\|F\|_{X^1([0,T] \times \R^3/\Z^3)}$ is sufficiently small depending on $T$, $M$, then there exists a periodic mild solution $(u',p',u'_0,f',T,1)$ with
$$ \| u-u'\|_{X^1([0,T] \times \R^3/\Z^3)} \lesssim_{T,M} \|F\|_{X^1([0,T] \times \R^3/\Z^3)}.$$
\end{itemize}
\end{theorem}

\begin{proof}  
We first prove the strong solution claim (i).  The linear solution
$$ e^{t\Delta} u_0 + \int_0^t e^{(t-t')\Delta} Pf(t')\ dt'$$
is easily verified to lie in $C^0_t H^1_x([0,T] \times \R^3/\Z^3)$, so in view of \eqref{duhamel-2}, it suffices to show that
$$ \int_0^t e^{(t-t')\Delta} P B(u(t'), u(t'))\ dt'$$
also lies in $C^0_t H^1_x([0,T] \times \R^3/\Z^3)$.  But as $u$ is an $H^1$ mild solution, $u$ lies in $X^1([0,T] \times \R^3/\Z^3)$, so by \eqref{bilinear}, $PB(u,u)$ lies in $L^4_t L^2_x([0,T] \times \R^3/\Z^3)$.  The claim (i) then follows easily from \eqref{energy-duh2}.

Now we establish local existence (ii).  Let $\delta := \|u_0\|_{H^1_x(\R^3/\Z^3)} + \| f \|_{L^1_t H^1_x(\R^3/\Z^3)}$, thus by \eqref{d4} we have $\delta^4 T \leq c$.  Using this and \eqref{bilinear-2}, \eqref{energy-duh2} one easily establishes that the nonlinear map $u \mapsto \Phi(u)$ defined by
$$ \Phi(u)(t) := e^{t\Delta} u_0 + \int_0^t e^{(t-t')\Delta} (PB(u(t'),u(t')+Pf(t'))\ dt'$$
is a contraction on the ball
$$ \{ u \in X^1([0,T] \times \R^3/\Z^3): \|u\|_{X^1([0,T] \times \R^3/\Z^3)} \leq C\delta \}$$
if $C$ is large enough.  From the contraction mapping principle we may then find a fixed point of $\Phi$ in this ball, and the claim (ii) follows (the estimates for higher $k$ follow from variants of the above argument and an induction on $k$, and are left to the reader).

Now we establish uniqueness (iii).  Suppose for contradiction that we have distinct solutions $(u,p,u_0,f,T,1)$, $(u',p',u_0,f,T,1)$ for the same data.  Then we have
$$ \| u \|_{X^1([0,T] \times \R^3/\Z^3)}, \|u'\|_{X^1([0,T] \times \R^3/\Z^3)} \leq M.$$
To show uniqueness, it suffices to do so assuming that $T$ is sufficiently small depending on $M$, as the general case then follows by subdividing $[0,T]$ into small enough time intervals and using induction.  Subtracting \eqref{duhamel-2} for $u, u'$ and writing $v := u'-u$, we see that
$$ v(t) = \int_0^t e^{(t-t')\Delta} P ( 2 B(u(t'), v(t')) + B(v(t'),v(t'))\ dt'$$
and thus by \eqref{energy-duh2}
$$ \| v \|_{X^1([0,T] \times \R^3/\Z^3)} \lesssim M T^{1/4} \| v \|_{X^1([0,T] \times \R^3/\Z^3)}.$$
If $T$ is sufficiently small depending on $M$, this forces $\| v \|_{X^1([0,T] \times \R^3/\Z^3)}=0$, giving uniqueness up to time $T$; iterating this argument gives the claim (iii).

Now we establish regularity (iv).  To abbreviate the notation, all norms will be on $[0,T] \times \R^3/\Z^3$. As $u$ is an $H^1$ mild solution, it lies in $X^1$, hence by \eqref{bilinear-2}, $P B(u,u)$ lies in $L^4_t L^2_x$.  Applying \eqref{duhamel-2}, \eqref{energy-duh3}, and the smoothness of $u_0, f$, we conclude that $u \in X^s$ for all $s < 3/2$.  In particular, by Sobolev embedding we see that $u \in L^\infty_t L^{12}_x$, $\nabla u \in L^2_t L^{12}_x \cap L^\infty_t L^{12/5}_x$ and $\nabla^2 u \in L^2_t L^{12/5}_x$, hence $P B(u,u) \in L^2_t H^1_x([0,T] \times \R^3/\Z^3)$.  Returning to \eqref{duhamel-2}, \eqref{energy-duh3}, we now conclude that $u \in X^2([0,T] \times \R^3/\Z^3)$.  One can then repeat these arguments iteratively to conclude that $u \in X^k([0,T] \times \R^3/\Z^3)$ for all $k \geq 1$, and thus $u \in L^\infty_t C^k([0,T] \times \R^3/\Z^3)$ for all $k \geq 0$.  From \eqref{pressure-point} we then have $p \in L^\infty C^k([0,T] \times \R^3/\Z^3)$ for all $k \geq 0$, and then from \eqref{ns} we have $\partial_t u \in L^\infty_t C^k([0,T] \times \R^3/\Z^3)$ for all $k \geq 0$.  One can then obtain bounds on $\partial_t p$ and then on higher time derivatives of $u$ and $t$, giving the desired smoothness, and the claim (iv) follows.

Now we establish stability (v).   It suffices to establish the claim in the short-time case when $T$ is sufficiently small depending only on $M$ (more precisely, we take $M^4 T \leq c$ for some sufficiently small absolute constant $c>0$), as the long-time case then follows by subdividing time and using induction.  The existence of the solution $(u',p',u'_0,f'_0,T,1)$ is then guaranteed by (ii).  Evaluating \eqref{duhamel-2} for $u,u'$ and subtracting, and setting $v := u' - u$, we see that
$$ v(t) = F + \int_{0}^t e^{(t-t')\Delta} P ( 2B(u,v) + B(v,v)) (t')\ dt'$$
for all $t \in [0,T]$.  Applying \eqref{energy-duh2}, \eqref{bilinear-2}, we conclude that
$$
\|v\|_{X^1}
\lesssim \|F\|_{X^1} 
+ T^{1/4} ( \| u \|_{X^1} + \| v \|_{X^1} ) \| v \|_{X^1}$$
where all norms are over $[t_0,t_1] \times \R^3$.
Since $\| u \|_{X^1} + \| v \|_{X^1}$ is finite, we conclude (if $T$ is small enough)
that $\|v\|_{X^1([0,T] \times \R^3/\Z^3)} \lesssim \|F\|_{X^1([0,T] \times \R^3/\Z^3)}$, as the claim follows.
\end{proof}

We may iterate the local well-posedness theory to obtain a dichotomy between existence and blowup.  Define an \emph{incomplete periodic mild $H^1$ solution} $(u,p,u_0,f,T_*^-,1)$ from periodic $H^1$ data $(u_0,f,T_*,1)$ to be fields $u: [0,T_*) \times \R^3/\Z^3 \to \R^3$ and $v: [0,T_*) \times \R^3/\Z^3 \to \R$ such that for any $0 < T < T_*$, the restriction $(u,p,u_0,f,T,1)$ of $(u,p,u_0,f,T_*^-,1)$ to the slab $[0,T] \times \R^3/\Z^3$ is a periodic mild $H^1$ solution.  We similarly define the notion of an incomplete periodic smooth solution.

\begin{corollary}[Maximal Cauchy development]\label{max-cauchy-periodic}  Let $(u_0,f,T,1)$ be periodic $H^1$ data.  Then at least one of the following two statements hold:
\begin{itemize}
\item There exists a periodic $H^1$ mild solution $(u,p,u_0,f,T,1)$ with the given data.
\item There exists a blowup time $0 < T_* < T$ and an incomplete periodic $H^1$ mild solution $(u,p,u_0,f,T_*^-,1)$ up to time $T_*^-$, which blows up in $H^1$ in the sense that
$$ \lim_{t \to T_*^-} \| u(t) \|_{H^1_x(\R^3/\Z^3)} = +\infty.$$
We refer to such solutions as \emph{maximal Cauchy developments}.
\end{itemize}
A similar statement holds with ``$H^1$ data'' and ``$H^1$ mild solution'' replaced by ``smooth data'' and ``smooth solution'' respectively.
\end{corollary}

Next we establish a compactness property of the periodic $H^1$ flow.

\begin{proposition}[Compactness]\label{compactness} If $(u^{(n)}_0,f^{(n)},T,1)$ is a sequence of periodic $H^1$ data obeying \eqref{meanzero-1}, \eqref{meanzero-2} which is uniformly bounded in $H^1_x(\R^3/\Z^3)_0 \times L^\infty_t H^1_x([0,T] \times \R^3/\Z^3)_0$ and converges weakly\footnote{Strictly speaking, we should use ``converges in the weak-* sense'' or ``converges in the sense of distributions'' here, in order to avoid the pathological (and irrelevant) elements of the dual space of $L^\infty_t H^1_x$ that can be constructed from the axiom of choice.} to $(u_0,f,T,1)$, and $(u,p,u_0,f,T,1)$ is a periodic $H^1$ mild solution with the indicated data, then for $n$ sufficiently large, there exists periodic $H^1$ mild solutions $(u^{(n)},p^{(n)},u_0^{(n)},f^{(n)},T,1)$ with the indicated data, with $u^{(n)}$ converging weakly in $X^1([0,T] \times \R^3/\Z^3)$ to $u$.  Furthermore, for any $0 < \tau < T$, $u^{(n)}$ converges strongly in $X^1([\tau,T] \times \R^3/\Z^3)$ to $u$.  

If $u^{(n)}_0$ converges strongly in $H^1_x(\R^3/\Z^3)_0$ to $u_0$, then one can set $\tau=0$ in the previous claim.
\end{proposition}

\begin{proof}  This result is essentially in \cite[Proposition 2.2]{tao-quantitative}, but for the convenience of the reader we give a full proof here. 

To begin with, we assume that $u^{(n)}$ converges strongly in $H^1_x(\R^3/\Z^3)_0$ to $u_0$, and relax this to weak convergence later.
In view of the stability component of Theorem \ref{lwp-h1}, it suffices to show that $F^{(n)}$ converges strongly in $X^1([0,T] \times \R^3/\Z^3)$ to zero, where
$$ F^{(n)}(t) := e^{t\Delta} (u_0^{(n)}-u_0) + \int_0^t e^{(t-t')\Delta} P (f^{(n)}(t')-f(t'))\ dt'.$$
We have $u_0^{(n)}-u_0$ converges strongly in $H^1_x(\R^3/\Z^3)$ to zero, while $f^{(n)} - f$ converges weakly in $L^\infty_t H^1_x([0,T] \times \R^3/\Z^3) \to 0$, and hence strongly in $L^2_t L^2_x([0,T] \times \R^3/\Z^3)$.  The claim then follows from \eqref{energy-duh2}.

Now we only assume that $u^{(n)}$ converges weakly in $H^1_x(\R^3/\Z^3)_0$ to $u_0$.  Let $0 < \tau < T$ be a sufficiently small time, then from local existence (Theorem \ref{lwp-h1}(ii)) we see that $u^{(n)}$ and $u$ are bounded in $X^1([0,\tau] \times \R^3/\Z^3)$ uniformly in $n$ by some finite quantity $M$.  Writing $v^{(n)} := u^{(n)} - u$, then from \eqref{duhamel-2} we have the difference equation
$$ v^{(n)}(t) = F^{(n)}(t) + \int_0^t e^{(t-t')\Delta} P( B( u, v^{(n)} ) + B(u^{(n)},v^{(n)}) )(t')\ dt'.$$
Since $u_0^{(n)}-u_0$ converges weakly in $H^1_x(\R^3/\Z^3)$ to zero, it converges strongly in $L^2_x(\R^3/\Z^3)$ to zero too.  Using \eqref{energy-duh} as before we see that $F^{(n)}$ converges strongly in $X^0([0,\tau] \times \R^3/\Z^3)$ to zero.  From \eqref{energy-duh2} we thus have
$$ \| v^{(n)} \|_{X^0} \lesssim o(1) + \| B(u,v^{(n)}) \|_{L^2_t H^{-1}_x}
+ \| B(u^{(n)},v^{(n)}) \|_{L^2_t H^{-1}_x}$$
where $o(1)$ goes to zero as $n \to \infty$, and all spacetime norms are over $[0,\tau] \times \R^3/\Z^3$.  From the form of $B$ and H\"older's inequality we have
\begin{align*}
\| B(u^{(n)},v^{(n)}) \|_{L^2_t H^{-1}_x}
&\lesssim
\| \bigO(u^{(n)} v^{(n)}) \|_{L^2_t L^2_x}\\
&\lesssim \tau^{1/4} \| u^{(n)} \|_{L^\infty_t L^6_x} \| v^{(n)} \|_{L^\infty_t L^2_x}^{1/2} 
 \| v^{(n)} \|_{L^2_t L^6_x}^{1/2} \\
&\lesssim M \tau^{1/4} \|v^{(n)}\|_{X^0}
\end{align*}
and similarly for $B(u,v^{(n)})$, and thus
$$ \| v^{(n)} \|_{X^0} \lesssim o(1) + M \tau^{1/4} \|v^{(n)}\|_{X^0}.$$
Thus, for $\tau$ small enough, one has
$$ \| v^{(n)} \|_{X^0} = o(1),$$
which among other things gives weak convergence of $u^{(n)}$ to $u$ in $[0,\tau] \times \R^3/\Z^3$.  Also, by the pigeonhole principle, one can find times $\tau^{(n)}$ in $[0,\tau]$ such that
$$ \| v^{(n)}(\tau^{(n)}) \|_{H^1_x(\R^3/\Z^3)} = o(1).$$
Using the stability theory, and recalling that $\tau$ is small, this implies that
$$ \| v^{(n)}(\tau) \|_{H^1_x(\R^3/\Z^3)} = o(1),$$
thus $u^{(n)}(\tau)$ converges strongly to $u(\tau)$.  Now we can use our previous arguments to extend $u^{(n)}$ to all of $[0,T] \times \R^3/\Z^3$ and obtain strong convergence in $X^1([\tau,T] \times \R^3/\Z^3)$ as desired.
\end{proof}

Now we turn to the non-periodic setting.  We have the following analogue of Theorem \ref{lwp-h1}:

\begin{theorem}[Local well-posedness in $H^1$]\label{lwp-h1-r3}  Let $(u_0,f,T)$ be $H^1$ data.
\begin{itemize}
\item[(i)] (Strong solution) If $(u,p,u_0,f,T,1)$ is an $H^1$ mild solution, then
$$ u \in C^0_t H^1_x([0,T] \times \R^3).$$
\item[(ii)] (Local existence and regularity)  If
\begin{equation}\label{D4}
 (\|u_0\|_{H^1_x(\R^3)} + \| f \|_{L^1_t H^1_x(\R^3)})^4 T \leq c
\end{equation}
for a sufficiently small absolute constant $c>0$, then there exists a $H^1$ mild solution $(u,p,u_0,f,T)$ with the indicated data, with
$$ \|u\|_{X^1([0,T] \times \R^3)} \lesssim \|u_0\|_{H^1_x(\R^3)} + \| f \|_{L^1_t H^1_x(\R^3)}$$
and more generally
$$ \| u \|_{X^k([0,T] \times \R^3)} \lesssim_{k,\|u_0\|_{H^k_x(\R^3)},\| f \|_{L^1_t H^k_x(\R^3)},1}$$
for each $k \geq 1$.  In particular, one has local existence whenever $T$ is sufficiently small depending on ${\mathcal H}^1(u_0,f,T)$.  
\item[(iii)] (Uniqueness) There is at most one $H^1$ mild solution $(u,p,u_0,f,T)$ with the indicated data.
\item[(iv)] (Regularity)  If $(u,p,u_0,f,T,1)$ is a $H^1$ mild solution, and $(u_0,f,T)$ is Schwartz, then $u$ and $p$ are smooth; in fact, one has $\partial_t^j u, \partial_t^j p \in L^\infty_t H^k([0,T] \times \R^3)$ for all $j,k \geq 0$.
\item[(v)] (Lipschitz stability)  Let $(u,p,u_0,f,T)$, $(u',p',u'_0,f',T)$ be $H^1$ mild solutions with the bounds $0 < T \leq T_0$ and
$$ \| u \|_{X^1([0,T] \times \R^3)}, \| u' \|_{X^1([0,T] \times \R^3)} \leq M.$$
Define the function
$$ F(t) := e^{t\Delta} (u'_0-u_0) + \int_0^t e^{(t-t')\Delta} (f'(t')-f(t'))\ dt'.$$
If the quantity $\|F\|_{L^2_t L^2_x([0,T] \times \R^3)}$ is sufficiently small depending on $T$, $M$, then
$$ \| u-u'\|_{X^1([0,T] \times \R^3)} \lesssim_{T,M} \|F\|_{L^2_t L^2_x([0,T] \times \R^3)}.$$
\end{itemize}
\end{theorem}

\begin{proof} This proceeds by repeating the proof of Theorem \ref{lwp-h1} verbatim.  The one item which perhaps requires some care is the regularity item (iv).  The arguments from Theorem \ref{lwp-h1} yield the regularity
$$ u \in X^k([0,T] \times \R^3)$$
for all $k \geq 0$ without difficulty.  In particular, $u \in L^\infty_t H^k_x([0,T] \times \R^3)$ for all $k \geq 0$.  From \eqref{pressure-point} and Sobolev embedding one then has $p \in L^\infty_t H^k_x([0,T] \times \R^3)$ for all $k \geq 0$, and then from \eqref{ns} and more Sobolev embedding one has
$\partial_t u \in L^\infty_t H^k_x([0,T] \times \R^3)$ for all $k \geq 0$.  One can then obtain bounds on $\partial_t p$ and then on higher time derivatives of $u$ and $t$, giving the desired smoothness, and the claim (iv) follows.  (Note that these arguments did not require the full power of the hypothesis that $(u_0, f, T)$ was Schwartz; it would have sufficed to have $u_0 \in H^k_x(\R^3)$ and $f \in C^j_t H^k_x(\R^3)$ for all $j,k \geq 0$.)
\end{proof}

From the regularity component of the above theorem, we immediately conclude that Conjecture \ref{global-h1-spatial} implies Conjecture \ref{global-schwartz}, which is one half of Theorem \ref{main}(iv). 

We will also need a more quantitative version of the regularity statement in Theorem \ref{lwp-h1-r3}.

\begin{lemma}[Quantitative regularity]\label{quant-reg}  Let $(u,p,u_0,f,T)$ be an $H^1$ mild solution obeying \eqref{D4} for a sufficiently small absolute constant $c > 0$, and such that
$$ \|u_0\|_{H^1_x(\R^3)} + \| f \|_{L^1_t H^k_x(\R^3)} \leq M < \infty.$$
Then one has
$$ \|u\|_{L^\infty_t H^k_x([\tau,T] \times \R^3)} \lesssim_{k,\tau,T,M} 1$$
for all natural numbers $k \geq 1$ and all $0 < \tau < T$.
\end{lemma}

\begin{proof}  We allow all implied constants to depend on $k, T, M$.  From Theorem \ref{lwp-h1} we have
$$ \|u\|_{X^1([0,T] \times \R^3)} \lesssim 1$$
which already gives the $k=1$ case.  Now we turn to the $k \geq 2$ case.  From \eqref{bilinear-2} we have
$$ \| P B(u,u) \|_{L^4_t L^2_x([0,T] \times \R^3)} \lesssim 1$$
while from Fourier analysis one has
$$ \| e^{t\Delta} u_0 \|_{L^\infty_t H^k_x([\tau,T] \times \R^3)} \lesssim_\tau 1.$$
From this and \eqref{duhamel-2}, \eqref{energy-duh} we see that
$$ \|u\|_{X^s([\tau,T] \times \R^3)} \lesssim_{s,\tau} 1$$
for all $s < 3/2$.  From Sobolev embedding we conclude
\begin{align*}
\|u\|_{L^\infty_t L^{12}_x([\tau,T] \times \R^3)} &\lesssim_{\tau} 1,\\
\|\nabla u\|_{L^2_t L^{12}_x([\tau,T] \times \R^3)} &\lesssim_{\tau} 1,\\
\|\nabla u\|_{L^\infty_t L^{12/5}_x([\tau,T] \times \R^3)} &\lesssim_{\tau} 1,\\
\|\nabla^2 u\|_{L^2_t L^{12}_x([\tau,T] \times \R^3)} &\lesssim_{\tau} 1,
\end{align*}
and hence
$$ \| P B(u,u) \|_{L^2_t H^1_x([\tau,T] \times \R^3)} \lesssim 1.$$
Returning to \eqref{duhamel-2}, \eqref{energy-duh3} we now conclude that
$$ \|u\|_{X^2([\tau,T] \times \R^3)} \lesssim_\tau 1$$
which gives the $k = 2$ case.  One can repeat these arguments iteratively to then give the higher $k$ cases.
\end{proof}

We extract a particular consequence of the above lemma:

\begin{proposition}[Almost regularity]\label{almost-smooth} Let $(u,p,u_0,0,T)$ be a homogeneous $H^1$ mild solution obeying \eqref{D4} for a sufficiently small absolute constant $c > 0$.  Then $u, p$ are smooth on $[\tau,T] \times \R^3$ for all $0 < \tau < T$; in fact, all derivatives of $u, p$ lie in $L^\infty_t L^2_x([\tau,T] \times \R^3)$.  If furthermore $u_0$ is smooth, then $(u,p,u_0,0,T)$ is an almost smooth solution.
\end{proposition}

\begin{proof}  From Lemma \ref{quant-reg} we see that
$$ u \in L^\infty_t H^k_x([\tau,T] \times \R^3)$$
for all $k \geq 0$ and $0 < \tau < T$.  Arguing as in the proof of Theorem \ref{lwp-h1-r3}(iv) we conclude that $u,p$ are smooth on $[\tau,T] \times \R^3$.

Now suppose that $u_0$ is smooth.  Then (since $u_0$ is also in $H^1_x(\R^3)$) $e^{t\Delta} u_0$ is smooth\footnote{To obtain smoothness at a point $(t_0,x_0)$, one can for instance split $u_0$ into a smooth compactly supported component, and a component that vanishes near $x_0$ but lies in $H^1_x(\R^3)$, and verify that the contribution of each component to $e^{t\Delta} u_0$ is smooth at $(t_0,x_0)$.} on $[0,T] \times \R^3$, and in particular one has
$$ \eta e^{t\Delta} u_0  \in L^\infty_t H^k_x([0,T] \times \R^3)$$
for any smooth, compactly supported cutoff function $\eta: \R^3 \to \R$.  Meanwhile, by arguing as in Lemma \ref{quant-reg} one has
\begin{equation}\label{pebble}
P B(u,u) \in L^4_t L^2_x([0,T] \times \R^3).
\end{equation}
Using \eqref{duhamel-2}, \eqref{energy-duh} one concludes that
$$ \eta u \in X^s([0,T] \times \R^3)$$
for all cutoff functions $\eta$ and all $s < 3/2$.  Continuing the arguments from Lemma \ref{quant-reg} we conclude that
$$ \eta P B(u,u) \in L^2_t H^1_x([0,T] \times \R^3)$$
for all cutoffs $\eta$.  Using \eqref{duhamel-2}, \eqref{energy-duh3} (and using \eqref{kop}, \eqref{pebble}, to deal with the far field contribution of $P B(u,u)$, and shrinking $\eta$ as necessary) one then concludes that
$$ \eta u \in X^2([0,T] \times \R^3)$$
for all cutoffs $\eta$.  Repeating these arguments iteratively one eventually concludes that
$$ \eta u \in X^k([0,T] \times \R^3)$$
for all cutoffs $\eta$, and in particular
$$ u \in L^\infty_t H^k_x([0,T] \times K)$$
for all $k \geq 0$ and all compact sets $K$.  By Sobolev embedding, this implies that
$$ u \in L^\infty_t C^k_x([0,T] \times K)$$
for all $k \geq 0$ and all compact sets $K$.  

We also have $u \in X^1([0,T] \times \R^3)$, and hence
$$ u \in L^\infty_t H^1_x([0,T] \times \R^3).$$
In particular, 
\begin{equation}\label{lii}
 u_i u_j \in L^\infty_t L^1_x([0,T] \times \R^3)
\end{equation}
and
$$ u_i u_j \in L^\infty_t C^k_x([0,T] \times K)$$
for all $k \geq 0$ and compact $K$.  From this and \eqref{pressure-point} (splitting the inverse Laplacian $\Delta^{-1}$ smoothly into local and global components) one has
$$ p \in L^\infty_t C^k_x([0,T] \times K);$$
inserting this into \eqref{ns} we then see that
\begin{equation}\label{tu}
 \partial_t u \in L^\infty_t C^k_x([0,T] \times K)
\end{equation}
for all $k \geq 0$ and compact $K$.

This is a little weaker than what we need for an almost smooth solution, because we want $\nabla^k u, \nabla^k p, \partial_t \nabla^k p$ to extend continuously down to $t=0$, and the above estimates merely give $L^\infty_t C^\infty_x$ control on these quantities.  To upgrade the $L^\infty_t$ control to continuity in time, we first observe\footnote{An alternate argument here would be to approximate the initial data $u_0$ by Schwartz divergence-free data (using Lemma \ref{divloc}) and using a limiting argument and the stability and regularity theory in Theorem \ref{lwp-h1}; we omit the details.} from \eqref{tu} and integration in time that we can at least make $\nabla^k u$ extend continuously to $t=0$:
$$
 u \in C^0_t C^k_x([0,T] \times K).
$$
In particular
\begin{equation}\label{uuk}
 u_i u_j \in C^0_t C^k_x([0,T] \times K)
\end{equation}
for all $k \geq 0$ and compact $K$.

Now we consider $\nabla^k p$ in a compact region $[0,T] \times K$.  From \eqref{pressure-point} we have
$$ \nabla^k p(t,x) =  \nabla^k \partial_i \partial_j \int_{\R^3} \frac{1}{4\pi|x-y|} u_i u_j(t,y)\ dy.$$
Using a smooth cutoff, we split the Newton potential $\frac{1}{4\pi|x-y|}$ into a ``local'' portion supported on $B(0,2R)$, and a ``global'' portion supported outside of $B(0,R)$, where $R$ is a large radius.  From \eqref{uuk} one can verify that the contribution of the local portion is continuous on $[0,T] \times K$, while from \eqref{lii} the contribution of the global portion is $O_u(1/R^3)$.  Sending $R \to \infty$ we conclude that $\nabla^k p$ is continuous on $[0,T] \times K$, and thus
$$ p \in C^0_t C^k_x([0,T] \times K)$$
for all $k \geq 0$ and compact $K$.  Inserting this into \eqref{ns} we then conclude that
$$ \partial_t u \in C^0_t C^k_x([0,T] \times K)$$
for all $k \geq 0$ and compact $K$, and so we have an almost smooth solution as required.
\end{proof}

\begin{remark}  Because $u$ has the regularity of $L^\infty_t H^1_x$, we can continue iterating the above argument a little more, and eventually get $u \in C^2_t C^k_x([0,T] \times K)$ and $p \in C^1_t C^k_x([0,T] \times K)$ for all $k \geq 0$ and compact $K$.  Using the vorticity equation (see \eqref{vorticity-eq} below) one can then also get $\omega \in C^3_t C^k_x([0,T] \times K)$ as well.  But without further decay conditions on higher derivatives of $u$ (or of $\omega$), one cannot gain infinite regularity on $u, p, \omega$ in time; see Section \ref{counter-sec}.

On the other hand, it is possible to use energy methods and the vorticity equation \eqref{vorticity-eq} to show (working in the homogeneous case $f=0$ for simplicity) that if $u_0$ is smooth and the initial vorticity $\omega_0 := \nabla \times u_0$ is Schwartz, then the solution in Proposition \ref{almost-smooth} is in fact smooth, with $\omega$ remaining Schwartz throughout the lifespan of that solution; we omit the details.
\end{remark}

As a corollary of the above proposition we see that Conjecture \ref{global-h1-spatial} implies Conjecture \ref{global-enstrophy-homog}, thus completing the proof of Theorem \ref{main}(iv).

As before, we obtain a dichotomy between existence and blowup.  Define an \emph{incomplete mild $H^1$ solution} $(u,p,u_0,f,T_*^-)$ from $H^1$ data $(u_0,f,T_*)$ to be fields $u: [0,T_*) \times \R^3 \to \R^3$ and $v: [0,T_*) \times \R^3 \to \R$ such that for any $0 < T < T_*$, the restriction $(u,p,u_0,f,T,1)$ of $(u,p,u_0,f,T_*^-,1)$ to the slab $[0,T] \times \R^3$ is a mild $H^1$ solution.  We similarly define the notion of an incomplete smooth $H^1$ solution.

\begin{corollary}[Maximal Cauchy development]\label{max-cauchy}  Let $(u_0,f,T)$ be $H^1$ data.  Then at least one of the following two statements hold:
\begin{itemize}
\item There exists a mild $H^1$ solution $(u,p,u_0,f,T)$ with the given data.
\item There exists a blowup time $0 < T_* < T$ and an incomplete mild $H^1$ solution $(u,p,u_0,f,T_*^-)$ up to time $T_*^-$, which blows up in the enstrophy norm in the sense that
$$ \lim_{t \to T_*^-} \| u(t) \|_{H^1_x(\R^3)} = +\infty.$$
\end{itemize}
\end{corollary}

\begin{remark} In the second conclusion of Corollary \ref{max-cauchy}, more information about the blowup is known.  For instance, in the paper \cite{ess} it was demonstrated that the $L^3_x(\R^3)$ norm must also blow up (in the homogeneous case $f=0$, at least).
\end{remark}

\section{Homogenisation}\label{homog-sec}

In this section we prove Proposition \ref{force}.

Fix smooth periodic data $(u_0,f,T,L)$; our objective is to find a smooth periodic solution $(u,p,u_0,f,T,L)$ (without pressure normalisation) with this data.  By the scaling symmetry \eqref{scaling} we may normalise the period $L$ to equal $1$.  Using the symmetry \eqref{gal-2} we may impose the mean zero conditions \eqref{meanzero-1}, \eqref{meanzero-2} on this data.

By hypothesis, one can find a smooth periodic solution $(\tilde u, \tilde p,u_0,0,T,1)$ with data $(u_0,0,T,1)$.  By Lemma \ref{reduction}, and applying a Galilean transform \eqref{galilean} if necessary, we may assume the pressure is normalised, which in particular makes $(\tilde u, \tilde p, u_0,0,T,1)$ a periodic $H^1$ mild solution. 

By the Galilean invariance \eqref{galilean} (with a linearly growing velocity $v(t) := 2wt$), it suffices to find a smooth periodic solution $(u,p,u_0,f_w,T)$ (this time \emph{with} pressure normalisation) for the Galilean-shifted data $(u_0,f_w,T)$, where
$$ f_w(t,x) := f(t,x-w t^2),$$
and $w\in\R^3$ is arbitrary.
Note that the data $(u_0,f_w,T)$ continues to obey the mean zero conditions \eqref{meanzero-1}, \eqref{meanzero-2}, and is bounded in $H^1_x(\R^3/\Z^3)_0 \times L^\infty_t H^1_x([0,T] \times \R^3/\Z^3)_0$ uniformly in $w$.  We now make a key observation:

\begin{lemma}  If $\alpha \in \R^3/\Z^3$ is \emph{irrational} in the sense that $k \cdot \alpha \neq 0$ in $\R/\Z$ for all $k \in \Z^3 \backslash \{0\}$, then $f_{\lambda \alpha}$ converges weakly (or more precisely, converges in the sense of spacetime distributions) to zero in $L^\infty_t H^1_x([0,T] \times \R^3/\Z^3)_0$.
\end{lemma}

\begin{proof}  It suffices to show that
$$ \int_0^T \int_{\R^3/\Z^3} f_{\lambda \alpha}(t,x) \phi(t,x)\ dx dt \to 0$$
for all smooth functions $\phi: [0,T] \times \R^3/\Z^3 \to \R$.  Taking the Fourier transform, the left-hand side becomes
$$ \sum_{k \in \Z^3} \int_0^T e^{-2\pi i \lambda k t^2 \cdot \alpha} \widehat{f(t)}(k) \widehat{\phi(t)}(-k)\ dt,$$
with the sum being absolutely convergent due to the rapid decrease of the Fourier transform of $\phi(t)$.
Because $f$ has mean zero, we can delete the $k=0$ term from the sum.  This makes $k \cdot \alpha$ non-zero by irrationality, and so by the Riemann-Lebesgue lemma, each summand goes to zero as $\lambda \to \infty$.  The claim then follows from the dominated convergence theorem.
\end{proof}

Let $\alpha \in \R^3/\Z^3$ be irrational.  By the above lemma, $(u_0,f_{\lambda \alpha},T,1)$ converges weakly to $(u_0,0,T,1)$ while being bounded in 
$H^1_x(\R^3/\Z^3)_0 \times L^\infty_t H^1_x(\R^3/\Z^3)_0$.  As $(u_0,0,T,1)$ has a periodic mild $H^1$ solution $(\tilde u,\tilde p,u_0,0,T,1)$, we conclude from Proposition \ref{compactness} that for $\lambda$ sufficiently large, $(u_0,f_{\lambda\alpha},T,1)$ also has a periodic mild $H^1$ solution, which is necessarily smooth since $u_0, f_{\lambda\alpha}$ is smooth.  The claim follows.

\begin{remark}  Suppose that $(u_0,f,\infty,1)$ is periodic $H^1$ data extending over the half-infinite time interval $[0,+\infty)$.  The above argument shows (assuming Conjecture \ref{global-periodic-homog}) that one can, for each $0 < T < \infty$, construct a smooth periodic (but not pressure normalised) solution $(u^{(T)},p^{(T)},u_0,f,T,1)$ up to time $T$ with the above data, by choosing a sufficiently rapidly growing linear velocity $v^{(T)} = 2w^{(T)} t$, applying a Galilean transform, and then using the compactness properties of the $H^1$ local well-posedness theory.  As stated, this argument gives a different solution $(u^{(T)},p^{(T)},u_0,f,T,1)$ for each time $T$ (note that we do not have uniqueness once we abandon pressure normalisation).  However, it is possible to modify the argument to obtain a single global smooth periodic solution $(u,p,u_0,f,\infty,1)$ (which is still not pressure normalised, of course), by using the ability in \eqref{galilean} to choose a \emph{nonlinear} velocity $v(t)$ rather than a linear one.  By reworking the above argument, and taking $v(t)$ to be a sufficiently rapidly growing function of $t$, it is then possible to obtain a global smooth periodic solution $(u,p,u_0,f,\infty,1)$ to the indicated data; we omit the details.
\end{remark}

\section{Compactness}\label{compact-sec}

In this section we prove Theorem \ref{main}(i), by following the compactness arguments of \cite{tao-quantitative}.  By the scaling symmetry \eqref{scaling}, we may normalise $L=1$.

We first assume that Conjecture \ref{global-h1-quant} holds, and deduce Conjecture \ref{global-h1}.  Suppose for contradiction that Conjecture \ref{global-h1} failed.  By Corollary \ref{max-cauchy-periodic}, there thus exists an incomplete periodic pressure-normalised mild $H^1$ solution $(u,p,u_0,f,T_*^-,1)$ such that
\begin{equation}\label{blow}
 \lim_{t \to T_*^-} \|u(t)\|_{H^1_x(\R^3/\Z^3)} = \infty.
\end{equation}
By Galilean invariance \eqref{gal-2} we may assume that $u_0$ and $f$ (and hence $u$) have mean zero.

Let $(u_0^{(n)}, f^{(n)}, T_*, 1)$ be a sequence of periodic smooth mean zero data converging strongly in $H^1_x(\R^3/\Z^3)_0 \times L^\infty_t H^1_x([0,T_*] \times \R^3/\Z^3)_0$ to the periodic $H^1$ data $(u,f,T_*,1)$.  For each time $0 < T < T_*$, we see from Theorem \ref{lwp-h1} that for $n$ sufficiently large, we may find a smooth solution $(u^{(n)}, p^{(n)}, u_0^{(n)}, T, 1)$ with this data, with $u^{(n)}$ converging strongly in $L^\infty_t H^1_x([0,T] \times \R^3/\Z^3)$ to $u$.  By Conjecture \ref{global-h1-quant}, the $L^\infty_t H^1_x([0,T] \times \R^3/\Z^3)$ norm of $u^{(n)}$ is bounded uniformly in both $T$ and $n$, so by taking limits as $n \to \infty$ we conclude that $\|u(t)\|_{H^1_x(\R^3/\Z^3)}$ is bounded uniformly for $0 \leq t < T_*$, contradicting \eqref{blow} as desired.

Conversely, suppose that Conjecture \ref{global-h1} held, but Conjecture \ref{global-h1-quant} failed.  Carefully negating all the quantifiers, we conclude that there exists a time $0 < T_0 < \infty$ and a sequence $(u^{(n)}, p^{(n)}, u_0^{(n)}, f^{(n)}, T^{(n)}, 1)$ of smooth periodic data with $0 < T^{(n)} < T_0$ and ${\mathcal H}^1(u_0^{(n)}, f^{(n)}, T^{(n)},1)$ uniformly bounded in $n$, such that
\begin{equation}\label{hot}
 \lim_{n \to \infty} \| u \|_{L^\infty_t H^1_x([0,T^{(n)}] \times \R^3/\Z^3)} = \infty.
\end{equation}
Using Galilean transforms \eqref{gal-2} we may assume that $u_0^{(n)}, f^{(n)}$ (and hence $u^{(n)}$) have mean zero.
From the short-time local existence (and uniqueness) theory in Theorem \ref{lwp-h1} we see that $T^{(n)}$ is bounded uniformly away from zero.  Thus by passing to a subsequence we may assume that $T^{(n)}$ converges to a limit $T_*$ with $0 < T_* \leq T_0$.  

By sequential weak compactness, we may pass to a further subsequence and assume that for each $0 < T < T_*$, $(u_0^{(n)}, f^{(n)}, T,1)$ converges weakly (or more precisely, in the sense of distributions) to a periodic $H^1$ limit $(u_0, f, T, 1)$; gluing these limits together one obtains periodic $H^1$ data $(u_0,f,T_*, 1)$, which still has mean zero.  By Conjecture \ref{global-h1}, we can then find a periodic $H^1$ mild solution $(u,p,u_0,f,T_*,1)$ with this data, which then necessarily also has mean zero.  

By Theorem \ref{lwp-h1} and Proposition \ref{compactness}, we see that for every $0 < \tau < T < T_*$, $u^{(n)}$ converges strongly in $L^\infty_t H^1_x([\tau,T] \times \R^3/\Z^3)$ to $u$.  In particular, for any $0 < T < T_*$, one has
$$ \limsup_{n \to \infty} \| u^{(n)}(T) \|_{H^1_x(\R^3/\Z^3)} \leq \| u \|_{L^\infty_t H^1_x([0,T_*] \times \R^3/\Z^3)} < \infty.$$
Taking $T$ sufficiently close to $T_*$ and then taking $n$ sufficiently large, we conclude from Theorem \ref{lwp-h1} that
$$ \limsup_{n \to \infty} \| u^{(n)} \|_{L^\infty_t H^1_x([T,T^{(n)}] \times \R^3/\Z^3)} < \infty;$$
also, from the strong convergence in $L^\infty_t H^1_x([\tau,T] \times \R^3/\Z^3)$ we have
$$ \limsup_{n \to \infty} \| u^{(n)} \|_{L^\infty_t H^1_x([\tau,T] \times \R^3/\Z^3)} < \infty$$
for any $0 < \tau < T$, and finally from the local existence (and uniqueness) theory in Theorem \ref{lwp-h1} one has
$$ \limsup_{n \to \infty} \| u^{(n)} \|_{L^\infty_t H^1_x([0,\tau] \times \R^3/\Z^3)} < \infty$$
for sufficiently small $\tau$. Putting these bounds together, we contradict \eqref{hot}, and the claim follows.

\begin{remark} It should be clear to the experts that one could have replaced the $H^1$ regularity in the above conjectures by other subcritical regularities, such as $H^k$ for $k>1$, and obtain a similar result to Theorem \ref{main}(i).
\end{remark}

As remarked previously, the homogeneous case $f=0$ of Theorem \ref{main}(i) was established in \cite{tao-quantitative}.  We recall the main results of that paper.  We introduce the following homogeneous periodic conjectures:

\begin{conjecture}[\emph{A priori} homogeneous periodic $H^1$ bound]\label{global-h1-quant-homog}  There exists a function $F: \R^+ \times \R^+ \times \R^+ \to \R^+$ with the property that whenever $(u,p,u_0,0,T,L)$ is a smooth periodic, homogeneous normalised-pressure solution with $0 < T < T_0 < \infty$ and
$$ {\mathcal H}^1(u_0,0,T,L) \leq A < \infty$$
then
$$ \| u \|_{L^\infty_t H^1_x([0,T] \times \R^3/L\Z^3)} \leq F( A, L, T_0 ).$$
\end{conjecture}

\begin{conjecture}[\emph{A priori} homogeneous global periodic $H^1$ bound]\label{global-h1-quant-global-period}  There exists a function $F: \R^+ \times \times \R^+ \to \R^+$ with the property that whenever $(u,p,u_0,0,T,L)$ is a smooth periodic, homogeneous normalised-pressure solution with
$$ {\mathcal H}^1(u_0,0,T,L) \leq A < \infty$$
then
$$ \| u \|_{L^\infty_t H^1_x([0,T] \times \R^3/L\Z^3)} \leq F( A, L ).$$
\end{conjecture}

\begin{conjecture}[Global well-posedness in periodic homogeneous $H^1$]\label{global-h1-periodic} Let $(u_0,0,T,L)$ be a homogeneous periodic $H^1$ set of data.  Then there exists a periodic $H^1$ mild solution $(u,p,u_0,0,T,L)$ with the indicated data.
\end{conjecture}

\begin{conjecture}[Global regularity for homogeneous periodic data with normalised pressure]\label{global-periodic-normalised-homog}  Let $(u_0,0,T)$ be a smooth periodic set of data.  Then there exists a smooth periodic solution $(u,p,u_0,0,T)$ with the indicated data and with normalised pressure.
\end{conjecture}

In \cite[Theorem 1.4]{tao-quantitative} it was shown that Conjectures \ref{global-periodic-homog}, \ref{global-h1-quant-homog}, \ref{global-h1-quant-global-period} are equivalent.  As implicitly observed in that paper also, Conjecture \ref{global-periodic-homog} is equivalent to Conjecture \ref{global-periodic-normalised-homog} (this can be seen from Lemma \ref{reduction}), and from the local well-posedness and regularity theory (Theorem \ref{lwp-h1} or \cite[Proposition 2.2]{tao-quantitative}) we also see that Conjecture \ref{global-periodic-normalised-homog} is equivalent to Conjecture \ref{global-h1-periodic}.

\section{Energy localisation}\label{energy-sec}

In this section we establish the energy inequality for the Navier-Stokes equation in the smooth finite energy setting.  This energy inequality is utterly standard (see e.g. \cite{sch}) for weaker notions of solutions, so long as one has regularity of $L^2_t H^1_x$, but (somewhat ironically) requires more care in the smooth finite energy setting, because we do \emph{not} assume \emph{a priori} that smooth finite energy solutions lie in the space $L^2_t H^1_x$.  The methods used here are local in nature, and will also provide an energy  localisation estimate for the Navier-Stokes equation (see Theorem \ref{local}).

We begin with the global energy inequality.

\begin{lemma}[Global energy inequality]\label{energy-est}  Let $(u,p,u_0,f,T)$ be a finite energy almost smooth solution.  Then
\begin{equation}\label{ul2-en}
 \| u \|_{L^\infty_t L^2_x([0,T] \times \R^3)} + \| \nabla u \|_{L^2_t L^2_x([0,T] \times \R^3)} \lesssim E(u_0,f,T)^{1/2}.
 \end{equation}
In particular, $u$ lies in the space $X^1([0,T] \times \R^3)$.
\end{lemma}

\begin{proof}  To abbreviate the notation, all spatial norms here will be over $\R^3$.

Using the forcing symmetry \eqref{forcing}, we may set $f$ to be divergence-free, so in particular by Corollary \ref{as-mild} we have 
\begin{equation}\label{napp}
\nabla p(t) = \nabla \tilde p(t)
\end{equation}
for almost all times $t$, where
\begin{equation}\label{pps}
\tilde p = -\Delta^{-1} \partial_i \partial_j (u_i u_j).
\end{equation}
As $(u,p,u_0,f,T)$ is finite energy, we have the \emph{a priori} hypothesis
$$ \|u\|_{L^\infty_t L^2_x([0,T] \times \R^3)} \leq A$$
for some $A<\infty$, though recall that our final bounds are not allowed to depend on this quantity $A$.  Because $u$ is smooth, we see in particular from Fatou's lemma that
\begin{equation}\label{uta}
 \|u(t)\|_{L^2_x} \leq A
\end{equation}
for all $t \in [0,T]$.

Taking the inner product of the Navier-Stokes equation \eqref{ns} with $u$ and rearranging, we obtain the energy density identity
\begin{equation}\label{energy-ident}
\partial_t (\frac{1}{2} |u|^2) + u \cdot \nabla (\frac{1}{2} |u|^2) = \Delta(\frac{1}{2} |u|^2) - |\nabla u|^2 - u \cdot \nabla p + u \cdot f.
\end{equation}
We would like to integrate this identity over all of $\R^3$, but we do not yet have enough decay in space to achieve this, even with the normalised pressure.  Instead, we will localise by integrating the identity against a cutoff $\eta^4$, where $\eta(x) := \chi(\frac{|x|-R}{r})$, and $\chi: \R \to \R^+$ is a fixed smooth function that equals $0$ on $[0,+\infty]$ and $1$ on $[-\infty,-1]$, and $0 < r < R/2$ are parameters to be chosen later.  (The exponent $4$ is convenient for technical reasons, in that $\eta^4$ and $\nabla(\eta^4)$ share a large common factor $\eta^3$, but it should be ignored on a first reading.) Thus we see that $\eta^4$ is supported on the ball $B(0,R)$ and equals $1$ on $B(0,R-r)$, with the derivative bounds
\begin{equation}\label{nablab}
\nabla^j \eta = O( r^{-j} )
\end{equation}
for $j=0,1,2$.  We define the localised energy
\begin{equation}\label{ert}
E_{\eta^4}(t) := \int_{\R^3} \frac{1}{2} |u|^2(t,x) \eta^4(x)\ dx.
\end{equation}
Clearly we have the initial condition
\begin{equation}\label{initial-cond}
E_{\eta^4}(0) \lesssim E(u_0,f).
\end{equation}
Because $\eta^4$ is compactly supported and $u$ is almost smooth, $E_{\eta^4}$ is $C^1_t$, and we may differentiate under the integral sign and integrate by parts without difficulty; using \eqref{napp}, we see for almost every time $t$ that
\begin{equation}
\partial_t E_{\eta^4} = -X_1 + X_2 + X_3 + X_4 + X_5
\end{equation}
where $X_1$ is the dissipation term
\begin{equation}\label{x1-def}
X_1 := \int_{\R^3} |\nabla u|^2 \eta^4\ dx = \| \eta^2 \nabla u \|_{L^2_x}^2,
\end{equation}
$X_2$ is the heat flux term
$$ X_2 := \frac{1}{2} \int_{\R^3} |u|^2 \Delta (\eta^4)\ dx,$$
$X_3$ is the transport term
$$ X_3 := 4 \int_{\R^3} |u|^2 u \cdot \eta^3 \nabla \eta\ dx,$$
$X_4$ is the forcing term
$$ X_4 := \int_{\R^3} u \cdot f \eta^4\ dx,$$
and $X_5$ is the pressure term
$$ X_5 := 4 \int_{\R^3} u \tilde p \eta^3 \nabla \eta\ dx.$$
The dissipation term $X_1$ is non-negative, and will be useful in controlling some of the other terms present here. The heat flux term $X_2$ can be bounded using \eqref{uta} and \eqref{nablab} by 
$$ X_2 \lesssim \frac{A^2}{r^2},$$
so we turn now to the transport term $X_3$.  Using H\"older's inequality and \eqref{nablab}, we may bound
\begin{equation}\label{x4-bound}
 X_3 \lesssim \frac{1}{r} \| u \eta^2 \|_{L^6_x}^{3/2} \| u \|_{L^2_x}^{3/2} 
\end{equation}
and thus by \eqref{uta} and Sobolev embedding
$$ X_3 \lesssim \frac{A^{3/2}}{r} \| \nabla( u \eta^2 ) \|_{L^2_x}^{3/2}.$$
By the Leibniz rule and \eqref{x1-def}, \eqref{uta}, \eqref{nablab} one has
$$ \| \nabla( u \eta^2 ) \|_{L^2_x} \lesssim X_1^{1/2} + \frac{A}{r}$$
and thus
$$ X_3 \lesssim \frac{A^{3/2}}{r} X_1^{3/4} + \frac{A^3}{r^{5/2}}.$$
Now we move on to the forcing term $X_4$.  By Cauchy-Schwarz, we can bound this term by
$$ X_4 \lesssim E_{\eta^4}^{1/2} a(t)$$
where $a(t) := \|f(t)\|_{L^2_x(B(0,R))}$.  Note from \eqref{energy-def} that
\begin{equation}\label{energy-t}
\int_0^T a(t)\ dt \lesssim E(u_0,f,T)^{1/2}.
\end{equation}
Now we turn to the pressure term $X_5$.  From \eqref{pps} we have
$$ X_{5} = \int_{\R^3} \bigO( u (\Delta^{-1} \nabla^2( u u )) \eta^3 \nabla \eta ).$$
We will argue as in the estimation of $X_4$, but we will first need to move the $\eta^3$ weight past the singular integral $\Delta^{-1} \nabla^2$.  We therefore bound $X_5 = X_{5,1} + X_{5,2}$ where
$$ X_{5,1} = \int_{\R^3} \bigO( u (\Delta^{-1} \nabla^2( u u \eta^3 )) \nabla \eta )$$
and
$$ X_{5,2} = \int_{\R^3} \bigO( u [\Delta^{-1} \nabla^2, \eta^3](u u) \nabla \eta ),$$
where $[A,B] := AB-BA$ is the commutator, and $\eta^3$ is interpreted as the multiplication operator $\eta^3: u \mapsto \eta^3 u$.
For $X_{6,1}$, we apply H\"older's inequality and \eqref{nablab} to obtain
$$ X_{5,1} \lesssim \frac{1}{r} \| u \|_{L^2_x} \| \Delta^{-1} \nabla^2( u u \eta^3 ) \|_{L^2_x}.$$
The singular integral $\Delta^{-1} \nabla^2$ is bounded on $L^2$, so it may be discarded; applying H\"older's inequality again we conclude that
$$ X_{5,1} \lesssim \frac{1}{r} \| u \|_{L^2_x}^{3/2} \| u \eta^2 \|_{L^6_x}^{3/2}.$$
This is the same bound \eqref{x4-bound} used to bound $X_3$, and so by repeating the $X_3$ analysis we conclude that
$$ X_{5,1} \lesssim \frac{A^{3/2}}{r} X_1^{3/4} + \frac{A^3}{r^{5/2}}.$$
As for $X_{5,2}$, we observe from direct computation of the integral kernel that when $r=1$, $[\Delta^{-1} \nabla^2,\chi^3]$ is a smoothing operator of infinite order (cf. \cite{kato}), and in particular
$$ \| [\Delta^{-1} \nabla^2,\eta^3] f \|_{L^2_x} \lesssim \|f\|_{L^1_x}$$
in the $r=1$ case.  In the general case, a rescaling argument then gives
$$ \| [\Delta^{-1} \nabla^2,\eta^3] f \|_{L^2_x} \lesssim \frac{1}{r^{3/2}} \|f\|_{L^1_x}.$$
Applying H\"older's inequality and \eqref{uta} we conclude that
$$ X_{5,2} \lesssim \frac{A^3}{r^{5/2}}.$$
Putting all the estimates together, we conclude that
$$ \partial_t E_{\eta^4} \leq -X_1 + O\left( \frac{A^2}{r^2} + \frac{A^{3/2}}{r} X_1^{3/4} + \frac{A^3}{r^{5/2}} +  E_{\eta^4}^{1/2} a(t) \right).$$
By Young's inequality we have
$$ -\frac{1}{2} X_1 + O\left( \frac{A^{3/2}}{r} X_1^{3/4} \right) \lesssim \frac{A^6}{r^4}$$
and 
$$ \frac{A^3}{r^{5/2}} \lesssim \frac{A^2}{r^2} + \frac{A^6}{r^4}$$
and so we obtain
\begin{equation}\label{pater} 
\partial_t E_{\eta^4} + X_1 \lesssim \frac{A^2}{r^2} + \frac{A^6}{r^4} + E_{\eta^4}^{1/2} a(t)
\end{equation}
and hence for almost every time $t$
$$ \partial_t (E_{\eta^4} + E(u_0,f,T))^{1/2} \lesssim E(u_0,f,T)^{-1/2} \left(\frac{A^2}{r^2} + \frac{A^6}{r^4}\right) + a(t).$$
By the fundamental theorem of calculus, \eqref{energy-t} and \eqref{initial-cond}, we conclude that
$$ E_{\eta^4}(t)^{1/2} \lesssim E(u_0,f,T)^{1/2} + E(u_0,f,T)^{-1/2} \left( \frac{A^2}{r^2} + \frac{A^6}{r^4} \right) T$$
for all $t \in [0,T]$ and all sufficiently large $R$; sending $r, R \to \infty$ and using the monotone convergence theorem we conclude that
$$ \| u \|_{L^\infty_t L^2_x([0,T] \times \R^3)} \lesssim E(u_0,f,T)^{1/2}.$$
In particular we have
$$ E_{\eta^4}(t) \lesssim E(u_0,f,T)$$
for all $r,R$; inserting this back into \eqref{pater} and integrating we obtain that
$$ \int_0^T X_1(t)\ dt \lesssim \left(\frac{A^2}{r^2} + \frac{A^6}{r^4}\right) T + E(u_0,f,T).$$
Sending $r,R \to \infty$ and using monotone convergence again, we conclude that
$$ \| \nabla u \|_{L^2_t L^2_x([0,T] \times \R^3)} \lesssim E(u_0,f,T)^{1/2}$$
and Lemma \ref{energy-est} follows.
\end{proof}

We can bootstrap the proof of Lemma \ref{energy-est} as follows.  \emph{A posteriori}, we see that we may take $A \lesssim E(u_0,f,T)^{1/2}$.  If we return to \eqref{pater}, we may then obtain
$$ \partial_t (E_{\eta^4} + e)^{1/2} \lesssim e^{-1/2} \left( \frac{E(u_0,f,T)}{r^2} + \frac{E(u_0,f,T)^3}{r^4} \right) + a(t).$$ 
where $e > 0$ is an arbitrary parameter which we will optimise later.  From the fundamental theorem of calculus we then have
$$ E_{\eta^4}^{1/2} \lesssim E_{\eta^4}(0)^{1/2} + e^{1/2} + e^{-1/2} \left( \frac{E(u_0,f,T)}{r^2} + \frac{E(u_0,f,T)^3}{r^4} \right) T + \|f\|_{L^1_t L^2_x},$$
where the $L^1_t L^2_x$ norm is over $[0,T] \times B(0,R)$;
optimising in $e$, we conclude that
$$ E_{\eta^4}^{1/2} \lesssim E_{\eta^4}(0)^{1/2} + \left(\frac{E(u_0,f,T)}{r^2} + \frac{E(u_0,f,T)^3}{r^4}\right)^{1/2} T^{1/2} + \|f\|_{L^1_t L^2_x}.$$
Inserting this back into \eqref{pater} and integrating, we also conclude that
$$ \int_0^T X_1(t)\ dt \lesssim
\left(E_{\eta^4}(0)^{1/2} + \left(\frac{E(u_0,f,T)}{r^2} + \frac{E(u_0,f,T)^3}{r^4}\right)^{1/2} T^{1/2} + \|f\|_{L^1_t L^2_x}\right)^2.$$

Applying spatial translation invariance \eqref{space-translate} to move the origin from $0$ to an arbitrary point $x_0$, we conclude an energy localisation result:

\begin{theorem}[Local energy estimate]\label{local} Let $(u,p,u_0,f,T)$ be a finite energy almost smooth solution with $f$ divergence-free.  Then for any $x_0 \in \R^3$ and any $0 < r < R/2$, one has
\begin{equation}\label{local-energy}
\begin{split}
\| u \|_{L^\infty_t L^2_x([0,T] \times B(x_0,R-r))} &+ \| \nabla u \|_{L^2_t L^2_x([0,T] \times B(x_0,R-r))}  \\
&\lesssim \| u_0 \|_{L^2_x(B(x_0,R))} + \|f\|_{L^1_t L^2_x([0,T] \times B(x_0,R))} \\
&\quad + \frac{E(u_0,f,T)^{1/2} T^{1/2}}{r} + \frac{E(u_0,f,T)^{3/2} T^{1/2}}{r^2}.
\end{split}
\end{equation}
\end{theorem}

\begin{remark} One can verify that the estimate \eqref{local-energy} is dimensionally consistent.  Indeed, if $L$ denotes a length scale, then $r, R, E(u_0,f)$ have the units of $L$, $T$ has the units of $L^2$, $u$ has the units of $L^{-1}$, and all terms in \eqref{local-energy} have the scaling of $L^{1/2}$.  Note also that the global energy estimate \eqref{energy-est} can be viewed as the limiting case of \eqref{local-energy} when one sends $r, R$ to infinity.
\end{remark}

\begin{remark}\label{inversion}   A minor modification of the proof of Theorem \ref{local} allows one to replace the ball $B(x_0,R)$ by an annulus
$$ B(x_0,R') \backslash B(x_0,R)$$
for some $0 < R < R'$ with $r < (R'-R)/2, R/2$, with smaller ball $B(x_0,R-r)$ being replaced by the smaller annulus
$$ B(x_0,R'-r) \backslash B(x_0,R+r).$$
The proof is essentially the same, except that the cutoff $\eta$ has to be adapted to the two indicated annuli rather than to the two indicated balls; we omit the details.  Sending $R' \to \infty$ using the monotone convergence theorem, we conclude in particular an external local energy estimate of the form
\begin{equation}\label{local-energy-external}
\begin{split}
\| u \|_{L^\infty_t L^2_x([0,T] \times (\R^3 \backslash B(x_0,R+r)))} &+ \| \nabla u \|_{L^2_t L^2_x([0,T] \times (\R^3 \backslash B(x_0,R+r)))}  \\
&\lesssim \| u_0 \|_{L^2_x(\R^3 \backslash B(x_0,R))} + \|f\|_{L^1_t L^2_x([0,T] \times (\R^3 \backslash B(x_0,R)))} \\
&\quad + \frac{E(u_0,f,T)^{1/2} T^{1/2}}{r} + \frac{E(u_0,f,T)^{3/2} T^{1/2}}{r^2}
\end{split}
\end{equation}
whenever $0 < r < R/2$.
\end{remark}

\begin{remark} The hypothesis that $f$ is divergence-free can easily be removed using the symmetry \eqref{forcing}, but then $f$ needs to be replaced by $Pf$ on the right-hand side of \eqref{local-energy}.
\end{remark}

\begin{remark} Theorem \ref{local} can be extended without difficulty to the periodic setting, with the energy $E(u_0,f,T)$ being replaced by the periodic energy
$$ E_L(u_0,f,T) := \frac{1}{2} ( \| u_0 \|_{L^2_x( \R^3/L\Z^3 )} + \|f\|_{L^1_t L^2_x([0,T] \times \R^3/L\Z^3)} )^2,$$
as long as the radius $R$ of the ball is significantly smaller than the period $L$ of the solution e.g. $R < L/100$.  The reason for this is that the analysis used to prove Theorem \ref{local} takes place almost entirely inside the ball $B(x_0,R)$, and so there is almost no distinction between the finite energy and the periodic cases.  The only place where there is any ``leakage'' outside of $B(x_0,R)$ is in the estimation of the term $X_{5,2}$, which involves the non-local commutator $[\Delta^{-1} \nabla^2, \eta^3]$.  However, in the regime $R < L/100$ one easily verifies that the commutator essentially obeys the same sort of kernel bounds in the periodic setting as it does in the non-periodic setting, and so the argument goes through as before.  We omit the details.
\end{remark}

\begin{remark}
Theorem \ref{local} asserts, roughly speaking, that if the energy of the data is small in a large ball, then the energy will remain small in a slightly smaller ball for future times $T$; similarly, \eqref{local-energy-external} asserts that if the energy of the data is small outside a ball, then the energy will remain small outside a slightly larger ball for future times $T$.  Unfortunately, this estimate is not of major use for the purposes of establishing Theorem \ref{main}, because energy is a supercritical quantity for the Navier-Stokes equation, and so smallness of energy (local or global) is not a particularly powerful conclusion.  To achieve this goal, we will need a variant of Theorem \ref{local} in which the energy $\frac{1}{2} \int |u|^2$ is replaced by the \emph{enstrophy} $\frac{1}{2} \int |\omega|^2$, which is subcritical and thus able to control the regularity of solutions effectively.
\end{remark}

\begin{remark} It should be possible to extend Theorem \ref{local} to certain classes of weak solutions, such as mild solutions or Leray-Hopf solutions, perhaps after assuming some additional regularity on the solution $u$.  We will not pursue these matters here.
\end{remark}

\section{Bounded total speed}\label{speed-sec}

Let $(u,p,u_0,f,T)$ be an almost smooth finite energy solution.
Applying the Leray projection $P$ to \eqref{ns} (and using Corollary \ref{as-mild}), we see that
\begin{equation}\label{ns-project}
 \partial_t u = \Delta u + P B(u,u) + P f
\end{equation}
for almost all times $t$, where $B(u,v) = \bigO( \nabla(uv) )$ was defined in \eqref{buj}.
As all expressions here are tempered distributions, we thus have the Duhamel formula \eqref{duhamel-2}, which we rewrite here as
\begin{equation}\label{duhamel-form}
u(t) = e^{t\Delta} u_0 + \int_0^t e^{(t-t')\Delta} (P \bigO( \nabla( u u ) ) + P f)(t')\ dt'.
\end{equation}

One can then insert the \emph{a priori} bounds from Lemma \ref{energy-est} into \eqref{duhamel-form} to obtain further \emph{a priori} bounds on $u$ in terms of the energy $E(u_0,f,T)$ (although, given that \eqref{ul2-en} was supercritical with respect to scaling, any further bounds obtained by this scheme must be similarly supercritical).

Many such bounds of this type already exist in the literature.  For instance\footnote{These bounds are usually localised in both time and space, or are restricted to the periodic setting, and some bounds were only established in the model case $f=0$; some of these bounds also apply to weaker notions of solution than classical solutions.  For the purposes of this exposition we will not detail these technicalities.}, 

\begin{itemize}
\item One can bound the vorticity $\omega := \nabla \times u$ in $L^\infty_t L^1_x$ norm \cite{constantin}, \cite{qian};
\item One can bound $\nabla^2 u$ in $L^{4/3,\infty}_{t,x}$ \cite{constantin}, \cite{lyons};
\item More generally, for any $\alpha \geq 1$, one can bound $\nabla^\alpha u$ in $L^{\frac{4}{\alpha+1},\infty}_t L^{\frac{4}{\alpha+1},\infty}_x$ \cite{vasseur}, \cite{choi};
\item For any $k \geq 0$, one can bound $t^{k} \partial_t^k u$ in $L^2_{t,x}$ \cite{chae};
\item One can bound $\nabla u$ in $L^{1/2}_t L^\infty_x$ \cite{foias};
\item For any $r \geq 0$ and $k \geq 1$, one can bound $D_t^r \nabla_x^s u$ in $L^{\frac{2}{4r+2k-1}}_t L^2_x$ \cite{foias}, \cite{doering}, \cite{duff};
\item For any $1 \leq m \leq \infty$, one can bound $\omega$ in $L^{\frac{2m}{4m-3}}_t L^{2m}_x$ \cite{gibbon};
\item One can bound moments of wave-number like quantities \cite{doering-gibbon}, \cite{chess}.
\end{itemize}

In this section we present another \emph{a priori} bound which will be absolutely crucial for our localisation arguments, and which (somewhat surprisingly) does not appear to be previously in the literature:

\begin{proposition}[Bounded total speed]\label{bounded-speed}  Let $(u,p,u_0,f,T)$ be a finite energy almost smooth solution.  Then we have
\begin{equation}\label{l1x}
 \|u\|_{L^1_t L^\infty_x( [0,T] \times \R^3 )} \lesssim E(u_0,f,T)^{1/2} T^{1/4} + E(u_0,f,T).
\end{equation}
\end{proposition}

We observe that the estimate \eqref{l1x} is dimensionally consistent with respect to the scaling \eqref{scaling}.  Indeed, if $L$ denotes a length scale, then $T$ scales like $L^2$, $u$ scales like $L^{-1}$, and $E_0$ scales like $L$, so both sides of \eqref{l1x} have the scaling of $L$.

Before we prove this proposition rigorously, let us first analyse the equation \eqref{ns-project} heuristically, using Littlewood-Paley projections, to get some feel of what kind of \emph{a priori} estimates one can hope to establish purely from \eqref{ns-project} and \eqref{ul2-en}.  For simplicity we shall assume $f=0$ for the sake of exposition.  We consider a high-frequency component $u_N := P_N u$ of the velocity field $u$ for some $N \gg 1$.  Applying $P_N$ to \eqref{ns-project}, and using the ellipticity of $\Delta$ to adopt the heuristic\footnote{One can informally justify this heuristic by inspecting the symbols of the Fourier multipliers appearing in these expressions.} $P_N \Delta \sim -N^2 P_N$ and $P_N P \nabla \sim N P_N$, we arrive at the heuristic equation
$$ \partial_t u_N = - N^2 u_N + \bigO( N P_N( u^2 ) ).$$
Let us cheat even further and pretend that $P_N(u^2)$ is analogous to $u_N u_N$ (in practice, there will be more terms than this, but let us assume this oversimplification for the sake of discussion).  Then we have
$$ \partial_t u_N = - N^2 u_N + \bigO( N u_N^2 ).$$
Heuristically, this suggests that the high-frequency component $u_N$ should quickly damp itself out into nothingness if $|u_N| \ll N$, but can exhibit nonlinear behaviour when $|u_N| \gg N$.  Thus, as a heuristic, one can pretend that $u_N$ has magnitude $\gg N$ on the regions where it is non-negligible.

This heuristic, coupled with the energy bound \eqref{ul2-en}, already can be used to informally justify many of the known \emph{a priori} bounds on Navier-Stokes solutions.  In particular, projecting \eqref{ul2-en} to the $u_N$ component, one expects that
\begin{equation}\label{un-2}
 \| u_N \|_{L^2_t L^2_x} \lesssim N^{-1}
\end{equation}
(dropping the dependencies of constants on parameters such as $E_0$, and being vague about the spacetime region on which the norms are being evaluated), which by Bernstein's inequality implies that
$$ \| u_N \|_{L^2_t L^\infty_x} \lesssim N^{1/2}.$$
However, with the heuristic that $|u_N| \gg N$ on the support of $u_N$, we expect that
$$ \| u_N \|_{L^1_t L^\infty_x} \lesssim \frac{1}{N} \| u_N \|_{L^2_t L^\infty_x}^2 \lesssim 1;$$
summing in $N$ (and ignoring the logarithmic divergence that results, which can in principle be recovered by using Bessel's inequality to improve upon
\eqref{un-2}), we obtain a non-rigorous derivation of Proposition \ref{bounded-speed}.

We now turn to the formal proof of Proposition \ref{bounded-speed}.  All spacetime norms are understood to be over the region $[0,T] \times \R^3$ (and all spatial norms over $\R^3$) unless otherwise indicated.  We abbreviate $E_0 := E(u_0,f,T)$.  From \eqref{ul2-en} and \eqref{energy-def} we have the bounds
\begin{align}
 \| u \|_{L^\infty_t L^2_x} &\lesssim E_0^{1/2} \label{u-infty-2}\\
 \| \nabla u \|_{L^2_t L^2_x} &\lesssim E_0^{1/2} \label{nosy}\\
 \|u_0\|_{L^2_x} + \|f\|_{L^1_t L^2_x} &\lesssim E_0^{1/2}.\label{Data}
\end{align}

We expand out $u$ using \eqref{duhamel-form}.  For the free term $e^{t\Delta} u_0$, one has by \eqref{dispersive}
$$ \|e^{t\Delta} u_0\|_{L^\infty_x} \lesssim t^{-3/4} \|u_0\|_{L^2_x}$$
for $t \in [0,T]$, so this contribution to \eqref{l1x} is acceptable by \eqref{Data}.  In a similar spirit, we have
$$ \| e^{(t-t')\Delta} Pf(t') \|_{L^\infty_x} \lesssim (t-t')^{-3/4} \|Pf(t')\|_{L^2_x} \lesssim (t-t')^{-3/4} \|f(t')\|_{L^2_x}$$
and so this contribution is also acceptable by the Minkowski and Young inequalities and \eqref{Data}.

It remains to show that
$$ \| \int_0^t e^{(t-t')\Delta} \bigO( P \nabla(u u)(t'))\ dt' \|_{L^1_t L^\infty_x} \lesssim E_0.$$
By Littlewood-Paley decomposition, the triangle inequality, and Minkowski's inequality, we can bound the left-hand side by
$$ \lesssim \sum_{N} \int_0^T \int_0^t \| P_N e^{(t-t')\Delta} \bigO( P \nabla(u u)(t')) \|_{L^\infty_x}\ dt' dt.$$
Using \eqref{energy-decay}, and bounding the first order operator $P \nabla$ by $N$ on the range of $P_N$, we may bound this by
$$ \lesssim \sum_{N} \int_0^T \int_0^t \exp(-c (t-t') N^2) N \| P_N \bigO(u u)(t') \|_{L^\infty_x}\ dt' dt$$
for some $c>0$; interchanging integrals and evaluating the $t$ integral, this becomes
\begin{equation}\label{nose}
 \lesssim \sum_{N} \int_0^T N^{-1} \| P_N \bigO(u u)(t') \|_{L^\infty_x}\ dt'.
\end{equation}
We now apply the Littlewood-Paley trichotomy (see Section \ref{notation-sec}) and symmetry to write
$$ P_N \bigO(u u) = \sum_{N_1 \sim N} \sum_{N_2 \lesssim N} P_N \bigO( u_{N_1} u_{N_2} ) + \sum_{N_1 \gtrsim N} \sum_{N_2 \sim N_1} P_N \bigO( u_{N_1} u_{N_2} )$$
where $u_N := P_N u$.  For $N_1,N_2$ in the first sum, we use Bernstein's inequality to estimate
\begin{align*}
\| P_N \bigO( u_{N_1} u_{N_2} ) \|_{L^\infty_x} &\lesssim
\| u_{N_1} \|_{L^\infty_x}  \| u_{N_2} \|_{L^\infty_x} \\
&\lesssim N_1^{3/2} \| u_{N_1} \|_{L^2_x} N_2^{3/2} \| u_{N_2} \|_{L^2_x}  \\
&\lesssim N (N_2/N_1)^{1/2} \| \nabla u_{N_1} \|_{L^2_x} \| \nabla u_{N_2} \|_{L^2_x}.
\end{align*}
For $N_1,N_2$ in the second sum, we use Bernstein's inequality in a slightly different way to estimate
\begin{align*}
\| P_N \bigO( u_{N_1} u_{N_2} ) \|_{L^\infty_x} &\lesssim
N^{3} \| \bigO( u_{N_1} u_{N_2} ) \|_{L^1_x} \\
&\lesssim N^3 \| u_{N_1} \|_{L^2_x}  \| u_{N_2} \|_{L^2_x} \\
&\lesssim N (N/N_1)^2 \| \nabla u_{N_1} \|_{L^2_x} \| \nabla u_{N_2} \|_{L^2_x}.
\end{align*}
Applying these bounds, we can estimate \eqref{nose} by
\begin{align*}
&\lesssim \sum_{N} \sum_{N_1 \sim N} \sum_{N_2 \lesssim N} (N_2/N_1)^{1/2} \int_0^T \| \nabla u_{N_1}(t') \|_{L^2_x} \| \nabla u_{N_2}(t') \|_{L^2_x}\ dt' \\
&\quad + \sum_{N} \sum_{N_1 \gtrsim N} \sum_{N_2 \sim N_1} (N/N_1)^2 \int_0^T \| \nabla u_{N_1}(t') \|_{L^2_x} \| \nabla u_{N_2}(t') \|_{L^2_x}\ dt'.
\end{align*}
Performing the $N$ summation first, then using Cauchy-Schwarz, one can bound this by
$$ \lesssim \sum_{N_1 \gtrsim 1} \sum_{N_2 \lesssim N_1} (N_2/N_1)^{1/2} a_{N_1} a_{N_2} + \sum_{N_1 \gtrsim 1} \sum_{N_2 \sim N_1} a_{N_1} a_{N_2},$$
where
$$ a_N := \| \nabla u_N \|_{L^2_t L^2_x}.$$
But from \eqref{nosy} and Bessel's inequality (or the Plancherel theorem) one has
$$ \sum_N a_N^2 \lesssim E_0$$
and the claim \eqref{l1x} then follows from Schur's test (or Young's inequality).

\begin{remark}  An inspection of the above argument reveals that the $L^\infty_x$ norm in \eqref{l1x} can be strengthened to a Besov norm $(\dot B^{0,\infty}_1)_x$, defined by
$$ \|u\|_{(\dot B^{0,\infty}_1)_x} := \sum_{N} \|P_N u \|_{L^\infty_x}.$$
\end{remark}

\begin{remark}  An inspection of the proof of Proposition \ref{bounded-speed} reveals that the time-dependent factor $T^{1/4}$ on the right-hand side of Proposition \ref{bounded-speed} was only necessary in order to bound the linear components
$$ e^{t\Delta} u_0 + \int_0^t e^{(t-t')\Delta} (P f)(t')\ dt'$$
of the Duhamel formula \eqref{duhamel-form}.  If one had some other means to bound these components in $L^1_t L^\infty_x$ by a bound independent of $T$ (for instance, if one had some further control on the decay of $u_0$ and $f$, such as $L^1_x$ and $L^1_t L^1_x$ bounds), then this would lead to a similarly time-independent bound in Proposition \ref{bounded-speed}, which could be useful for analysis of the long-time asymptotics of Navier-Stokes solutions (which is not our primary concern here).
\end{remark}

\begin{remark} It is worth comparing the (supercritical) control given by Proposition \ref{bounded-speed} with the well-known (critical) Prodi-Serrin-Ladyzhenskaya regularity condition \cite{prodi, serrin, lady, fabes, struwe}, a special case of which (roughly speaking) asserts that smooth solutions to the Navier-Stokes system can be continued as long as $u$ is bounded in $L^2_t L^\infty_x$, and the equally well known (and also critical) regularity condition of Beale, Kato, and Majda \cite{bkm}, which asserts that smooth solutions can be continued as long as the \emph{vorticity}
\begin{equation}\label{vorticity-def}
\omega := \nabla \times u
\end{equation}
stays bounded in $L^1_t L^\infty_x$.
\end{remark}

\begin{remark} As pointed out by the anonymous referee, one can also obtain $L^1_t L^\infty_x$ bounds on the velocity field $u$ by a Gagliardo-Nirenberg type interpolation between the $L^{1/2}_t L^\infty_x$ bound on $\nabla u$ from \cite{foias} with the $L^2_t L^6_x$ bound on $u$ arising from the energy inequality and Sobolev embedding.
\end{remark}

Although we will not need it in this paper, Proposition \ref{bounded-speed} when combined with the Picard well-posedness theorem for ODE yields the following immediate corollary, which may be of use in future applications:

\begin{corollary}[Existence of material coordinates]  Let $(u,p,u_0,f,T)$ be a finite energy smooth solution.  Then there exists a unique smooth map $\Phi: [0,T] \times \R^3 \to \R^3$ such that
$$ \Phi(0,x) = x$$
for all $x \in \R^3$, and
$$ \partial_t \Phi(t,x) = u( \Phi(t,x) )$$
for all $(t,x) \in [0,T] \times \R^3$, and furthermore $\Phi(t): \R^3 \to \R^3$ is a diffeomorphism for all $t \in [0,T]$.  Finally, one has
$$ |\Phi(t,x) - x| \lesssim E(u_0,f,T)^{1/2} T^{1/4} + E(u_0,f,T)$$
for all $(t,x) \in [0,T] \times \R^3$.
\end{corollary}

\begin{remark}\label{lota}  One can extend the results in this section to the periodic case, as long as one assumes normalised pressure and imposes the additional condition $T \leq L^2$, which roughly speaking ensures that the periodic heat kernel behaves enough like its non-periodic counterpart that estimates such as \eqref{dispersive} are maintained; we omit the details.  (Without normalised pressure, the Galilean invariance \eqref{galilean} shows that one cannot hope to bound the $L^1_t L^\infty_x$ norm of $u$ by the initial data, and even energy estimates do not work any more.)    When the inequality $T \leq L^2$ fails, one can still obtain estimates (but with weaker bounds) by using the crude observation that a solution which is periodic with period $L$, is also periodic with period $kL$ for any positive integer $k$, and choosing $k$ to be the first integer so that $T \leq (kL)^2$.
\end{remark}

\section{Enstrophy localisation}

The purpose of this section is to establish a subcritical analogue of Theorem \ref{local}, in which the energy $\frac{1}{2} \int |u|^2$ is replaced by the enstrophy $\frac{1}{2} \int |\omega|^2$.  Because the latter quantity is not conserved, we will need a smallness condition on the initial local enstrophy; however, the initial \emph{global} enstrophy is allowed to be arbitrarily large (or even infinite).

\begin{theorem}[Enstrophy localisation]\label{enstrophy-loc} Let $(u,p,u_0,f,T)$ be a finite energy almost smooth solution.  Let $B(x_0,R)$ be a ball such that
\begin{equation}\label{eeta}
\| \omega_0 \|_{L^2_x(B(x_0,R))} + \| \nabla \times f \|_{L^1_t L^2_x([0,T] \times B(x_0,R))} \leq \delta
\end{equation}
for some $\delta > 0$, where $\omega_0 := \nabla \times u_0$ is the initial vorticity.  Assume the smallness condition
\begin{equation}\label{delta-4} 
\delta^4 T + \delta^5 E(u_0,f,T)^{1/2} T \leq c
\end{equation}
for some sufficiently small absolute constant $c > 0$ (independent of all parameters).  Let $0 < r < R/2$ be a quantity such that
\begin{equation}\label{r-large}
 r > C (E(u_0,f,T) + E(u_0,f,T)^{1/2} T^{1/4} + \delta^{-2})
\end{equation}
for some sufficiently large absolute constant $C$ (again independent of all parameters).  Then
$$
\| \omega \|_{L^\infty_x L^2_x([0,T] \times B(x_0,R-r))} + \| \nabla \omega \|_{L^2_t L^2_x([0,T] \times B(x_0,R-r))} 
\lesssim \delta.$$
\end{theorem}

\begin{remark} Once again, this theorem is dimensionally consistent (and so one could use \eqref{scaling} to normalise one of the non-dimensionless parameters above to equal $1$ if desired).  Indeed, if $L$ is a unit of length, then $u$ has the units of $L^{-1}$, $\omega$ has the units of $L^{-2}$, $E(u_0,f,T), r, R$ have the units of $L$, $T$ has the units of $L^2$, and $\delta$ has the units of $L^{-1/2}$ (so in particular $\delta^4 T$ and $\delta^5 E(u_0,f,T)^{1/2} T$ are dimensionless).
\end{remark}

\begin{remark}  The smallness of $\delta^4 T$ also comes up, not coincidentally, as a condition in the local well-posedness theory for the Navier-Stokes at the level of $H^1$; see \eqref{D4}.  The smallness of $\delta^5 E(u_0,f,T)^{1/2} T$ is a more artificial condition, and it is possible that a more careful argument would eliminate it, but we will not need to do so for our applications.  For future reference, it will be important to note the fact that $\delta$ is permitted to be large in the above theorem, so long as the time $T$ is small.
\end{remark}

\begin{remark} A variant to Theorem \ref{enstrophy-loc} can also be deduced from the result\footnote{We thank the anonymous referee for this observation.} in \cite[Theorem D]{ckn}.  Here, instead of assuming a small $L^2$ condition on the enstrophy, one needs to assume smallness of quantities such as $\int_{\R^3} \frac{|u_0(x)|^2}{|x-x_0|}\ dx$ for all sufficiently large $x_0$, and then regularity results are obtained outside of a sufficiently large ball in spacetime.
\end{remark}

We now prove the theorem.  Let $(u,p,u_0,f,T)$, $B(x_0,R)$, $\delta, r$ be as in the theorem.  We may use spatial translation symmetry \eqref{space-translate} to normalise $x_0=0$.  We assume $c >0$ is a sufficiently small absolute constant, and then assume $C > 0$ is a sufficiently large constant (depending on $c$).  We abbreviate $E_0 := E(u_0,f,T)$.

In principle, this is a subcritical problem, because the local enstrophy $\frac{1}{2} \int_{B(x_0,R)} |\omega|^2$ (or regularised versions thereof) is subcritical with respect to scaling \eqref{scaling}.  As such, standard energy methods should in principle suffice to keep the enstrophy small for small times (using the smallness condition \eqref{delta-4}, of course).  The main difficulty is that the local enstrophy is not fully \emph{coercive}: it controls $\omega$ (and, to a lesser extent, $u$) inside $B(x_0,R)$, but not outside $B(x_0,R)$; while we do have some global control of the solution thanks to the energy estimate (Lemma \ref{energy-est}), this is supercritical and thus needs to be used sparingly.  We will therefore expend a fair amount of effort to prevent our estimates from ``leaking'' outside $B(x_0,R)$; in particular, one has to avoid the use of non-local singular integrals (such as the Leray projection or the Biot-Savart law) and work instead with more local techniques such as integration by parts.  This will inevitably lead to some factors that blow up as one approaches the boundary of $B(x_0,R)$ (actually, for technical reasons, we will be using a slightly smaller ball $B(x_0,R'(t))$ as our domain).  It turns out, however, that thanks to a moderate amount of harmonic analysis, these boundary factors can (barely) be controlled if one chooses exactly the right type of weight function to define the local enstrophy (it has to be Lipschitz continuous, but no better).

We turn to the details.  We will need an auxiliary initial radius $R' = R'(0)$ in the interval $[R-r/4,R]$
which we will choose later (by a pigeonholing argument).  Given this $R'$, we then define a time-dependent radius function
$$ R'(t) := R' - \frac{1}{c} \int_0^t \|u(s)\|_{L^\infty_x(\R^3)}\ ds.$$
From Proposition \ref{bounded-speed} one has
$$ R'(t) \geq R' - O_c( E_0 + E_0^{1/2} T^{1/4} )$$
and thus (by \eqref{r-large}) one has
$$ R'(t) \geq R - r/2$$
if the constant $C$ in \eqref{r-large} is sufficiently large depending on $c$.  The reason we introduce this rapidly shrinking radius is that we intend to ``outrun'' all difficulties caused by the transport component of the Navier-Stokes equation when we deploy the energy method.  Note that the bounded total speed property (Proposition \ref{bounded-speed}) prevents us from running the radius down to zero when we do this.

We introduce a time-varying Lipschitz continuous cutoff function
$$ \eta(t,x) = \min( \max(0, c^{-0.1} \delta^2(R'(t)-|x|)), 1).$$
This function is supported on the ball $B(0, R'(t))$ and equals one on $B(0,R'(t)-c^{0.1} \delta^{-2})$, and is radially decreasing; in particular, from \eqref{r-large}, we see that $\eta$ is supported on $B(0,R)$ and equals $1$ on $B(0,R-r)$ if $C$ is large enough.  As $t$ increases, this cutoff shrinks at speed $\frac{1}{c} \|u(t)\|_{L^\infty_x(\R^3)}$, leading to the useful pointwise estimate
\begin{equation}\label{tax}
 \partial_t \eta(t,x) \leq - \frac{1}{c} \| u(t) \|_{L^\infty_x(\R^3)} |\nabla_x \eta(t,x)|
\end{equation}
which we will use later in this argument to control transport-like terms in the energy estimate (or more precisely, the enstrophy estimate).

\begin{remark}
It will be important that $\eta$ is Lipschitz continuous but no better; Lipschitz is the minimal regularity for which one can still control the heat flux term (see $Y_3$ below), but is also the maximal regularity for which there is enough coercivity to control the nonlinear term (see $Y_6$ below).  The argument is in fact remarkably delicate, necessitating a careful application of harmonic analysis techniques (and in particular, a Whitney decomposition of the ball).
\end{remark}

We introduce the localised enstrophy
\begin{equation}\label{wdef}
 W(t) := \frac{1}{2} \int_{\R^3} |\omega(t,x)|^2 \eta(t,x)\ dx.
\end{equation}
From the hypothesis \eqref{eeta} one has the initial condition
\begin{equation}\label{w-init}
W(0) \lesssim \delta^2
\end{equation}
and to obtain the proposition, it will suffice to show that
\begin{equation}\label{claim}
 W(t) \lesssim_c \delta^2
\end{equation}
for all $t \in [0,T]$.

As $u$ is almost smooth, $W$ is $C^1_t$.  As in Section \ref{energy-sec}, we will compute the derivative $\partial_t W$.  We first take the curl of \eqref{ns} to obtain the well-known \emph{vorticity equation}
\begin{equation}\label{vorticity-eq}
\partial_t \omega + (u \cdot \nabla) \omega = \Delta \omega + \bigO( \omega \nabla u ) + \nabla \times f.
\end{equation}
This leads to the \emph{enstrophy equation}
$$
\partial_t \frac{1}{2} |\omega|^2 + (u \cdot \nabla) \frac{1}{2} |\omega|^2 = \Delta(\frac{1}{2} |\omega|^2) - |\nabla \omega|^2 + \bigO( \omega \omega \nabla u ) + \omega \cdot (\nabla \times f ).$$
All terms in this equation are smooth.  Integrating this equation against the Lipschitz, compactly supported $\eta$ and integrating by parts as in Section \ref{energy-sec} (interpreting derivatives of $\eta$ in a distributional sense), we conclude that
\begin{equation}\label{yyy}
\partial_t W = -Y_1 - Y_2 + Y_3 + Y_4 + Y_5 + Y_6
\end{equation}
where $Y_1$ is the dissipation term
$$
Y_1 := \int_{\R^3} |\nabla \omega|^2 \eta,
$$
$Y_2$ is the recession term
$$ 
Y_2 := -\frac{1}{2} \int_{\R^3} |\omega|^2 \partial_t \eta,
$$
$Y_3$ is the heat flux term
$$
Y_3 := \frac{1}{2} \int_{\R^3} |\omega|^2 \Delta \eta,$$
$Y_4$ is the transport term
$$ 
Y_4 := \frac{1}{2} \int_{\R^3} |\omega|^2 u \cdot \nabla \eta,$$
$Y_5$ is the forcing term
$$ 
Y_5 := \int_{\R^3} \omega \cdot (\nabla \times f ) \eta,
$$
and $Y_6$ is the nonlinear term
$$ 
Y_6 := \int_{\R^3} \bigO( \omega \omega \nabla u ) \eta.
$$
The term $Y_1$ is non-negative, and will be needed to control some of the other terms.  The term $Y_2$ is also non-negative; by \eqref{tax} we see that
\begin{equation}\label{tax-cut}
\int_{\R^3} |\omega|^2 |\nabla \eta| \lesssim c \| u(t) \|_{L^\infty_x(\R^3)} Y_2.
\end{equation}
We skip the heat flux term $Y_3$ for now and use \eqref{tax-cut} to bound the transport term $Y_4$ by
\begin{equation}\label{diaper}
|Y_4| \lesssim c Y_2.
\end{equation}
Now we turn to the forcing term $Y_5$.  By Cauchy-Schwarz and \eqref{wdef} we have
$$ |Y_5| \lesssim W^{1/2} a(t)$$
where 
$$ a(t) := \| \nabla \times f \|_{L^2_x(B(0,R))}.$$
Note from \eqref{eeta} that
\begin{equation}\label{dt-bash}
\int_0^T a(t)\ dt \lesssim \delta.
\end{equation}

We return now to the heat flux term $Y_3$.  Computing the distributional Laplacian\footnote{Alternatively, if one wishes to avoid distributions, one can regularise $\eta$ by a small epsilon parameter to become smooth, compute the Laplacian of the regularised term, and take limits as epsilon goes to zero.  One can also rescale either $R$ or $\delta$ (but not both) to equal $1$ to simplify the computations.} of $\eta$ in polar coordinates, we see that
$$ Y_3 \lesssim b(t)$$
where $b(t) = b_{R'}(t)$ is the quantity
$$ b(t) := c^{-0.1} \delta^{2} R^2 \int_{S^2} |\omega( t, R'(t) \alpha )|^2\ d\alpha + c^{-0.2} \delta^{4} \int_{R'(t)-c^{0.1} \delta^{-2} \leq |x| \leq R'(t)} |\omega(t,x)|^2\ dx
$$
and $d\alpha$ is surface measure on the unit sphere $S^2$.  (Note that while $\Delta \eta$ also has a component on the sphere $|x| = R'(t) - c^{0.1} \delta^{-2}$, this component is negative and thus can be discarded.)

To control $b(t)$, we take advantage of the freedom to choose $R'$.  From Fubini's theorem and a change of variables, we see that
$$ \int_{R-r/4}^R \int_0^T b_{R'}(t)\ dt dR'
\lesssim c^{-0.1} \delta^2 \int_0^T \int_{\R^3} |\omega(t,x)|^2\ dx.$$
From Lemma \ref{energy-est}, the right-hand side is $O( \delta^2 E_0 / c^{0.1} )$.  Thus, by the pigeonhole principle, we may select a radius $R'$ such that
$$ \int_0^T b(t)\ dt \lesssim \frac{\delta^2 E_0}{c^{0.1} r},$$
and in particular by \eqref{r-large}
\begin{equation}\label{b-bound}
 \int_0^T b(t)\ dt \lesssim \delta^2
\end{equation}
if $C$ is large enough.

Henceforth we fix $R'$ so that \eqref{b-bound} holds.  We now turn to the most difficult term, namely the nonlinear term $Y_6$.  Morally speaking, the $\nabla u$ term in $Y_6$ has the ``same strength'' as $\omega$, and so $Y_6$ is heuristically as strong as
$$ \int_{\R^3} \bigO( \omega^3 ) \eta.$$
A standard Whitney decomposition of the support of $\eta$, followed by rescaled versions of the Sobolev inequality, bounds this latter expression by
$$ O( (\int_{\R^3} |\omega|^2 \eta)^{1/2} (\int_{\R^3} |\nabla \omega|^2 \eta) ).$$
If we could similarly bound $Y_6$ by this expression by an analogous argument, this would greatly simplify the argument below.  Unfortunately, the relationship between $\nabla u$ and $\omega$ is rather delicate (especially when working relative to the weight $\eta$), and we have to perform a much more involved analysis (though still ultimately one which is inspired by the above argument).

We turn to the details. We fix $t$ and work in the domain
$$ \Omega := B(0, R'(t)).$$
We apply a Whitney-type decomposition, covering $\Omega$ by a boundedly overlapping collection of balls $B_i = B(x_i,r_i)$ with radius
$$ r_i := \frac{1}{100} \min( \operatorname{dist}( x_i, \partial \Omega ), c^{0.1}/\delta^2 ).$$
In particular, we have 
\begin{equation}\label{whitney}
\eta \sim c^{-0.1} \delta^2 r_i
\end{equation}
on $B(x_i,10r_i)$.  We can then bound
$$ |Y_6| \lesssim c^{-0.1} \delta^2 \sum_i r_i \int_{B_i} |\omega|^2 |\nabla u|.$$
The first step is to convert $\nabla u$ into an expression that only involves $\omega$ (modulo lower order terms), while staying inside the domain $\Omega$.  To do this, we first observe from the divergence-free nature of $u$ that
$$ \Delta u = \nabla \times \nabla \times u = \nabla \times \omega.$$
Let $\psi_i$ be a smooth cutoff to the ball $3B_i := B(x_i,3r_i)$ that equals $1$ on $2B_i := B(x_i,2r_i)$.  On $2B_i$, we thus have the local Biot-Savart law
$$ u = \bigO( \Delta^{-1} \nabla (\psi_i \omega) ) + v$$
where $v$ is harmonic on $2B_i$.  In particular, from Sobolev embedding one has
$$ \|v\|_{L^2_x(2B_i)} \lesssim \| \psi_i \omega \|_{L^{6/5}_x(\R^3)} + \|u\|_{L^2_x(2B_i)}.$$
From H\"older's inequality one has
$$ \| \psi_i \omega \|_{L^{6/5}_x(\R^3)} \lesssim r_i \| \omega \|_{L^2_x(2B_i)} $$
while from the mean value principle for harmonic functions one has
$$ \|\nabla v\|_{L^\infty_x(B_i)} \lesssim r_i^{-5/2} \|v\|_{L^2_x(2B_i)}.$$
We conclude that
$$ \|\nabla v\|_{L^\infty_x(B_i)} \lesssim r_i^{-3/2} \| \omega \|_{L^2_x(2B_i)} + r_i^{-5/2} \|u\|_{L^2_x(2B_i)}$$
and we thus have the pointwise estimate
$$ |\nabla u| \lesssim |\nabla \Delta^{-1} \nabla (\psi_i \omega)| + r_i^{-3/2} \| \omega \|_{L^2_x(2B_i)} + r_i^{-5/2} \|u\|_{L^2_x(2B_i)}$$
on $B_i$.  We can thus bound $|Y_6| \leq Y_{6,1} + Y_{6,2}$, where
\begin{equation}\label{y6-a}
Y_{6,1} \lesssim c^{-0.1} \delta^2 \sum_i r_i \int_{B_i} |\omega|^2 F_i
\end{equation}
and
$$ F_i := |\nabla \Delta^{-1} \nabla (\psi_i \omega)| + r_i^{-3/2} \| \omega \|_{L^2_x(2B_i)} $$
and
$$ Y_{6,2} \lesssim c^{-0.1} \delta^2 \sum_i r_i^{-3/2} \|u\|_{L^2_x(2B_i)} \int_{B_i} |\omega|^2.$$
Let us first deal with $Y_{6,2}$, which is the only term that is not locally controlled by the vorticity alone.  If the ball $B_i$ is contained in the annnular region 
$$ \{x \in \Omega: |x| \geq R'(t) - c^{0.1} \delta^{-2} \},$$
which is the region where $\eta$ is not constant, then we use H\"older to bound
$$ r^{-3/2} \|u\|_{L^2_x(2B_i)} \lesssim \|u\|_{L^\infty_x(\R^3)}$$
and observe that $c^{-0.1} \delta^2 = |\nabla \eta|$ on $B_i$.  Thus, by \eqref{tax-cut}, the contribution of this term to $Y_{6,2}$ is $O( c^{0.9} Y_2 )$.  If instead the ball $B_i$ intersects the ball $B(0,R'(t) - c^{0.1}\delta^{-2})$, then $r_i \sim c^{0.1}\delta^{-2}$ and $\eta \sim 1$ on $B_i$, and we use Lemma \ref{energy-est} to bound $r_i^{-3/2} \|u\|_{L^2_x(2B_i)} \lesssim c^{-0.15} \delta^3 E_0^{1/2}$, and then by \eqref{wdef}, \eqref{delta-4} the contribution of this case is $O(c^{-0.25} \delta^5 E_0^{1/2} W) = O( c^{0.75} W / T)$, thus
$$ Y_{6,2} \lesssim c^{0.9} Y_2 + c^{0.75} W / T.$$

Now we turn to $Y_{6,1}$.  From Plancherel's theorem we have
$$ \|\nabla \Delta^{-1} \nabla (\psi_i \omega)\|_{L^2_x(\R^3)} \lesssim \| \psi_i \omega \|_{L^2_x(\R^3)} \lesssim \|\omega \|_{L^2_x(2B_i)}$$
and thus
$$ \|F_i\|_{L^2_x(B_i)} \lesssim \|\omega \|_{L^2_x(2B_i)}.$$
From H\"older's inequality we thus have
$$ Y_{6,1} \lesssim c^{-0.1} \delta^2 \sum_i r_i^{3/2} \| \omega \|_{L^6_x(B_i)}^2 \|\omega \|_{L^2_x(2B_i)}.$$
To deal with this, we let $w_i$ denote the averages
$$ w_i := \left(\frac{1}{|3B_i|} \int_{3B_i} |\omega|^2\right)^{1/2},$$
then
$$ \|\omega \|_{L^2_x(2B_i)} \lesssim r_i^{3/2} w_i.$$
Also, from the Sobolev inequality one has
\begin{align*}
\| \omega \|_{L^6_x(B_i)} &\lesssim  \| \omega \psi_i \|_{L^6_x(\R^3)} \\
&\lesssim \| \nabla (\omega \psi_i) \|_{L^2_x(\R^3)} \\
&\lesssim \| \nabla \omega \|_{L^2_x(3B_i)} + r_i^{-1} \| \omega \|_{L^2_x(3B_i)} \\
 &\lesssim \| \nabla \omega \|_{L^2_x(3B_i)} + r_i^{1/2} w_i
\end{align*}
and thus
\begin{equation}\label{y61}
 Y_{6,1} \lesssim c^{-0.1} \delta^2 \sum_i r_i^3 w_i \| \nabla \omega \|_{L^2_x(3B_i)}^2 + c^{-0.1} \delta^2 \sum_i r_i^4 w_i^3.
\end{equation}
To deal with the first term of \eqref{y61}, observe from \eqref{wdef} and \eqref{whitney} that
\begin{equation}\label{w-bong}
 \sum_i r_i^4 w_i^2 \lesssim c^{0.1} \delta^{-2} W
\end{equation}
and in particular
\begin{equation}\label{w-bang}
w_i \lesssim c^{0.05} \delta^{-1} W^{1/2} r_i^{-2}
\end{equation}
for all $i$. We may thus bound
$$ c^{-0.1} \delta^2 \sum_i r_i^3 w_i \| \nabla \omega \|_{L^2_x(3B_i)}^2 \lesssim c^{-0.05} \delta W^{1/2} \sum_i r_i \| \nabla \omega \|_{L^2_x(3B_i)}^2,$$
which by \eqref{whitney} and the bounded overlap of the $B_i$ is
$$ \lesssim c^{0.05} \delta^{-1} W^{1/2} \int_\Omega |\nabla \omega|^2 \eta \lesssim c^{0.05} \delta^{-1} W^{1/2} Y_1.$$
The second term of \eqref{y61}, $c^{-0.1} \delta^2 \sum_i r_i^4 w_i^3$, is trickier to handle.  Call a ball ``large'' if its radius is at least $10^{-4} c^{-0.1} \delta^{-2}$ (say), and ``small'' otherwise. To deal with the small balls we use the Poincar\'e inequality.  From this inequality, we see in particular that
$$ \left|(\frac{1}{|3B_i|} \int_{3B_i} |\omega|^2)^{1/2} - (\frac{1}{|3B_j|} \int_{3B_j} |\omega|^2)^{1/2}\right| \lesssim (r_i^{-1} \int_{10B_i} |\nabla \omega|^2)^{1/2}$$
whenever $B_i, B_j$ intersect.  (Indeed, the Poincar\'e inequality implies that both terms in the left-hand side are within $O((r_i^{-1} \int_{10B_i} |\nabla \omega|^2)^{1/2})$ of $|\frac{1}{|10B_i|} \int_{10B_i} \omega|$.)  In other words, we have
\begin{equation}\label{wij}
 |w_i - w_j| \lesssim r_i^{-1/2} (\int_{10B_i} |\nabla \omega|^2)^{1/2}
\end{equation}
whenever $B_i, B_j$ intersect.

Now for any small ball $B_i$, we may assign a ``parent'' ball $B_{p(i)}$ which touches the ball but has radius at least $1.001$ (say) as large as that of $B_i$.  We may iterate this until we reach a large ball $B_{a(i)}$, and write
$$ w_i \leq w_{a(i)} +  \sum_{k \geq 0} |w_{p^k(i)} - w_{p^{k+1}(i)}|$$
where the sum is over all $k$ for which $p^{k+1}(i)$ is well-defined; note that this inequality also holds for large balls if we set $a(i)=i$.  Taking cubes and using H\"older's inequality, we obtain
$$ w_i^3 \lesssim w_{a(i)}^3 +  \sum_{k \geq 0} (1+k)^{10} |w_{p^k(i)} - w_{p^{k+1}(i)}|^3$$
and so we can bound $c^{-0.1} \delta^2 \sum_i r_i^4 w_i^3$ by
$$ \lesssim c^{-0.1} \delta^2 \sum_i r_i^4 w_{a(i)}^3 + c^{-0.1} \delta^2 \sum_{k \geq 0} (1+k)^{10} \sum_i r_i^4 |w_{p^k(i)} - w_{p^{k+1}(i)}|^3.$$
If one fixes a large ball $B_j$, one easily checks that $\sum_{i: a(i)=j} r_i^4 \lesssim r_j^4$, and thus
$$ c^{-0.1} \delta^2 \sum_i r_i^4 w_{a(i)}^3 \lesssim c^{-0.1} \delta^2 \sum_{j: r_j > 10^{-4} c^{0.1} \delta^{-2}} r_j^4 w_j^3;$$
applying \eqref{w-bang} and \eqref{w-bong} we thus have
$$ c^{-0.1} \delta^2 \sum_i r_i^4 w_{a(i)}^3 \lesssim c^{-0.25} \delta^5 W^{1/2} \sum_{j} r_j^4 w_j^2 \lesssim c^{-0.15} \delta^3 W^{3/2}.$$
Similarly, if one fixes a small ball $B_j$, one verifies that
$$ \sum_{k \geq 0} (1+k)^{10} \sum_{i: p^k(i) = j} r_i^4 \lesssim r_j^4$$
and thus
$$ c^{-0.1} \delta^2 \sum_{k \geq 0} (1+k)^{10} \sum_i r_i^4 |w_{p^k(i)} - w_{p^{k+1}(i)}|^3 \lesssim c^{-0.1} \delta^2 \sum_{j: r_j \leq 10^{-4} c^{0.1} \delta^{-2}} r_j^4 |w_j - w_{p(j)}|^3.$$
From \eqref{w-bang} (once) and \eqref{wij} (twice) one has
$$ |w_j - w_{p(j)}|^3 \lesssim c^{0.05} \delta^{-1} W^{1/2} r_j^{-3} \int_{10B_j} |\nabla \omega|^2$$
and so we may bound the preceding expression by
$$ \lesssim c^{-0.05} \delta W^{1/2} \sum_j r_j \int_{10B_j} |\nabla \omega|^2$$
which by \eqref{whitney} and the bounded overlap of the $B_j$ can be bounded by
$$ \lesssim c^{0.05} \delta^{-1} W^{1/2} \int_\Omega |\nabla \omega|^2 \eta \lesssim c^{0.05} \delta^{-1} W^{1/2} Y_1.$$
Putting the $Y_{6,1}$ bounds together, we conclude that
$$ Y_{6,1} \lesssim c^{-0.15} \delta^3 W^{3/2} + c^{0.05} \delta^{-1} W^{1/2} Y_1;$$
collecting the bounds for $Y_1, \ldots, Y_6$ we thus have
$$ \partial_t W \leq -Y_1 + O(c^{0.05} \delta^{-1} W^{1/2} Y_1 + c^{-0.15} \delta^3 W^{3/2} + c^{0.75} W/T + a(t) W^{1/2} + b(t) ).$$
To solve this differential inequality we use the continuity method.  Suppose that $0 \leq T' \leq T$ is a time for which
\begin{equation}\label{woot}
 \sup_{t \in [0,T']} W(t) \leq c^{-0.01} \delta^2.
 \end{equation}
Then, if $c$ is small enough, we can absorb the $O(c^{0.05} \delta^{-1} W^{1/2} Y_1)$ term by the $-Y_1$ term, and can also use this bound and \eqref{delta-4} to obtain
$$ c^{-0.15} \delta^3 W^{3/2} \lesssim c^{-0.155} \delta^4 W \lesssim c^{0.75} W/T$$
and
$$ a(t) W^{1/2} \lesssim c^{-0.005} \delta a(t).$$ 
We thus have
$$ \partial_t W \lesssim c^{0.75} W/T + c^{-0.005} \delta a(t) + b(t).$$
From Gronwall's inequality and \eqref{w-init}, \eqref{dt-bash}, \eqref{b-bound} we thus have
$$ \sup_{t \in [0,T']} W(t) \lesssim c^{-0.005} \delta^2.$$
For $c$ a small enough absolute constant, this is (slightly) better than the hypothesis \eqref{woot}, and so from the continuity method (and \eqref{w-init}) we conclude that
$$ \sup_{t \in [0,T]} W(t) \lesssim c^{-0.005} \delta^2.$$
and the claim \eqref{claim} follows.  The proof of Theorem \ref{enstrophy-loc} is now complete.

\begin{remark}\label{inversion-again}  As with Remark \ref{inversion}, we may adapt the proof of Theorem \ref{enstrophy-loc} to an annulus, replacing the ball $B(x_0,R)$ with an annulus $B(x_0,R') \backslash B(x_0,R)$ for some $0 <R < R'$ with $0 < r < R/2, (R'-R)/2$, and replacing the smaller ball $B(x_0,R-r)$ with the smaller annulus $B(x_0,R'-r) \backslash B(x_0,R+r)$.  To do this, one has to replace the cutoff $\eta$ (which was shrinking inside the ball $B(x_0,R)$ towards $B(x_0,R-r)$) with a slightly more complicated cutoff (which is shrinking inside the annulus $B(x_0,R') \backslash B(x_0,R)$ towards the smaller annulus $B(x_0,R'-r) \backslash B(x_0,R+r)$).  However, aside from this detail, the proof method is essentially identical and is omitted.  Sending $R'$ to infinity and using the monotone convergence theorem, we may in fact replace the annulus $B(x_0,R') \backslash B(x_0,R)$ with the exterior region $\R^3 \backslash B(x_0,R)$, and the annulus $B(x_0,R'-r) \backslash B(x_0,R+r)$ with $\R^3 \backslash B(x_0,R+r)$.
\end{remark}

Theorem \ref{enstrophy-loc} asserts, roughly speaking, that if the $H^1_x$ norm of the data is small on a ball, then for a quantitative amount of later time, the $H^1_x$ norm of the solution remains small on a slightly smaller ball.  As the $H^1$ norm is subcritical, we expect this sort of result to persist to higher regularities, in the spirit of \cite{serrin-0}.  It is therefore unsurprising that this is indeed the case:

\begin{proposition}[Higher regularity]\label{higher}  Let $(u,p,u_0,f,T)$ be a finite energy almost smooth solution with $T \leq T_*$.  Let $B(x_0,R)$, $\eta$, $\delta$, $r$ obey the conditions \eqref{eeta}, \eqref{delta-4}, \eqref{r-large} from Theorem \ref{enstrophy-loc}.  Then for any compact subset $K$ in the interior of $B(x_0,R-r)$ and any $k \geq 1$, one can bound
$$ \| \nabla^k u \|_{L^\infty_t L^2_x( [0,T] \times K )} +
\| \nabla^{k+1} u \|_{L^2_t L^2_x( [0,T] \times K )}
 \lesssim_{k,K,E(u_0,f,T),\delta,T_*,R,A_k} 1$$
where
$$ A_k := \sum_{j=0}^{k} \| \nabla^j u_0 \|_{L^2_x(B(x_0,R))} + \| \nabla^j f \|_{L^\infty_t L^2_x([0,T] \times B(x_0,R))}.$$
In particular, one has
$$ \|u\|_{X^k([0,T] \times K)}  \lesssim_{k,K,E(u_0,f,T),\delta,T_*,R,A_k} 1.$$
\end{proposition}

\begin{proof}  We allow all implied constants to depend on $k,K,E(u_0,f,T),\delta,T_*,R,A_k$.  
We introduce a compact set
$$ K \subset K_1 \subset K_2 \subset K_3 \subset K_4 \subset K_5 \subset B(x_0,R-r)$$
which each set lying in the interior if the next set.  Let $\eta$ be a smooth function supported on $K_2$ that equals $1$ on $K_1$; we allow implied constants to depend on $\eta$.  

We begin with the $k=1$ case.  From Theorem \ref{enstrophy-loc} one already has
$$ \| \omega \|_{L^\infty_t L^2_x( [0,T] \times K_1 )} +
\| \nabla \omega \|_{L^2_t L^2_x( [0,T] \times K_1 )} \lesssim 1.$$
To pass from $\omega$ to $u$, we use integration by parts.  Since $\omega = \nabla \times u$ and $u$ is divergence-free, a standard integration by parts shows that
$$
\frac{1}{2} \int_{\R^3} |\omega|^2 \eta = \int_{\R^3} |\nabla u|^2 \eta + \int_{\R^3} \bigO( |u|^2 \nabla^2 \eta ).$$
By Lemma \ref{energy-est}, the error term is $O(1)$, and so we have
$$ \int_{K} |\nabla u|^2 \lesssim 1.$$
Similarly, by replacing $\omega$ and $u$ by their derivatives we also see that
$$
\frac{1}{2} \int_{\R^3} |\nabla \omega|^2 \eta = \int_{\R^3} |\nabla^2 u|^2 \eta + \int_{\R^3} \bigO( |\nabla u|^2 \nabla^2 \eta ).$$
By Lemma \ref{energy-est}, the error term is $O(1)$ after integration in time, and so we also have
$$ \int_0^T \int_{K} |\nabla^2 u|^2\ dx dt \lesssim 1$$
as desired.

We now turn to the $k=2$ case.  This is the most difficult, as we currently only control regularities that are half a derivative better than the critical regularity (which would place $u$ in $H^{1/2}_x$), and wish to boost this to three halves of a derivative above critical; this requires at least two iterations of the Duhamel formula.  The arguments will be analogous to the regularity arguments in Theorem \ref{lwp-h1} or Lemma \ref{quant-reg}.
By \eqref{ns-project} we see that $u\eta$ obeys the truncated equation
\begin{equation}\label{tea}
 \partial_t (\eta u) - \Delta (\eta u) = \eta \bigO( P \nabla( u u ) ) + \eta P f + \bigO( \nabla u \nabla \eta ) + \bigO( u \nabla^2 \eta )
\end{equation}
for almost all $t$.  Meanwhile, from the $k=1$ case and Lemma \ref{energy-est} we already have the estimates
\begin{equation}\label{start} 
\| u \|_{L^\infty_t L^2_x([0,T] \times \R^3)} + \| \nabla u \|_{L^\infty_t L^2_x([0,T] \times K_4)} + \| \nabla^2 u \|_{L^2_t L^2_x([0,T] \times K_4)} \lesssim 1
\end{equation}
and from the definition of $A_2$ we have
\begin{equation}\label{noble}
\| \nabla^j u_0 \|_{L^2_x(B(x_0,R))} + \| \nabla^j f \|_{L^\infty_t L^2_x([0,T] \times B(x_0,R))} \lesssim 1
\end{equation}
for $j=0,1,2$.

We claim that all terms on the right hand side of \eqref{tea} have an $L^4_t L^2_x([0,T] \times \R^3$ norm of $O(1)$.  The only difficult term here is $\eta P \bigO( \nabla( u u ) )$; the other three terms on the right-hand side are easily estimated in $L^4_t L^2_x$ (and even in $L^2_t L^2_x$) using \eqref{start} and \eqref{noble}.  We now estimate
$$ \| \eta \bigO( P \nabla(u u) ) \|_{L^4_t L^2_x([0,T] \times \R^3)}.$$
We split $u u = \tilde \eta u u + (1-\tilde \eta) u u$, where $\tilde \eta$ is a smooth cutoff supported on $K_4$ that equals $1$ on $K_3$.  For the contribution of the nonlocal portion $(1-\tilde \eta)$, one can use the smoothness of the kernel of the operator $P$ away from the origin to bound this contribution by $\lesssim \| \bigO( u u ) \|_{L^4_t L^1_x([0,T] \times \R^3)}$, which is acceptable by \eqref{start}; for future reference we note that this argument bounds this contribution on $L^2_t L^2_x$ norm as well as in $L^4_t L^2_x$ norm.  For the local portion $\tilde \eta u u$, we discard the $\eta$ and $P$ projections and bound this by
$$ \lesssim \|\bigO( \nabla( \tilde \eta u u ) ) \|_{L^4_t L^2_x([0,T] \times \R^3)}.$$
But this is acceptable by \eqref{bilinear}.

We have now placed the right-hand side of \eqref{tea} in $L^4_t L^2_x([0,T] \times \R^3)$ with norm $O(1)$.  Meanwhile, from \eqref{noble} the initial data $u_0 \eta$ is in $H^2_x(\R^3)$ with norm $O(1)$.  Applying the energy estimate \eqref{energy-duh3}, we conclude that 
$$ \| u \eta \|_{L^\infty_t H^{3/2-\sigma}_x([0,T] \times \R^3)} + \| u \eta \|_{L^2_t H^{5/2-\sigma}_x([0,T] \times \R^3)} \lesssim_\sigma 1$$
for any $\sigma>0$.  A similar argument (shifting the compact sets) also gives
$$ \| u \eta' \|_{L^\infty_t H^{3/2-\sigma}_x([0,T] \times \R^3)} + \| u \eta' \|_{L^2_t H^{5/2-\sigma}_x([0,T] \times \R^3)} \lesssim_\sigma 1$$
where $\eta'$ is a smooth function supported on $K_5$ that equals $1$ on $K_4$.  In particular, by Sobolev embedding, on $[0,T] \times K_4$, $u$ is in $L^\infty_t L^{12}_x$, $\nabla u$ is in $L^2_t L^{12}_x \cap L^\infty_t L^{12/5}_x$, and $\nabla^2 u$ is in $L^2_t L^{12/5}_x$, which together with \eqref{start} and the H\"older inequality now allows one to conclude that $\bigO( \nabla( \tilde \eta u u ) )$ has an $L^2_t H^1_x([0,T] \times \R^3)$ of $O(1)$.  Repeating the previous arguments, we now conclude that the right-hand side of \eqref{tea} lies in $L^2_t H^1_x([0,T] \times \R^3)$ with norm $O(1)$, and hence by \eqref{energy-duh2}
$$
\| \eta u \|_{L^\infty_t H^2_x([0,T] \times \R^3)} + \| \eta u \|_{L^2_t H^3_x([0,T] \times \R^3)}
$$
which gives the $k=2$ case.

The higher $k$ cases are proven by similar arguments, but are easier as we now have enough regularity to place $u$ in $L^\infty_t L^\infty_x([0,T] \times K_5)$ with norm $O(1)$; we leave the details to the reader.  (For instance, to establish the $k=3$ case, one can verify using the estimates already obtained from the $k=2$ case that the right-hand side of \eqref{tea} has an $L^2_t H^1_x([0,T] \times \R^3)$ norm of $O(1)$.
\end{proof}

\begin{remark} As in Remark \ref{lota}, one can extend the results here to the periodic setting so long as one has $T \leq L^2$ and $R \leq L$; we omit the details.
\end{remark}

For our application to constructing Leray-Hopf weak solutions, we will need a generalisation of Theorem \ref{enstrophy-loc} to the case when one has hyperdissipation.  More precisely, we introduce a small hyperdissipation parameter $\eps > 0$, and consider solutions $(u^{(\eps)},p^{(\eps)}, u_0, f, T)$ to the regularised Navier-Stokes equation, which are defined precisely as with the usual concept of a Navier-Stokes solution, but with \eqref{ns} replaced by the regularised variant
\begin{equation}\label{ns-hyper}
 \partial_t u^{(\eps)} + (u^{(\eps)} \cdot \nabla) u^{(\eps)} = \Delta u^{(\eps)} - \eps \Delta^2 u^{(\eps)} - \nabla p^{(\eps)} + f.
\end{equation}
With hyperdissipation, the global regularity problem becomes much easier (the energy is now subcritical rather than supercritical), and indeed it is not difficult to use energy methods (see e.g. \cite{lions}) to show the existence of a unique almost smooth finite energy solution to this regularised equation $(u^{(\eps)},p^{(\eps)}, u_0, f, T)$ from any given smooth finite energy data $(u_0,f,T)$.  The energy estimate in Lemma \ref{energy-est} remains true in this case (uniformly in $\eps$), and one easily verifies that one obtains an additional estimate
\begin{equation}\label{hyperdis}
 \eps \int_0^T \int_{\R^3} |\nabla^2 u(t,x)|^2\ dt dx \lesssim E(u_0,f,T)
\end{equation}
in this hyperdissipative setting.  One can also verify (with a some tedious effort) that Proposition \ref{bounded-speed} also holds in this hyperdissipative setting as long as $\eps$ is sufficiently small, basically because the hyperdissipative heat operators $e^{t(\Delta-\eps\Delta^2)}$ obey essentially the same estimates \eqref{dispersive}, \eqref{energy-decay} as $e^{t\Delta}$ if $0 \leq t \leq T$ and $\eps$ is sufficiently small depending on $T$; we omit the details.

One can define the vorticity $\omega^{(\eps)} := \nabla \times u^{(\eps)}$ of a regularised solution as before.  This vorticity obeys an equation almost identical to \eqref{vorticity-eq}, but with an additional hyperdissipative term $- \eps \nabla^2 \omega^{(\eps)}$ on the right-hand side.  One can then repeat the proof of Theorem \ref{enstrophy-loc} with this additional term.  Integrating by parts a large number of times, one obtains a similar decomposition to \eqref{yyy} for the derivative of the localised enstrophy, but with the addition of a negative term $-\eps \int_{\R^3} |\nabla^2 \omega|^2 \eta$ on the right-hand side, plus some boundary terms which are bounded by $\tilde b(t)$, where
\begin{align*}
\tilde b(t) &:= \sum_{r = R'(t), R'(t)-c^{0.1} \delta^{-2}} \eps c^{-0.1} \delta^{2} R^2 \int_{S^2} |\nabla \omega( t, r \alpha )|^2\ d\alpha \\
&\quad + \eps c^{-0.2} \delta^{4} \int_{R'(t)-c^{0.1} \delta^{-2} \leq |x| \leq R'(t)} |\nabla \omega(t,x)|^2\ dx
\end{align*}
is a hyperdissipative analogue of $b(t)$.
By using the same averaging argument used to bound $\int_0^T b(t)\ dt$ for typical $R'$, one can also simultaneously obtain a comparable bound for $\int_0^T \tilde b(t)\ dt$ (taking advantage of the additional estimate \eqref{hyperdis}).  The rest of the argument in Theorem \ref{enstrophy-loc} works with essentially no changes; we omit the details.  The proof of Proposition \ref{higher} is also essentially identical, after one notes that energy estimates such as \eqref{energy-duh2} continue to hold in the hyperdissipative setting.  Summarising, we we obtain

\begin{proposition}\label{hyperdis-prop}  Theorem \ref{enstrophy-loc} and Proposition \ref{higher} continue to hold in the presence of hyperdissipation, uniformly in the limit $\eps \to 0$.
\end{proposition}

\section{Consequences of enstrophy localisation}

We now give a number of applications of the enstrophy localisation result, Theorem \ref{enstrophy-loc}.  Many of these applications resemble existing results in the literature, but with weaker decay hypotheses on the initial data and solution (in particular, we will usually only assumes either finite energy or finite $H^1$ norm); the main point is that the localisation afforded by Theorem \ref{enstrophy-loc} can significantly reduce the need to assume any stronger decay hypotheses.

We begin with the observation that finite energy smooth solutions automatically have bounded enstrophy if the initial data has bounded enstrophy:

\begin{corollary}[Bounded enstrophy]\label{ens-bound}  Let $(u,p,u_0,f,T)$ be an almost smooth, finite energy solution, such that the initial data $(u_0,f,T)$ has finite $H^1$ norm. Then $u \in X^1([0,T] \times \R^3)$; in particular, $(u,p,u_0,f,T)$ is an $H^1$ solution.
\end{corollary}

\begin{proof}  Let $\delta > 0$ be small enough (depending on $E(u_0,f,T), T$) that the condition \eqref{delta-4} holds.  As $(u_0,f,T)$ has finite $H^1$ norm, we have
$$
\| \omega_0 \|_{L^2_x(\R^3)} + \| \nabla \times f \|_{L^1_t L^2_x([0,T] \times \R^3)} < \infty.$$
By the monotone convergence theorem, we thus have for $R$ sufficiently large that
$$
\| \omega_0 \|_{L^2_x(\R^3 \backslash B(0,R))} + \| \nabla \times f \|_{L^1_t L^2_x([0,T] \times (\R^3 \backslash B(0,R)))} \leq \delta.$$
Applying Theorem \ref{enstrophy-loc} (inverted as in Remark \ref{inversion-again}), we conclude that
$$
\| \omega \|_{L^\infty_x L^2_x([0,T] \times (\R^3 \backslash B(0,R+r)))} + \| \nabla \omega \|_{L^2_t L^2_x([0,T] \times (\R^3 \backslash B(0,R+r)))} 
\lesssim \delta$$
for some finite radius $r$, if $R$ is sufficiently large; in particular, $\omega$ lies in $L^\infty_t L^2_x \cap L^2_t H^1_x$ in the exterior region $[0,T] \times (\R^3 \backslash B(0,R+r))$.  On the other hand, as $u$ is almost smooth, $\omega$ also lies in $L^\infty_t L^2_x \cap L^2_t H^1_x$ in the interior region $[0,T] \times B(0,R+r+1)$ (say).  Gluing these two bounds together, we conclude that
$$ \omega \in L^\infty_t L^2_x \cap L^2_t H^1_x( [0,T] \times \R^3 );$$
meanwhile, from Lemma \ref{energy-est} one has
$$ u \in L^\infty_t L^2_x \cap L^2_t H^1_x( [0,T] \times \R^3 ).$$
Since $u$ is divergence-free and $\omega = \nabla \times u$, the claim then follows from Fourier analysis.
\end{proof}

\begin{remark}  From Corollary \ref{max-cauchy} we know that smooth solutions to the Navier-Stokes solutions can be continued in time as long as the $H^1$ norm remains bounded.  However, Corollary \ref{ens-bound} certainly does not allow one to solve the global regularity problem for Navier-Stokes, because the proof heavily relies on the solution $u$ being \emph{complete} rather than \emph{incomplete}, thus it is (almost) smooth all the way up to the final time $T$, and not just smooth on $[0,T)$.  Instead, what Corollary \ref{ens-bound} does is to show that the solution from $H^1$ data is well-behaved when one is sufficiently close to spatial infinity; in particular, it does not prevent turbulent behaviour in bounded regions of spacetime.
\end{remark}

\begin{remark} If $(u,p,u_0,0,T)$ is an almost smooth homogeneous finite energy solution, then by Lemma \ref{energy-est} we see that $u(t) \in H^1_x(\R^3)$ for almost every time $t \in [0,T]$.  Applying the time translation symmetry \eqref{time-translate} for a small time shift $t_0$, we can then convert the finite energy data to $H^1$ data, and then by Corollary \ref{ens-bound}, we conclude that in fact $u(t) \in H^1_x(\R^3)$ for \emph{all} non-zero times $t \in (0,T]$, and furthermore that $u(t)$ is bounded in $H^1_x$ as soon as $t$ is bounded away from zero.
\end{remark}

Since $H^1$ almost smooth solutions with normalised pressure are automatically $H^1$ mild solutions, for which uniqueness was established in Theorem \ref{lwp-h1-r3}, we thus have uniqueness in the almost smooth finite energy category from smooth $H^1$ data:

\begin{corollary}[Unconditional uniqueness]\label{unconditional}  Let $(u_0,f,T)$ be smooth $H^1$ data.  Then there is at most one almost smooth finite energy solution $(u,p,u_0,f,T)$ with this data and with normalised pressure.
\end{corollary}

This result resembles the standard ``weak-strong uniqueness'' results in the literature, such as those in \cite{prodi}, \cite{serrin}, \cite{germain}, \cite{germain-2}.  The main novelty here is the lack of decay hypotheses beyond the finite energy hypothesis; note that the almost smoothness of the solution gives plenty of integrability on compact regions of space, but does not imply any global integrability in space.

\begin{remark}\label{uniq}  We conjecture that one still retains uniqueness even if the data $(u_0,f,T_*)$ is merely smooth and finite energy, rather than smooth and $H^1$.  Note from Lemma \ref{energy-est} that $u(t)$ has finite $H^1_x(\R^3)$ norm for almost every time $t$, which in principle allows one to enforce uniqueness after any given positive time (in the homogeneous case $f=0$, at least), but it is not clear to the author how to prevent instantaneous failure of uniqueness at the initial time $t=0$ with only a smooth finite energy hypothesis on the initial data.  It may however be possible to adapt the ``weak-strong'' uniqueness results of Germain \cite{germain, germain-2} to this category, perhaps in combination with the local $H^1$ control given by Theorem \ref{enstrophy-loc}.
\end{remark}

We now use the enstrophy localisation result to study solutions as they approach a (potential) blowup time $T_*$.  

\begin{proposition}[Uniform smoothness outside a ball]\label{uni-ball}  Let $(u,p,u_0,f,T_*^-)$ be an incomplete almost smooth $H^1$ solution with normalised pressure for all times $0 < T < T_*$.  Then there exists a ball $B(0,R)$ such that
\begin{equation}\label{upf}
 u, p, f, \partial_t u \in L^\infty_t C^k_x([0,T_*) \times K)
\end{equation}
for all $k \geq 0$ and all compact subsets $K$ of $\R^3 \backslash B(0,R)$.  
\end{proposition}

We remark that similar results were obtained in \cite{ckn} assuming additional spatial decay hypotheses on the data at infinity, and in particular that $\int_{\R^3} |u_0(0,x)|^2 |x|\ dx < \infty$.  The main novelty in this proposition is that one only assumes square-integrability of $u_0$ and its first derivatives, without any further decay assumption.

\begin{proof}  From the argument in the proof of Corollary \ref{ens-bound} (noting that the bounds are uniform for all times $T$ in a compact set), one can already find a ball $B(0,R_0)$ for which
$$ u \in X^1([0,T_*) \times (\R^3 \backslash B(0,R_0))).$$
Using Proposition \ref{higher}, we then conclude the existence of a larger ball $B(0,R)$ such that
$$ u \in X^k([0,T_*) \times K)$$
for all $k \geq 1$ and all compact subsets $K$ of $\R^3 \backslash B(0,R)$.  From this, Sobolev embedding, and \eqref{pressure-point} (using the smoothness of the kernel of $\nabla^k \Delta^{-1}$ away from the origin) we obtain \eqref{upf} for $u,p,f$ as desired.  If one then applies \eqref{ns} and solves for $\partial_t u$ one obtains the bound for $\partial_t u$ also.
\end{proof}

\begin{remark} From \eqref{upf} one can \emph{continuously} extend $u$ up to the portion $\{T_*\} \times (\R^3 \backslash B(0,R))$ of the boundary (cf. the partial regularity theory in \cite{ckn}).  
However, we were unable to demonstrate that $u$ could be extended \emph{smoothly} up to the boundary (or even that $\partial_t u$ is continuous in time at the boundary).  The problem is due to the non-local effects of pressure; the solution $u$ could be blowing up at time $T_*$ in the interior of $B(0,R)$, leading (via \eqref{pressure-point}) to time oscillations of the pressure in $K$ (which cannot be directly damped out by the smoothness of the $\Delta^{-1}$ kernel, which only attenuates \emph{spatial} oscillations) which by \eqref{ns} could lead to time oscillations of the solution $u$ in $K$.   Indeed, as Proposition \ref{counter} shows, these time oscillations can have a non-trivial effect on the regularity of the solution.
\end{remark}

\begin{remark}\label{uni-rem} For future reference, we observe that Proposition \ref{uni-ball} did not require the full spacetime smoothness on $f$; it would suffice to have $f \in L^\infty_t C^k_x([0,T_*) \times K)$ for all $k \geq 0$ and compact $K$ in order to obtain the conclusion \eqref{upf}.  This is because at no stage in the argument was it necessary to differentiate $f$ in time.
\end{remark}

In a similar spirit, we may construct Leray-Hopf weak solutions that are spatially smooth outside of a ball for any fixed time $T$.  More precisely, define a \emph{Leray-Hopf weak solution} $(u,p,u_0,f,T)$ to smooth finite energy data $(u_0,f,T)$ to be a distributional solution $u \in X^0([0,T] \times \R^3)$ to \eqref{ns} (after expressing this equation in divergence form) which is continuous in time in the weak topology of $L^2_x(\R^3)$, and which obeys the energy inequality
\begin{equation}\label{tim}
 \frac{1}{2} \|u(t)\|_{L^2_x(\R^3)}^2 + \int_0^t \| \nabla u(t) \|_{L^2_x(\R^3)}^2\ dx \leq E(u_0,f,T).
\end{equation}
The existence of such solutions was famously demonstrated by Leray \cite{leray} for arbitrary finite energy data $(u_0,f,T)$; the singularities of these solutions were analysed in a vast number of papers, which are too numerous to cite here, but we will point out in particular the seminal work of Caffarelli, Kohn, and Nirenberg \cite{ckn}.  

Our main regularity result for Leray-Hopf solutions is as follows.

\begin{proposition}[Existence of partially smooth Leray-Hopf weak solutions]\label{partial}  Let $(u_0,f,T)$ be smooth $H^1$ data.  Then there exists
a Leray-Hopf weak solution $(u,p,u_0,f,T)$ to the given data and a ball $B(0,R)$ such that $u$ is spatially smooth in $[0,T] \times (\R^3 \backslash B(0,R))$ (i.e. for each $t \in [0,T]$, $u(t)$ is smooth outside of $B(0,R)$).
\end{proposition}

Again, similar results were obtained in \cite{ckn} under stronger decay hypotheses on the initial data.  We also remark that weak solutions which were only locally of finite energy, from data of uniformly locally finite energy, were constructed in \cite{lem2}; the ability to localise the weak solution construction in this fashion is similar in spirit to the results in the above proposition.

\begin{proof} (Sketch) We use a standard hyperdissipation\footnote{It may also be possible to use other regularisation methods here, such as velocity regularisation, to construct the Leray-Hopf weak solution; however, due to the delicate nature of the proof of the localised enstrophy estimate (Theorem \ref{enstrophy-loc}), we were not able to verify that this estimate remained true in the velocity-regularised setting, uniformly in the regularisation parameter, due to the less favourable vorticity equation in this setting.} regularisation argument.  Let $\eps > 0$ be a small parameter, and consider the almost smooth finite-energy solution $(u^{(\eps)},p^{(\eps)}, u_0, f, T)$ to the regularised Navier-Stokes system \eqref{ns-hyper}, which can be shown to exist by energy methods.  By Proposition \ref{hyperdis-prop}, we can extend Theorem \ref{enstrophy-loc} and Proposition \ref{higher} (and thence Proposition \ref{uni-ball}), to these regularised solutions $u^{(\eps)}$, with bounds that are uniform in $\eps$ as $\eps \to 0$. As a consequence, we can find a ball $B(0,R)$ \emph{independent of $\eps$} such that for every compact set $K$ outside of $B(0,R)$ and every $k \geq 0$, $\nabla^k u^{(\eps)}$ lies in $L^\infty_t L^\infty_x([0,T_*] \times K)$ uniformly in $N$.  If we then extract a weak limit point $u$ of the $u^{(\eps)}$, then one by standard arguments one verifies that $u$ is a Leray-Hopf weak solution which is spatially smooth outside of $B(0,R)$.  
\end{proof}

\begin{remark} As before, we are unable to demonstrate regularity of $u$ in time due to potential non-local effects caused by the pressure, which could in principle cause singularities inside $B(0,R)$ to create time singularities outside of $B(0,R)$. 
\end{remark}

\begin{remark} Uniqueness of Leray-Hopf solutions remains a major unsolved problem, for which we have nothing new to contribute; in particular, we do not assert that \emph{all} Leray-Hopf solutions from smooth data obey the conclusions of Proposition \ref{partial}.  However, if $(u_0,f,\infty)$ is globally defined smooth $H^1$ data, the above argument gives a single global Leray-Hopf weak solution $(u,p,u_0,f,\infty)$ with the property that, for each finite time $T<\infty$, there exists a radius $R_T <\infty$ such that $u$ is smooth in $[0,T] \times (\R^3 \backslash B(0,R))$.  If we restrict to the case $f=0$, then from \eqref{tim} we see that $\|\nabla u(t) \|_{L^2_x(\R^3)}$ must become arbitrarily small along some sequence of times $t = t_n$ going to infinity.  If $\|\nabla u(t) \|_{L^2_x(\R^3)}$ is small enough depending on $E(u_0,0,\infty)$, then standard perturbation theory arguments (see e.g. \cite{kato-lp}) allow one to obtain a smooth, bounded enstrophy solution from the data $u(t)$ on $(t,+\infty)$, which by the uniqueness theory of Serrin \cite{serrin} must match the Leray-Hopf weak solution $u$ on $(t,+\infty)$.  As such, we conclude in the homogeneous smooth $H^1$ case that one can construct a global Leray-Hopf weak solution which is spatially smooth outside of a compact subset of spacetime $[0,+\infty) \times \R^3$.  Again, we emphasise that this global weak solution need not be unique.
\end{remark}

\section{Smooth $H^1$ solutions}\label{enstrophy-sec}

The purpose of this section is to establish Theorem \ref{main}(iii).  To do this, we will need the ability to localise smooth divergence-free vector fields, as follows.

\begin{lemma}[Localisation of divergence-free vector fields]\label{divloc}  Let $T>0$, $0 < R_1 < R_2 < R_3 < R_4$, and let $u: [0,T) \times (B(0,R_4) \backslash B(0,R_1)) \to \R^3$ be spatially smooth and divergence-free, and such that
$$ u, \partial_t u \in L^\infty_t C^k_x( [0,T) \times (B(0,R_4) \backslash B(0,R_1)) )$$
for all $k \geq 0$ and
\begin{equation}\label{stokes}
 \int_{|x|=r} u(t,x) \cdot n\ d\alpha(x) = 0
 \end{equation}
for all $R_1 < r < R_4$ and $t \in [0,T)$, where $n$ is the outward normal and $d\alpha$ is surface measure.  Then there exists a spatially smooth and divergence-free vector field $\tilde u: [0,T) \times (B(0,R_4) \backslash B(0,R_1)) \to \R^3$ which agrees with $u$ on $[0,T) \times (B(0,R_2) \backslash B(0,R_1))$, but vanishes on $[0,T) \times (B(0,R_4) \backslash B(0,R_3))$.  Furthermore, we have 
$$ \tilde u, \partial_t u \in L^\infty_t C^k_x( [0,T) \times (B(0,R_4) \backslash B(0,R_1)) )$$
for all $k \geq 0$.  

Finally, if we have 
$$ 1 \leq 2R_2 \leq R_3 \lesssim R_2$$
then we have the more quantitative bound
\begin{equation}\label{quant}
 \| \tilde u \|_{L^\infty_t H^k( [0,T) \times (B(0,R_4) \backslash B(0,R_1)) )} \lesssim_k \| u \|_{L^\infty_t H^{k+1}( [0,T) \times (B(0,R_4) \backslash B(0,R_1)) )}
 \end{equation}
for any $k$.  (This latter property will come in handy in the next section.)
\end{lemma}

Note that the hypothesis \eqref{stokes} is necessary, as can be seen from Stokes' theorem.  Lemmas of this type first appear in the work of Bogovskii \cite{bog}.

\begin{proof}  One could obtain this lemma as a consequence of the machinery of compactly supported divergence-free wavelets \cite{lemarie}, but for the convenience of the reader we give a self-contained proof here.

Let $X$ denote the vector space of all divergence-free smooth functions $u: B(0,R_4) \backslash B(0,R_1) \to \R^3$ obeying the mean zero condition 
\begin{equation}\label{meanzero}
\int_{|x|=r} u(x) \cdot n\ d\alpha(x) = 0
\end{equation}
for all $R_1 < r < R_4$, and such that $\|u\|_{C^k( (0,R_4) \backslash B(0,R_1) )} < \infty$ for all $k$.  It will suffice to construct a linear transformation $P: X \to X$ that is bounded\footnote{One can reduce this loss of regularity by working in more robust spaces than the classical $C^k$ spaces, such as Sobolev spaces $H^s$ or H\"older spaces $C^{k,\alpha}$, but we will not need to do so here.} from $C^{k+2}$ to $C^k$, i.e.
$$ \|Pu\|_{C^k( (0,R_4) \backslash B(0,R_1) )} \lesssim_{R_1,R_2,R_3,R_4,k} \|u\|_{C^{k+2}( (0,R_4) \backslash B(0,R_1) )}$$
for all $k \geq 0$, and such that $Pu$ equals $u$ on $B(0,R_2)\backslash B(0,R_1)$ and vanishes on $B(0,R_4) \backslash B(0,R_3)$, as one can then simply define $\tilde u(t) := P \tilde u(t)$ for each $t \in [0,T)$.

We now construct $P$.  We work in polar coordinates $x = r \alpha$ with $R_1 \leq r \leq R_4$ and $\alpha \in S^2$ (thus avoiding the coordinate singularity at the origin), and decompose $u(r,\alpha)$ as the sum of a radial vector field $u_r(r,\alpha) \alpha$ for some scalar field $u_r$, and an angular vector field $u_\alpha(r,\alpha)$ which is orthogonal to $\alpha$; thus, for fixed $r$, $u_\alpha(r)$ can be viewed as a smooth vector field on the unit sphere $S^2$ (i.e. a smooth section of the tangent bundle of $S^2$).  The divergence-free condition on $u$ in these coordinates then reads
\begin{equation}\label{ror}
 \partial_r u_r(r) + \frac{1}{r} \nabla_\alpha \cdot u_\alpha(r) = 0
\end{equation}
while the mean zero condition \eqref{meanzero} reads
$$ \int_{S^2} u_r(r,\alpha)\ d\alpha = 0.$$
Note that either of these conditions implies that $\partial_r u_r(r)$ has mean zero on $S^2$ for each $r$.  From \eqref{ror} and Hodge theory we see that
$$ u_\alpha(r) = r \Delta_\alpha^{-1} \nabla_\alpha \partial_r u_r(r) + v(r)$$
where $\Delta_\alpha^{-1}$ inverts the Laplace-Beltrami operator $\Delta_\alpha$ on smooth mean zero functions on $S^2$, and $v(r)$ is a smooth divergence-free vector field on $S^2$ that varies smoothly with $r$.

Let $\eta: [R_1,R_4] \to \R^+$ be a smooth function that equals $1$ on $[R_1,R_2]$ and vanishes on $[R_3,R_4]$.  We define
$$ \tilde u_r := \eta(r) u_r$$
and
$$ \tilde u_\alpha(r) = r \Delta_\alpha^{-1} \nabla_\alpha \partial_r \tilde u_r(r) + \eta(r) v(r)$$
and
$$ Tu := \tilde u := \tilde u_r \alpha + \tilde u_\alpha.$$
One then easily verifies that $\tilde u$ is smooth, divergence-free, obeys \eqref{meanzero}, depends linearly on $u$, equals $u$ on $B(0,R_2) \backslash B(0,R_1)$, and vanishes on $B(0,R_4) \backslash B(0,R_4)$.  It is also not difficult (using the fundamental solution of $\Delta_\alpha^{-1}$) to see that $T$ maps $C^{k+2}$ to $C^k$ (with some room to spare).  The claim follows.

Finally, we prove \eqref{quant}.  It suffices to show that
$$ 
 \| Tu \|_{H^k( B(0,R_3) \backslash B(0,R_2) )} \lesssim_k 1$$
whenever $k \geq 0$, and $u \in X$ is such that
$$ \| u \|_{H^{k+2}( B(0,R_4) \backslash B(0,R_1) )} \lesssim 1.$$

Henceforth all spatial norms will be on $B(0,R_3) \backslash B(0,R_2)$, and all implied constants may depend on $k$.  As $u$ has an $H^{k+1}$ norm of $O(1)$, $u_r$ and hence $\tilde u_r$ has an $H^{k+1}$ norm of $O(1)$ also.  As for $\tilde u_\alpha$, we observe from the Leibniz rule that
$$ \tilde u_\alpha = \eta u_\alpha + (r \partial_r \eta(r)) \Delta_\alpha^{-1} \nabla_\alpha u_r(r).$$
As $u$ has an $H^{k+1}$ norm of $O(1)$, we have $r^{-i} \nabla_\alpha^i \partial_r^j u_\alpha$ has an $L^2$ norm of $O(1)$ whenever $i+j \leq k+1$, which (using elliptic regularity in the angular variable) implies that $r^{-i} \nabla_\alpha^i \partial_r^j \tilde u_\alpha$ has an $L^2$ norm of $O(1)$ whenever $i+j \leq k$.  This gives $\tilde u = \tilde u_r + \tilde u_\alpha$ an $H^k$ norm of $O(1)$ as claimed.
\end{proof}

We can now establish Theorem \ref{main}(iii):

\begin{theorem}\label{enstrophy-thm} Suppose Conjecture \ref{global-h1} is true.  Then Conjecture \ref{global-h1-spatial} is true.
\end{theorem}

\begin{proof}  
In view of Corollary \ref{max-cauchy}, it suffices to show that if $(u,p,u_0,f,T_*^-)$ is an incomplete $H^1$ mild solution up to time $T_*$, with $u_0, f$ spatially smooth in the sense of Conjecture \ref{global-h1-spatial}, then $u$ does not blow up in enstrophy norm, thus
$$ \limsup_{t \to T_*^-} \|u(t)\|_{H^1_x(\R^3)} < \infty.$$

Let $R > 0$ be a sufficiently large radius.  By arguing as in Corollary \ref{ens-bound}, we have
$$ u \in L^\infty_t H^1_x( \R^3 \backslash B(0,R) ) $$
and thus the blowup must be localised in space:
\begin{equation}\label{beebo}
 \limsup_{t \to T_*^-} \|u(t)\|_{H^1_x(B(0,R))} < \infty.
\end{equation}

By Proposition \ref{uni-ball} and Remark \ref{uni-rem} (and increasing $R$ if necessary) we also have
\begin{equation}\label{uphut}
 u, p, f, \partial_t u \in L^\infty_t C^k_x([0,T_*) \times (B(0,5R) \backslash B(0,2R)))
\end{equation}
for all $k \geq 0$.  From Stokes' theorem and the divergence-free nature of $u$, we also have
$$ \int_{|x|=r} u(t,x) \cdot n\ d\alpha(x) = 0$$
for all $r>0$ and $t \in [0,T)$.  Applying Lemma \ref{divloc}, we can then find a spatially smooth divergence-free vector field $\tilde u: [0,T) \times (B(0,5R) \backslash B(0,2R)) \to \R^3$ which agrees with $u$ on $B(0,3R) \backslash B(0,2R)$ and vanishes outside of $B(0,4R)$, with
\begin{equation}\label{leo}
 \tilde u, \partial_t \tilde u \in L^\infty_t C^k_x( B(0,5R) \backslash B(0,2R) )
\end{equation}
for all $k \geq 0$.  We then extend $\tilde u$ by zero outside of $B(0,5R)$ and by $u$ inside of $B(0,2R)$, then $\tilde u$ is now smooth on all of $[0,T) \times \R^3$.

Let $\eta$ be a smooth function supported on $B(0,5R)$ that equals $1$ on $B(0,4R)$.
We define a new forcing term $\tilde f: [0,T) \times \R^3 \to \R$ by the formula
\begin{equation}\label{faf}
\tilde f := \partial_t \tilde u + (\tilde u \cdot \nabla) \tilde u - \Delta \tilde u + \nabla (p\eta),
\end{equation}
then $\tilde f$ is spatially smooth, supported on $B(0,5R)$ and agrees with $f$ on $B(0,3R)$.  From this and \eqref{leo}, \eqref{uphut} we easily verify that
$$ \tilde f \in L^\infty_t H^1_x([0,T_*) \times \R^3).$$

Note from taking divergences in \eqref{faf} and using the compact support of $p\eta$, $\tilde u$, $\tilde f$ that
$$ p\eta = - \Delta^{-1} ((\tilde u \cdot \nabla) \tilde u) + \Delta^{-1} \nabla \cdot \tilde f.$$
Thus, $(\tilde u, p\eta, \tilde u(0), \tilde f, T_*^-)$ is an incomplete $H^1$ pressure-normalised (and hence mild) solution with all components supported in $B(0,5R)$.  If we then choose a period $L$ larger than $10R$, then we may embed $B(0,5R)$ inside $\R^3/L\Z^3$ and obtain an incomplete periodic smooth solution 
$(\iota(\tilde u), \iota(p\eta), \iota(\tilde u(0)), \iota(\tilde f), T_*^-,L)$, where we use $\iota(f)$ to denote the extension by zero of a function $f$ supported in $B(0,5R)$, after embedding the latter in $\R^3/L\Z^3$.  By construction we then have
$$ \iota(\tilde f) \in L^\infty_t H^1_x([0,T_*) \times \R^3/L\Z^3).$$
As $\{T_*\}$ has measure zero, we may arbitrarily extend $\tilde f$ to $[0,T_*] \times \R^3/L\Z^3$ while staying in $L^\infty_t H^1_x$.  Applying either Conjecture \ref{global-h1} (and the uniqueness component to Theorem \ref{lwp-h1}) or Conjecture \ref{global-h1-quant}, we conclude that
$$ \iota(\tilde u) \in L^\infty_t H^1_x([0,T_*) \times \R^3/L\Z^3)$$
which implies (since $u$ and $\tilde u$ agree on $B(0,R)$) that
$$ u \in L^\infty_t H^1_x([0,T_*) \times B(0,R))$$
which contradicts \eqref{beebo}.  The claim follows.
\end{proof}

Observe that if we omit the embedding of $B(0,5R)$ in $\R^3/L\Z^3$ in the above argument, we can also deduce Conjecture \ref{global-h1-spatial} from Conjecture \ref{global-schwartz-spatial}.  Since Conjecture \ref{global-h1-spatial} clearly implies Conjecture \ref{global-schwartz-spatial} as a special case, we obtain Theorem \ref{main}(iii).

\begin{remark}\label{ref}  The referee has pointed out a variant of the above argument using the partial regularity theory of Caffarelli, Kohn, and Nirenberg \cite{ckn}, which allows one to partially reverse the above implications, and in particular deduce Conjecture \ref{global-periodic-normalised} from Conjecture \ref{global-h1-spatial}.  We sketch the argument as follows.  Assume Conjecture \ref{global-h1-spatial}, and assume for contradiction that Conjecture \ref{global-periodic-normalised} fails, thus there is a periodic solution with smooth inhomogeneous data which first develops singularities at some finite time $T$, and in particular at some location $(T,x_0)$.  We may extend the solution beyond this time as a weak solution.  Applying a periodic version of the theory in \cite{ckn}, we see that the set of singularities has zero one-dimensional parabolic measure, which among other things implies that the set of radii $r>0$ such that the solution is singular at $(T,x)$ for some $x$ with $|x-x_0|=r$ has measure zero.  Because of this, one can find radii $r_2 > r_1 > 0$ such that the solution is smooth in the annular region $\{ (t,x): 0 \leq t \leq T; r_1 \leq |x-x_0| \leq r_2 \}$.  By smoothly truncating the solution $u$ to this annulus as in the proof of Theorem \ref{enstrophy-thm}, one can then create a non-periodic $H^1$ mild solution to the inhomogeneous Navier-Stokes equation with spatially smooth data which develops a singularity at $(T,x_0)$ while remaining smooth up to time $T$, contradicting Conjecture \ref{global-h1-spatial} (when combined with standard uniqueness and regularity results, such as those in Theorem \ref{lwp-h1-r3}).  
\end{remark}

\section{Smooth finite energy solutions}\label{energy-sol-sec}

In this section we establish Theorem \ref{main}(v).  It is trivial that Conjecture \ref{global-energy-homog} implies Conjecture \ref{global-enstrophy-homog}, so it suffices to establish

\begin{theorem}\label{energy-thm}   Suppose that Conjecture \ref{global-enstrophy-homog} is true.  Then Conjecture \ref{global-energy-homog} is true.
\end{theorem}

\begin{proof}  Let $(u_0,0,T)$ be smooth homogeneous finite energy data.  Our task is to obtain an almost smooth finite energy solution $(u,p,u_0,0,T)$ with this data.  We allow all implied constants to depend on $u_0$.

We use a regularisation argument.  Let $N_n$ be a sequence of frequencies going to infinity, and set $u_0^{(n)} := P_{\leq N_n} u_0$, then $u_0^{(n)}$ converges to $u_0$ strongly in $L^2_x(\R^3)$, and $(u_0^{(n)}, 0, T)$ is smooth $H^1$ data for each $n$.  Thus, by hypothesis, we may find a sequence of almost smooth finite energy solutions $(u^{(n)},p^{(n)},u_0^{(n)},0,T)$ with this data.  

One could try invoking weak compactness right now to extract a solution, but as is well known, one only obtains a Leray-Hopf weak solution by doing so, which need not be smooth.  So we will first work to establish some additional regularity on the sequence (after passing to a subsequence as necessary) before extracting a weakly convergent limit.

Since the $(u_0^{(n)},0,T)$ are uniformly bounded in energy, we see from Lemma \ref{energy-est} that 
\begin{equation}\label{energy-awk}
\| u^{(n)} \|_{X^0([0,T] \times \R^3)} \lesssim 1.
\end{equation}
Now let $0 < \tau_0 < T/2$ be a small time.  From \eqref{energy-awk} and the pigeonhole principle, we may find a sequence of times $\tau^{(n)} \in [0,\tau_0]$ such that
$$ \| u^{(n)}(\tau^{(n)}) \|_{H^1_x(\R^3)} \lesssim \tau_0^{-1}.$$
Passing to a subsequence, we may assume that $\tau^{(n)}$ converges to a limit $\tau \in [0,\tau_0]$.  If we then take $\tau' \in [\tau,2\tau_0]$ sufficiently close to $\tau$, we may apply Lemma \ref{quant-reg} and conclude that
$$ \| u^{(n)}(\tau') \|_{H^{10}_x(\R^3)} \lesssim_{\tau,\tau',\tau_0} 1$$
(say) for all sufficiently large $n$.  Passing to a further subsequence, we may then assume that $u^{(n)}(\tau')$ converges weakly in $H^{10}_x(\R^3)$ (and thus locally strongly in $H^9_x$) to a limit $u'_0 \in H^{10}_x(\R^3)$.  By hypothesis, we may thus find an almost smooth $H^1$ solution $(u', p', u'_0, 0, T-\tau')$ with this data.

Meanwhile, by time translation symmetry \eqref{time-translate}, $(u^{(n)}(\cdot+\tau'), p^{(n)}(\cdot+\tau'), u^{(n)}(\tau'), 0, T-\tau')$ is also a sequence of almost smooth $H^1$ solutions.  Since $u^{(n)}(\tau')$ converges locally strongly in $H^9_x(\R^3)$ to $u'_0$, we would like to conclude that $u^{(n)}(t+\tau')$ also converges locally strongly to $u(t)$ in $H^1_x(\R^3)$, uniformly in $t \in [0,T-\tau']$.  This does not quite follow from the standard local well-posedness theory in Theorem \ref{lwp-h1-r3}, because this theory requires strong convergence in the \emph{global} $H^1_x(\R^3)$ norm.  However, we may take advantage of the local enstrophy estimates to spatially localise the local well-posedness theory, as follows.

Let $\eps > 0$ be a small quantity (depending on the solution $u' = (u', p', u'_0, 0, T-\tau')$) to be chosen later, let $R > 0$ be a sufficiently large radius (depending on $\eps$ and $(u', p', u'_0, 0, T-\tau')$) to be chosen later.  Since $u'_0$ is in $H^{10}_x(\R^3)$, we see from monotone convergence that
\begin{equation}\label{upo}
\| u'_0 \|_{H^{10}_x(\R^3 \backslash B(0,R))} \lesssim \eps
\end{equation}
if $R$ is sufficiently large depending on $\eps$.  Since the $u^{(n)}(\tau')$ converge locally strongly in $H^1_x(\R^3)$ to $u'_0$, we conclude that
$$ \| u^{(n)}(\tau') \|_{H^{10}_x(B(0,10R) \backslash B(0,R))} \lesssim \eps$$
if $n$ is sufficiently large depending on $R,\eps$.  Applying Theorem \ref{enstrophy-loc}, we conclude (if $R$ is large enough depending $u'_0$ and $T-\tau'$) that
$$ \| u^{(n)}(\cdot+\tau') \|_{X^1([0,T-\tau'] \times (B(0,9R) \backslash B(0,2R)))} \lesssim \eps$$
for $n$ sufficiently large depending on $R,\eps$.  Using Duhamel's formula (and Corollary \ref{as-mild}) repeatedly as in the proof of Proposition \ref{higher}, we may in fact conclude that
\begin{equation}\label{waffle}
\| \partial_t^i u^{(n)}(\cdot+\tau') \|_{L^\infty_t H^6_x([0,T-\tau'] \times (B(0,8R) \backslash B(0,3R)))} \lesssim_{u', T} \eps
\end{equation}
(say) for $i=0,1$, taking $R$ large enough depending on $u',T,\eps$ to ensure that the contributions to the Duhamel formula coming outside $B(0,9R)$ or inside $B(0,2R)$ are negligible, and taking $n$ sufficiently large as always.  

We let $\tilde p^{(n)}$ be the normalised pressure, defined by \eqref{pressure-point}; by Corollary \ref{as-mild}, $\tilde p^{(n)}(t)$ and $p^{(n)}(t)$ differ by a constant $C(t)$ for almost every $t$.  Using \eqref{pressure-point}, \eqref{waffle} and Lemma \ref{energy-est}, we see that
$$ \| \tilde p^{(n)} \|_{L^\infty_t H^2_x([0,T-\tau'] \times (B(0,7R) \backslash B(0,4R)))} \lesssim_{u',T} \eps.$$
if $R$ is large enough depending on $u',T,\eps$.

Applying Lemma \ref{divloc}, we may find divergence-free smooth vector fields $\tilde u^{(n)}: [\tau',T] \times \R^3 \to \R^3$ which agree with $u^{(n)}$ on $[\tau',T] \times B(0,5R)$ but vanish outside of $[\tau',T] \times B(0,6R)$, with
\begin{equation}\label{bingo}
 \| \partial_t^i \tilde u^{(n)}(\cdot+\tau') \|_{L^\infty_t H^5_x([0,T-\tau'] \times (B(0,8R) \backslash B(0,3R)))} \lesssim_{u', T} \eps
 \end{equation}
(say) for $n$ sufficiently large and $i=0,1$.  

Let $\eta$ be a smooth function that equals $1$ on $B(0,6R)$ and is supported on $B(0,7R)$, and obeys the usual derivative bounds in between.
We then consider the smooth solutions 
\begin{equation}\label{a1}
(\tilde u^{(n)}(\cdot +\tau'),\eta \tilde p^{(n)}(\cdot+\tau'), \tilde u^{(n)}(\tau'), \tilde f^{(n)}, T-\tau')
\end{equation}
where
$$ \tilde f^{(n)} := (\partial_t \tilde u^{(n)} + \tilde u^{(n)} \cdot \nabla \tilde u^{(n)} - \Delta \tilde u^{(n)} + \nabla(\eta p^{(n)}))(\cdot+\tau')$$
By construction, $\tilde f'$ and $\tilde f^{(n)}$ are smooth and supported on $[0,T-\tau'] \times (B(0,7R) \backslash B(0,5R))$, and the \eqref{a1} are smooth, compactly supported solutions.  From the preceding bounds on $\tilde u^{(n)}, \tilde p^{(n)}$ we see that
$$ \| \tilde f^{(n)} \|_{L^\infty_t H^1_x([0,T-\tau] \times \R^3)} \lesssim_{u',T} \eps$$
for $n$ sufficiently large.  

Also, using \eqref{upo}, \eqref{bingo} we have
$$ \| \tilde u^{(n)}(\tau') - u'_0 \|_{H^1_x(\R^3)} \lesssim_{u',T} \eps$$
for $n$ sufficiently large.  If $\eps$ is sufficiently small, we conclude from the local $H^1$ well-posedness theory (Theorem \ref{lwp-h1-r3}) that
$$ \| \tilde u^{(n)}(\cdot+\tau') - u' \|_{X^1([0,T-\tau'] \times \R^3)} \lesssim_{u',T} \eps$$
and in particular
$$ \| u^{(n)}(\cdot+\tau') - u' \|_{X^1([0,T-\tau'] \times B(0,R) )} \lesssim_{u',T} \eps$$
for $n$ large enough.  Sending $\eps$ to zero (and $R$ to infinity) we conclude that $u^{(n)}(\cdot+\tau')$ converges weakly to $u'$.  In particular, we see that any weak limit of the $u^{(n)}$ is smooth on $[\tau',T] \times \R^3$ (and furthermore, the weak limit is unique in this spacetime region).

The above analysis was for a single choice of $\tau$.  Choosing $\tau$ to be a sequence of times going to zero (and repeatedly taking subsequences of the $u^{(n)}$ and diagonalising as necessary) we may thus arrive at a subsequence $u^{(n)}$ with the property that there is a unique weak limit $u$ of the $u^{(n)}$, which is smooth on $(0,T] \times \R^3$.  If we then set $p$ by \eqref{pressure-point}, we see on taking distributional limits that $(u,p,u_0,0,T)$ is a Leray-Hopf weak solution to the initial data $(u_0,0,T)$.

To finish the argument, we need to show that $(u,p,u_0,0,T)$ is almost smooth at $(0,x_0)$ for every $x_0 \in \R^3$.  Fix $x_0$, and let $R > 0$ be a large radius.  As $u_0$ is smooth, $\| u_0 \|_{H^1(B(x_0,5R))}$ is finite, and hence $\| u_0^{(n)} \|_{H^1(B(x_0,5R))}$ is uniformly bounded.  Applying Theorem \ref{enstrophy-loc} (recalling that the $u^{(n)}$ have uniformly bounded energy), we conclude (for $R$ large enough) that there exists $0 < \tau < T$ such that $\| u^{(n)} \|_{X^1([0,\tau] \times B(x_0,4R))}$ is uniformly bounded in $n$.  Using Duhamel's formula as in Proposition \ref{uni-ball}, and noting that $u^{(n)}$ is uniformly smooth on $B(x_0,4R)$, we conclude that $\| u^{(n)} \|_{L^\infty_t C^k((0,\tau] \times B(x_0,3R))}$ is uniformly bounded for all $k \geq 0$.  Taking weak limits, we conclude that
$$ u \in L^\infty_t C^k((0,\tau] \times B(x_0,3R))$$
for all $k \geq 0$.  From this and \eqref{pressure-point} (and Lemma \ref{energy-est}), we also see that
$$ p \in L^\infty_t C^k((0,\tau] \times B(x_0,2R))$$
for all $k \geq 0$.  Using \eqref{ns}, we conclude that
$$ \partial_t u \in L^\infty_t C^k((0,\tau] \times B(x_0,2R))$$
for all $k \geq 0$.  A similar argument also shows that
$$ \partial_t u^{(n)} \in L^\infty_t C^k((0,\tau] \times B(x_0,2R))$$
uniformly in $n$.  From this, we see that the $\nabla_x^k u^{(n)}$ are uniformly Lipschitz in a neighbourhood of $(0,x_0)$.  Since $\nabla_x^k u^{(n)}$ converges weakly to the smooth function $\nabla_x^k u$ in $(0,T] \times \R^3$, and also converges strongly at time zero in $H^1_x(\R^3)$ to the smooth function $\nabla_x^k u_0$, we conclude that $\nabla_x^k u$ can be extended in a locally Lipschitz continuous manner from $(0,T] \times \R^3$ to $[0,T] \times \R^3$ in such a way that it agrees with $\nabla_x^k u_0$ at time zero.  

Now we consider derivatives $\nabla^k p$ of the pressure near $(0,x_0)$.  Let $\eps > 0$ be arbitrary.  Then by the monotone convergence theorem, we see that if $R' > 0$ is a sufficiently large radius, then
$$ \| u_0 \|_{L^2_x(\R^3 \backslash B(x_0,R'))} \leq \eps$$
and thus
$$ \| u_0^{(n)} \|_{L^2_x(\R^3 \backslash B(x_0,R'))} \lesssim \eps$$
for $n$ large enough.

By Theorem \ref{local}, we conclude that if $R'$ is large enough, there exists a time $0 < \tau < T$ such that
$$ \| u^{(n)} \|_{L^\infty_t L^2_x([0,\tau] \times (\R^3 \backslash B(x_0,2R')))} \lesssim \eps$$
and hence on taking weak limits
$$ \| u \|_{L^\infty_t L^2_x([0,\tau] \times (\R^3 \backslash B(x_0,2R')))} \lesssim \eps.$$
On the other hand, as $\nabla^k u$ is continuous at $t=0$, $u(t)$ converges in $C^k(B(x_0,2R'))$ to $u_0$ as $t \to 0$ for any $k \geq 0$.  From 
this and \eqref{pressure-point} (and the decay of derivatives of the kernel of $\Delta^{-1}$ away from the origin) we see that
$$ \limsup_{(t,x) \to (0,x_0); t>0} |\nabla^k p(t,x) - \nabla^k p_0(x_0)| \lesssim_k \eps$$
for any $k \geq 0$, where $p_0$ is defined from $u_0$ using \eqref{pressure-point}.  Sending $\eps \to 0$ and $R' \to \infty$ we conclude that $\nabla^k p$ extends continuously to $\nabla^k p_0(x_0)$ at $(0,x_0)$, and thus extends continuously to $\nabla^k p_0$ on all of the initial slice $\{0\} \times \R^3$.  By \eqref{ns} we conclude that $\partial_t \nabla^k u$ also extends continuously to the initial slice, with the Navier-Stokes equation \eqref{ns} being obeyed both for times $t>0$ and times $t=0$.  We have thus constructed an almost smooth finite energy solution $(u,p,u_0,0,T)$ as desired.
\end{proof}

\begin{remark}   We emphasise that Theorem \ref{energy-thm} only establishes \emph{existence} of a smooth finite energy solution (assuming Conjecture \ref{global-enstrophy-homog}), and not uniqueness; see Remark \ref{uniq}.  However, it is not difficult to see from the argument that one can at least ensure that the solution constructed is independent of the choice of time $T$, and can thus be extended to a single global smooth finite energy solution.  (Alternatively, from Lemma \ref{energy-est} we see that the enstrophy of the solution will become arbitrarily small for a sequence of times going to infinity, so for a sufficiently large time one can in fact construct a global smooth solution by standard perturbation theory techniques.)
\end{remark}

\begin{remark}  One can modify the above argument to also establish Conjecture \ref{global-energy-homog} with a non-zero Schwartz forcing term $f$, provided of course that one also assumes Conjecture \ref{global-enstrophy-homog} can be extended to the same class of $f$.  We have not however investigated the weakest class of forcing terms $f$ for which the argument works, though certainly finite energy seems insufficient.
\end{remark}

\section{Quantitative $H^1$ bounds}\label{quant-sec}

In this section we prove Theorem \ref{main}(vi).  We begin with some easy implications. Firstly, it is trivial that Conjecture \ref{global-h1-quant-global} implies Conjecture \ref{global-h1-quant-r3}, and from the local well-posedness and regularity theory in Theorem \ref{lwp-h1-r3} (or Corollary \ref{max-cauchy}) we see that Conjecture \ref{global-h1-quant-r3} implies Conjecture \ref{global-h1-r3}, which in turn implies Conjecture \ref{global-enstrophy-homog} (thanks to Proposition \ref{almost-smooth}).

Next, we observe from Theorem \ref{lwp-h1} and Lemma \ref{quant-reg} that given any $H^1$ data $(u_0,0,T)$, there exists a time $0 < \tau < T$ such that one has an $H^1$ mild solution $(u,p,u_0,0,\tau)$ with $u(\tau)$ smooth.  If Conjecture \ref{global-enstrophy-homog} holds, then one can then continue the solution in an almost smooth finite energy manner (and hence in an almost smooth $H^1$ manner, thanks to Corollary \ref{ens-bound}) from $\tau$ up to $T$.  Normalising the pressure of this latter solution using Lemma \ref{reduction} and gluing the two solutions together, we obtain an $H^1$ mild solution up to time $T$.  From this we see that Conjecture \ref{global-enstrophy-homog} implies Conjecture \ref{global-h1-r3}.

Now we show that Conjecture \ref{global-h1-quant-r3} implies Conjecture \ref{global-h1-quant-global}.  Suppose that one has homogeneous $H^1$ data $(u_0,0,T)$ with
$$ \|u_0\|_{H^1_x(\R^3)} \leq A < \infty.$$
By Conjecture \ref{global-h1-quant-r3} (which implies Conjecture \ref{global-h1-r3}) we may obtain a mild $H^1$ solution $(u,p,u_0,0,T)$, which is smooth for positive times.   Our objective is to show that
$$ \| u \|_{L^\infty_t H^1_x([0,T] \times \R^3)} \lesssim_A 1.$$

Let $\eps > 0$ be a quantity depending on $A$ to be chosen later.  We may assume that $T$ is sufficiently large depending on $\eps,A$, otherwise the claim will follow immediately from Conjecture \ref{global-h1-quant-r3}.  Using Lemma \ref{energy-est} and the pigeonhole principle, we may then find a time $0 < T_1 < T$ with $T_1 \lesssim_A 1$ such that
$$ \| \nabla u(T_1) \|_{L^2_x(\R^3)} \leq \eps.$$
Meanwhile, from energy estimates one has
$$ \| u(T_1) \|_{L^2_x(\R^3)} \lesssim_A 1.$$
On $[T_1,T]$, we split $u = u_1 + v$, where $u_1$ is the linear solution $u_1(t) := e^{(t-T_1)} u(T_1)$ and $v := u-u_1$.  From \eqref{energy-duh} one thus has
$$ \| u_1 \|_{X^0} \lesssim_A 1$$
and
$$ \| \nabla u_1 \|_{X^0} \lesssim \eps.$$

From \eqref{duhamel-2}, \eqref{energy-duh2} one has
$$ \|v\|_{X^1([T_1,T] \times \R^3)} \lesssim \| \bigO( u_1 \nabla u_1 + u_1 \nabla v + v \nabla u_1 + v \nabla v ) \|_{L^2_t L^2_x([T_1,T] \times \R^3)}.$$
We now estimate various contributions to the right-hand side.  We begin with the nonlinear term $\bigO( v \nabla v)$.  By H\"older (and dropping the domain $[T_1,T] \times \R^3$ for brevity) followed by Lemma \ref{energy-est} we have
\begin{align*}
\| \bigO( v \nabla v) \|_{L^2_t L^2_x} &\lesssim
\| \nabla v \|_{L^2_t L^6_x}^{1/2} \|\nabla v\|_{L^\infty_t L^2_x}^{1/2} \| v \|_{L^\infty_t L^6_x}^{1/2} \| v \|_{L^2_t L^6_x}^{1/2} \\
&\lesssim \|v\|_{X^1}^{3/2} \|v\|_{X_0^{1/2}}\\
&\lesssim_A \|v\|_{X^1}^{3/2}.
\end{align*}
A similar argument gives
\begin{align*}
\| \bigO( v \nabla u_1) \|_{L^2_t L^2_x} &\lesssim \| \nabla u_1\|_{X^0} \|v\|_{X^1}^{1/2} \|v\|_{X_0}^{1/2} \\
&\lesssim \eps \|v\|_{X^1} 
\end{align*}
and
\begin{align*}
\| \bigO( u_1 \nabla u_1) \|_{L^2_t L^2_x} &\lesssim \| \nabla u_1\|_{X^0} \|\nabla u_1\|_{X^0}^{1/2} \|u_1\|_{X_0}^{1/2} \|\\
&\lesssim_A \eps^{3/2}  
\end{align*}
and
\begin{align*}
\| \bigO( u_1 \nabla v) \|_{L^2_t L^2_x} &\lesssim \| \nabla v\|_{X^0} \|\nabla u_1\|_{X^0}^{1/2} \|u_1\|_{X_0}^{1/2} \\
&\lesssim_A \eps^{1/2} \|v\|_{X^1}
\end{align*}
and thus
$$ \|v\|_{X^1} \lesssim_A  \eps^{3/2}   + \eps^{1/2} \|v\|_{X^1} + \|v\|_{X^1}^{3/2}.$$
If $\eps$ is small enough depending on $A$, a continuity argument in the $T$ variable then gives
$$ \|v\|_{X^1} \lesssim_A  \eps^{3/2} $$
and thus
$$ \|u\|_{X^1([T_1,T])} \lesssim_A 1.$$
Using this and the triangle inequality, we conclude that Conjecture \ref{global-h1-quant-r3} implies Conjecture \ref{global-h1-quant-global}.

We now turn to the most difficult implication:

\begin{proposition}[Concentration compactness]\label{compact} Suppose that Conjecture \ref{global-h1-r3} is true.  Then Conjecture \ref{global-h1-quant-r3} is true.
\end{proposition}

We now prove this proposition.  The methods are essentially those of \cite{gallagher} (which are in turn based in \cite{bahouri}, \cite{gerard}), which treated the (more difficult) critical analogue of this implication; indeed, one can view Proposition \ref{compact} as a subcritical analogue of the critical result \cite[Corollary 1]{gallagher}.  For the convenience of the reader, though, we give a self-contained proof here, which does not need the full power of the machinery in the previously cited papers because we are now working in a subcritical regularity $H^1$ rather than a critical regularity such as $\dot H^{1/2}$, and as such one does not need to consider the role of the scaling symmetry \eqref{scaling}.

We first make the remark that to prove Conjecture \ref{global-h1-quant-r3}, it suffices to do so with the condition
\begin{equation}\label{ua}
 \|u_0\|_{H^1_x(\R^3)} \leq A 
\end{equation}
replaced by (say)
\begin{equation}\label{ub}
\|u_0\|_{H^{100}_x(\R^3)} \leq A 
\end{equation}
To see this, observe that if we take data $u_0$ in $H^1_x(\R^3)$, then from Theorem \ref{lwp-h1-r3} and Lemma \ref{quant-reg} there exists a time $T_1 > 0$ depending only on $A$ such that 
$$ \| u \|_{L^\infty_t H^1_x([0,\min(T,T_1)] \times \R^3)} \lesssim_A 1,$$
and such that 
$$ \|u(T_1) \|_{H^{100}_x(\R^3)} \lesssim_A 1$$
if $T > T_1$.  From this and time translation symmetry \eqref{time-translate} we see that we can deduce the $H^1_x(\R^3)$ version of Conjecture \ref{global-h1-quant-r3} from the $H^{100}_x(\R^3)$ version.

Now suppose for contradiction that the $H^{100}_x(\R^3)$ version of Conjecture \ref{global-h1-quant-r3} failed.  Carefully negating the quantifiers, we can find a sequence $(u^{(n)}, p^{(n)}, u_0^{(n)}, 0, T^{(n)})$ of smooth homogeneous $H^1$ solutions, with $T^{(n)}$ uniformly bounded, and $u_0^{(n)}$ uniformly bounded in $H^{100}_x(\R^3)$, such that
\begin{equation}\label{unt}
 \lim_{n \to \infty} \| u^{(n)} \|_{L^\infty_t H^1_x([0,T^{(n)}] \times \R^3)} = \infty.
\end{equation}
By Lemma \ref{reduction} we may assume that these solutions have normalised pressure.  

If we were working on a compact domain, such as $\R^3/\Z^3$, we could now extract a subsequence of the $u_0^{(n)}$ that converged strongly in a lower regularity space, such as $H^{99}_x(\R^3/\Z^3)$.  But our domain $\R^3$ is non-compact, and in particular has the action of a non-compact symmetry group, namely the translation group $\tau_{x_0} u(x) := u(x-x_0)$.  However, as is well known, we have a substitute for compactness in this setting, namely \emph{concentration compactness}.  Specifically:

\begin{proposition}[Profile decomposition]  Let $u^{(n)}_0 \in H^{100}_x(\R^3)$ be a sequence with
$$ \limsup_{n \to \infty} \|u^{(n)}_0\|_{H^{100}_x(\R^3)} \leq A,$$
and let $\eps > 0$.  Then, after passing to a subsequence, then there exists a decomposition
$$ u_0^{(n)} = \sum_{j=1}^J \tau_{x_j^{(n)}} w_{j,0} + r^{(n)}_0,$$
where $|J| \lesssim_{A,\eps} 1$, $w_{1,0},\ldots,w_{J,0} \in H^{100}_x(\R^3)$, $x_j^{(n)} \in \R^3$, and the remainder $r^{(n)}_0$ obeys the estimates
$$ \limsup_{n \to \infty} \|r^{(n)}_0\|_{H^{100}_x(\R^3)} \leq A$$
and 
\begin{equation}\label{rino}
 \limsup_{n \to \infty} \| r^{(n)}_0 \|_{L^\infty_x(\R^3)} \leq \eps.
\end{equation}
Furthermore, for any $1 \leq j < j' \leq J$, one has
\begin{equation}\label{diverge}
 |x_j^{(n)} - x_{j'}^{(n)}| \to \infty,
 \end{equation}
and for any $1 \leq j \leq J$, the sequence $\tau_{-x_j^{(n)}} r^{(n)}_0$ converges weakly in $H^{100}_x(\R^3)$ to zero.

Finally, if the $u^{(n)}_0$ are divergence-free, then the $w_{j,0}$ and $r^{(n)}_0$ are also divergence-free.
\end{proposition}

\begin{proof}  See e.g. \cite{gerard}.  We sketch the (standard) proof as follows.  If $\|u^{(n)}_0\|_{L^\infty_x(\R^3)} \leq \eps$ for all sufficiently large $n$ then there is nothing to prove (just take $J=0$ and $r^{(n)}_0 := u^{(n)}_0$).  Otherwise, after passing to a subsequence, we can find a sequence $x_1^{(n)} \in \R^3$ such that $|u^{(n)}_0(x_1^{(n)})| \geq \eps/2$ (say).  The sequence $\tau_{-x_1^{(n)}} u^{(n)}_0$ is then bounded in $H^{100}_x(\R^3)$ and bounded away from zero at the origin; by passing to a further subsequence, we may assume that it converges weakly in $H^{100}_x(\R^3)$ to a limit $w_1$, which then has an $H^{100}_x(\R^3)$ norm of $\gtrsim_{A,\eps} 1$ and is asymptotically orthogonal in the Hilbert space $H^{100}_x(\R^3)$ to $\tau_{-x_1^{(n)}} u^{(n)}_0$.  We can then decompose
$$ u^{(n)}_0 = \tau_{x_1^{(n)}} w_{1,0} + u^{(n),1}_0,$$
and from an application of the cosine rule in the Hilbert space $H^{100}_x(\R^3)$ one can verify that
$$ \limsup_{n \to \infty} \| u^{(n),1}_0\|_{H^{100}_x(\R^3)}^2 \leq A^2 - c$$
for some $c > 0$ depending only on $\eps, A$.  We can then iterate this procedure $O_{J,\eps}(1)$ times to obtain the desired decomposition.
\end{proof}

We apply this proposition with a value of $\eps > 0$ depending on $A, T$ to be chosen later.  The $w_{j,0}$ lie in $H^{100}_x(\R^3)$, and thus by the assumption that Conjecture \ref{global-h1-r3} is true, we can find mild $H^1$ solutions $(w_j, p_j, w_{j,0}, 0, T)$ with this data.  By Theorem \ref{lwp-h1} we have
$$ \|w_j\|_{X^{100}} < \infty$$
for each $1 \leq j \leq J$, and to abbreviate the notation we adopt the convention that the spacetime domain is understood to be $[0,T] \times \R^3$.

Next, we consider the remainder term $r^{(n)}_0$.  From \eqref{energy-duh} one has
$$ \|e^{t\Delta} r^{(n)}_0 \|_{X^{100}} \lesssim A$$
while from \eqref{rino} one has
$$ \|e^{t\Delta} r^{(n)}_0 \|_{L^\infty_t L^\infty_x} \lesssim \eps$$
for $n$ sufficiently large.  Interpolating between the two, we soon conclude that
$$ \|e^{t\Delta} r^{(n)}_0 \|_{X^1} \lesssim_{A,T} \eps^c$$
for some absolute constant $c>0$.  If we take $\eps$ sufficiently small depending on $A, T$, we can use stability of the zero solution (see Theorem \ref{lwp-h1}; one could also have used here the results from \cite{chemin}) to conclude the existence of a mild $H^1$ solution $(r^{(n)}, p_*^{(n)}, r^{(n)}_0, 0, T)$ with this data, with the estimates
\begin{equation}\label{rn-small}
\| r^{(n)} \|_{X^1} \lesssim_{A,T} \eps^c;
\end{equation}
from Theorem \ref{lwp-h1} we then also have
$$ \| r^{(n)} \|_{X^{100}} \lesssim_{A,T} 1.$$
We now form the solution
$$ (\tilde u^{(n)}, \tilde p^{(n)}, u^{(n)}_0, \tilde f^{(n)}, T)$$
where the velocity field $\tilde u^{(n)}$ is given by
$$ \tilde u^{(n)} := \sum_{j=1}^J \tau_{x_j^{(n)}} w_{j} + r^{(n)},$$
the pressure field $\tilde p^{(n)}$ is given by \eqref{pressure-point}, and the forcing term $\tilde f^{(n)}$ is given by the formula
$$ \tilde f^{(n)} := \partial_t \tilde u^{(n)} - \Delta \tilde u^{(n)} - P B( \tilde u^{(n)}, \tilde u^{(n)} ).$$
This is clearly a mild $H^1$ solution, with
$$ \| \tilde u^{(n)} \|_{X^{100}} \lesssim_{A,T,\eps} 1.$$
We now estimate $\tilde f^{(n)}$.  From \eqref{ns-project} for the solutions $\tau_{x_j^{(n)}} w_{j} + r^{(n)}$, we have an expansion of $\tilde f^{(n)}$ purely involving nonlinear interaction terms:
\begin{align*}
\tilde f^{(n)} &= \sum_{1 \leq j < j' \leq J} P \bigO( \nabla ( \tau_{x_j^{(n)}} w_{j}, \tau_{x_{j'}^{(n)}} w_{j'} ) ) \\
&\quad + \sum_{1 \leq j < J} P \bigO( \nabla ( \tau_{x_j^{(n)}} w_{j}, r^{(n)} ) ).
\end{align*}
In particular, from the triangle inequality and translation invariance we have
\begin{align*}
\| \tilde f^{(n)} \|_{L^2_t L^2_x} &\lesssim \sum_{1 \leq j < j' \leq J} \| \bigO( \nabla ( w_{j}, \tau_{x_{j'}^{(n)} - x_j^{(n)}} w_{j'} ) ) \|_{L^2_t L^2_x} \\
&\quad + \sum_{1 \leq j < J} \| \bigO( \nabla ( w_{j}, \tau_{-x_j^{(n)}} r^{(n)} ) ) \|_{L^2_t L^2_x}.
\end{align*}
But by \eqref{diverge} and Sobolev embedding, $\tau_{x_{j'}^{(n)} - x_j^{(n)}} w_{j'}$ and $\tau_{-x_j^{(n)}} r^{(n)}$ are bounded in $L^\infty_t L^\infty_x$ and converge locally uniformly to zero, and so we conclude that
$$ \lim_{n \to \infty} \| \tilde f^{(n)} \|_{L^2_t L^2_x} = 0.$$
From this and the stability theory in Theorem \ref{lwp-h1-r3}, we conclude that for $n$ large enough, there is an $H^1$ mild solution $(u^{(n)}, p^{(n)}, u_0^{(n)}, 0, T)$ with
$$ \lim_{n \to \infty} \| \tilde u^{(n)} - u^{(n)} \|_{X^1} = 0,$$
and in particular
$$ \limsup_{n \to \infty} \| u^{(n)} \|_{L^\infty_t H^1_x([0,T] \times \R^3)} < \infty.$$
By the uniqueness theory in Theorem \ref{lwp-h1-r3}, this solution must agree with the original solutions $(u^{(n)}, p^{(n)}, u_0^{(n)}, 0, T^{(n)})$ on $[0,T^{(n)}] \times \R^3$; but then we contradict \eqref{unt}.  Proposition \ref{compact} follows.

\section{Non-existence of smooth solutions}\label{counter-sec}

In this section we establish Theorem \ref{counter}.  Informally, the reason for the irregularity is as follows.  Assuming normalised pressure, one concludes from \eqref{pressure-point} that
$$ p = \bigO( \Delta^{-1} \nabla^2( u u ) ).$$
If one then differentiates this twice in time, using \eqref{ns} to convert time derivatives of $u$ into $\Delta u$ plus lower order terms, integration by parts to redistribute derivatives, we eventually obtain (formally, at least) a formula of the form
$$ \partial_t^2 p = \bigO( \Delta^{-1} \nabla^2( \Delta u \Delta u ) ) + \hbox{l.o.t.}.$$
But if $u_0$ is merely assumed to be smooth and in $H^1$, then $\Delta u$ can grow arbitrarily fast at infinity at time $t=0$, and this should cause $p$ to fail to be $C^2_t$ at time zero.

We turn to the details.  To eliminate the normalised pressure assumption, we will work with $\nabla p$ instead of $p$, thus we will seek to establish bad behaviour for $\nabla \partial_t^2 p$ at time $t=0$.  For technical reasons it is convenient to work in the weak topology in space.  The key quantitative step is the following:

\begin{proposition}[Quantitative failure of regularity]\label{quant-fail}  Let $u_0: \R^3 \to \R^3$ be smooth, divergence-free, and compactly supported, and let $\psi: \R^3 \to \R$ be smooth, compactly supported, and have total mass $\int_{\R^3} \psi = 1$.  Let $R, M, \eps > 0$.  Then there exists a smooth divergence-free compactly supported function $u_1$ which vanishes on $B(0,R)$ with
$$ \|u_1\|_{H^1_x(\R^3)} \lesssim \eps$$
and such that if $(u,p,u_0+u_1,0,T)$ is a mild $H^1$ (and hence smooth, by Proposition \ref{almost-smooth}) solution with data $(u_0+u_1,0,T)$, then
\begin{equation}\label{abs}
\left|\int_{\R^3} \nabla \partial_t^2  p(0,x) \psi(x)\ dx\right| > M.
\end{equation}
\end{proposition}

Let us assume this proposition for now and conclude Theorem \ref{counter}.  
We will use an argument reminiscent of that used to establish the Baire category theorem or the uniform boundedness principle.  Let $\psi: \R^3 \to \R$ be a fixed smooth, compactly supported function with total mass $1$.
We will need a rapidly decreasing sequence
$$ \eps^{(1)} > \eps^{(2)} > \ldots > 0$$
of small quantities to be chosen later, with each $\eps^{(n)}$ sufficiently small depending on the previous $\eps^{(1)},\ldots,\eps^{(n-1)}$.  Applying Proposition \ref{quant-fail} recursively starting with $u_0=0$, one can then find a sequence of smooth, divergence-free, and compactly supported functions $u_1^{(n)}$ for $n=1,2,\ldots$ such that
$$ \|u_1^{(n)} \|_{H^1_x(\R^3)} \lesssim \eps^{(n)}$$
with $u_1^{(n)}$ vanishing on $B(0,1/\eps^{(n)})$, such that if $(u^{(n)},p^{(n)}, u_0^{(n)}, 0, T^{(n)})$ is a mild $H^1$ (and hence smooth) solution with data
$$ u_0^{(n)} := u_1^{(1)} + \ldots + u_1^{(n)},$$
then
\begin{equation}\label{alo}
 |\int_{\R^3} \nabla \partial_t^2  p^{(n)}(0,x) \psi(x)\ dx| > 1/\eps^{(n)}.
\end{equation}
Furthermore, each $u_1^{(n)}$ depends only on $\eps^{(1)},\ldots,\eps^{(n)}$, and in particular is independent of $\eps^{(n+1)}$.

By the triangle inequality (and assuming the $\eps^{(n)}$ decay fast enough), the data $u_0^{(n)}$ is strongly convergent in $H^1_x(\R^3)$ to a limit $u_0 = \sum_{n=1}^\infty u^{(n)}_1 \in H^1_x(\R^3)$, with
$$ \| u_0 - u^{(n)}_0 \|_{H^1_x(\R^3)} \lesssim \eps^{(n+1)}.$$
If we make each $\eps^{(n+1)}$ sufficienly small depending on $u^{(n)}_0$, and hence on $\eps^{(1)},\ldots,\eps^{(n)}$, then the $u^{(n)}_1$ will have disjoint supports; as each $u^{(n)}_1$ is smooth and divergence-free, this implies that $u_0 = \sum_{n=1}^\infty u^{(n)}_1$ is also smooth and divergence-free.

Applying Theorem \ref{lwp-h1}, we may then take the times $T^{(n)} = 1$ (if the $\eps^{(n)}$ are small enough), and $(u^{(n)},p^{(n)},u_0^{(n)},0,1)$ will converge to a mild $H^1$ solution $(u,p,u_0,0,1)$ in the sense that $u^{(n)}$ converges strongly in $X^1([0,1] \times \R^3)$ to $u$.  Indeed, from the Lipschitz stability property we see (if the $\eps^{(n)}$ decay fast enough) that
$$ \| u - u^{(n)} \|_{X^1([0,1] \times \R^3)} \lesssim \eps^{(n+1)}$$
Also, $u, u^{(n)}$ are bounded in $X^1([0,1] \times \R^3)$ by $O(1)$.  
Using \eqref{pressure-point} and Sobolev embedding, this implies that
$$ \| p - p^{(n)} \|_{L^\infty_t L^3_x([0,1] \times \R^3)} \lesssim \eps^{(n+1)},$$
and so if one sets
$$ F^{(n)}(t) := \int_{\R^3} \nabla p^{(n)}(t,x) \psi(x)\ dx$$
and
$$ F(t) := \int_{\R^3} \nabla p(t,x) \psi(x)\ dx$$
then from integration by parts we have
\begin{equation}\label{linf}
 \| F - F^{(n)} \|_{L^\infty_t([0,1])} \lesssim \eps^{(n+1)}.
 \end{equation}
Meanwhile, each $F^{(n)}$ is smooth, and $F$ continuous, from Proposition \ref{almost-smooth}, and from \eqref{alo} one has
$$ |\partial_t^2 F^{(n)}(0)| \geq 1/\eps^{(n)}.$$
In particular, if $\eps^{(n+1)}$ is sufficiently small depending on $F^{(n)}$ (which in turn depends on $\eps^{(1)},\ldots,\eps^{(n)}$), one has from Taylor's theorem with remainder that
$$ \frac{|F^{(n)}( 2 (\eps^{(n+1)})^{0.1} ) - 2 F^{(n)}( (\eps^{(n+1)})^{0.1} ) + F^{(n)}( 0 )|}{(\eps^{(n+1)})^{0.2}} \gtrsim 1 / \eps^{(n)}.$$
Applying \eqref{linf}, we conclude that
$$ \frac{|F( 2 (\eps^{(n+1)})^{0.1} ) - 2 F( (\eps^{(n+1)})^{0.1} ) + F( 0 )|}{(\eps^{(n+1)})^{0.2}} \gtrsim 1 / \eps^{(n)}$$
if $\eps^{(n+1)}$ is sufficiently small depending on $\eps^{(1)},\ldots,\eps^{(n)}$.  In particular,
$$ \limsup_{h \to 0^+} \frac{|F(2h) -2F(h) + F(0)|}{h^2} = +\infty$$
which by Taylor's theorem with remainder implies that $F$ is not smooth at $0$.

We claim that the data $u_0$ gives the desired counterexample to Theorem \ref{counter}.  Indeed, suppose for contradiction that there was a smooth solution $(\tilde u, \tilde p, u_0, 0, T)$ for some $T>0$.  By shrinking $T$ we may assume $T \leq 1$.  By Lemma \ref{reduction} we see that $\tilde p(t)$ has normalised pressure up to a constant for almost every $t$, and thus after adjusting $\tilde p(t)$ by that constant, $(\tilde u, \tilde p, u_0, 0, T)$ is a mild $H^1$ solution.  Using the uniqueness property in Theorem \ref{lwp-h1}, we conclude that $u = \tilde u$, and $p(t)$ and $\tilde p(t)$ differ by a constant for almost every $t$, and hence (by continuity of both $p$ and $\tilde p$) for every $t$.  In particular, $\nabla p = \nabla \tilde p$, and so
$$ F(t) = \int_{\R^3} \nabla \tilde p(t,x) \psi(x)\ dx.$$
But as $\tilde p$ is smooth on $[0,T] \times \R^3$, $F$ is smooth at $0$, a contradiction.

\begin{remark} The above argument showed that $\nabla p$ failed to be smooth at $t=0$; by using \eqref{ns} we conclude that the velocity field $u$ must then also be non-smooth at $t=0$ (though the velocity $u$ has one more degree of time regularity than the pressure $p$).  Thus the failure of regularity is not just an artefact of pressure normalisation.  Using the vorticity equation \eqref{vorticity-eq} one can then show a similar failure of time regularity for the vorticity, although again one gains an additional degree of time differentiability over the velocity $u$.  

The irregularities in time stem from the unbounded growth of high derivatives of the initial data.  If one assumes that \emph{all} spatial derivatives of $u_0$ are in $L^2_x(\R^3)$, i.e. that $u_0 \in H^\infty(\R^3)$ then one can prove iteratively\footnote{We thank Richard Melrose for this observation.} that all time derivatives of $u$ and $p$ at time zero are bounded, and also have first spatial derivatives in $H^\infty(\R^3)$ (basically because the first derivative of the kernel of the Leray projection is integrable at infinity).  In particular, $u$ and $p$ now remain smooth at time $0$.
\end{remark}

It remains to establish Proposition \ref{quant-fail}.  Fix $u_0,\psi,R,M,\eps$, and let $u_1$ be a smooth divergence-free compactly supported function $u_1$ vanishing on $B(0,R)$ with $H^1_x(\R^3)$ norm $O(\eps)$ to be chosen later.  Let $(u,p,u_0+u_1,0,T)$ be a mild $H^1$ solution with this given data.  By Theorem \ref{lwp-h1}, this is a smooth solution, with all derivatives of $u, p$ lying in $L^\infty_t L^2_x$.  From Lemma \ref{reduction} we thus have
\begin{equation}\label{above}
 \nabla p = -\nabla \Delta^{-1} \partial_i \partial_j(u_i u_j)
 \end{equation}
for almost all times $t$.  But both sides are smooth in $[0,T] \times \R^3$, so this formula is valid for all times $t$ (and in particular at $t=0$).  In particular, we may apply a Leray projection $P$ to \eqref{ns} and conclude that
\begin{equation}\label{boo}
 \partial_t u = \Delta u + P B( u, u ).
\end{equation}

We differentiate \eqref{above} once in time to obtain
$$ \nabla \partial_t p = -2\nabla \Delta^{-1} \partial_i \partial_j(u_i \partial_t u_j).$$
Expanding out $\partial_t u_j$ using \eqref{ns}, we obtain
$$ \nabla \partial_t p = -2\nabla \Delta^{-1} \partial_i \partial_j(u_i \Delta u_j) + \bigO( \Delta^{-1} \nabla^3( u P B(u,u) ) ).$$
Writing
$$ \partial_i \partial_j( u_i \Delta u_j ) = -2 \partial_i \partial_j( (\partial_k u_i) (\partial_k u_j) ) + \bigO( \nabla^4( u u ) )$$
we thus have
$$ \nabla \partial_t p = 2\nabla \Delta^{-1} \partial_i \partial_j(\partial_k u_i \partial_k u_j) + \bigO( \Delta^{-1} \nabla^5( u u ) )+ \bigO( \Delta^{-1} \nabla^3( u P B(u,u) ) ).$$
We differentiate this in time again and use \eqref{boo} to obtain
\begin{align*}
\nabla \partial_t^2 p &= 4\nabla \Delta^{-1} \partial_i \partial_j(\partial_k u_i \partial_k \Delta u_j)  \\
&\quad + \bigO( \Delta^{-1} \nabla^3( (\nabla u) \nabla PB(u,u) ) ) \\
&\quad + \bigO( \Delta^{-1} \nabla^5( u \partial_t u ) \\
&\quad + \bigO( \Delta^{-1} \nabla^3( (\partial_t u) P B(u,u) ) ) \\
&\quad + \bigO( \Delta^{-1} \nabla^3( u P B(u,\partial_t u) ) ).
\end{align*}
We can write $\partial_k u_i \partial_k \Delta u_j = -(\Delta u_i) (\Delta u_j) + \bigO( \nabla( \nabla u \Delta u ) )$, so that
\begin{align*}
\nabla \partial_t^2 p &= -4\nabla \Delta^{-1} \partial_i \partial_j(\Delta u_i \Delta u_j)  \\
&\quad + \bigO( \Delta^{-1} \nabla^4( \nabla u \Delta u ) \\
&\quad + \bigO( \Delta^{-1} \nabla^3( (\nabla u) \nabla PB(u,u) ) ) \\
&\quad + \bigO( \Delta^{-1} \nabla^5( u \partial_t u ) \\
&\quad + \bigO( \Delta^{-1} \nabla^3( (\partial_t u) P B(u,u) ) ) \\
&\quad + \bigO( \Delta^{-1} \nabla^3( u P B(u,\partial_t u) ) ).
\end{align*}
Integrating this against $\psi$, we may thus expand 
$$
\int_{\R^3} \nabla \partial_t^2  p(0,x) \psi(x)\ dx = 4 X_0 + \sum_{i=1}^5 \bigO( X_i )$$
where
\begin{align*}
X_0 &:= \int_{\R^3} (\partial_i \partial_j \nabla \Delta^{-1} \psi) \Delta u_i \Delta u_j \\
X_1 &:= \int_{\R^3} (\nabla^4 \Delta^{-1} \psi) \nabla u \Delta u\\
X_2 &:= \int_{\R^3} (\nabla^3 \Delta^{-1} \psi) \nabla u \nabla PB(u,u)\\
X_3 &:= \int_{\R^3} (\nabla^5 \Delta^{-1} \psi) u \partial_t u\\
X_4 &:= \int_{\R^3} (\nabla^3 \Delta^{-1} \psi) (\partial_t u) PB(u,u)\\
X_5 &:= \int_{\R^3} (\nabla^3 \Delta^{-1} \psi) u PB(u,\partial_t u),
\end{align*}
with all expressions being evaluated at time $0$.

From \eqref{boo} and Sobolev embedding one has
$$ \| \partial_t u(0) \|_{L^2_x(\R^3)} \lesssim_{u_0} 1 + \|u_1\|_{H^2_x(\R^3)}.$$
Meanwhile, if $\eps$ is small enough, we see that
$$ \| u(0) \|_{H^1_x(\R^3)} \lesssim_{u_0} 1$$
and thus from the Gagliardo-Nirenberg inequality
$$ \|u(0)\|_{L^\infty_x(\R^3)} \lesssim_{u_0} (1 + \|u_1\|_{H^2_x(\R^3)})^{1/2}.$$
From many applications of the Sobolev and H\"older inequalities (and, in the case of $X_5$, an integration by parts to move the derivative off of $\partial_t u$), we conclude that
$$ |X_i| \lesssim_{u_0,\psi} (1 + \|u_1\|_{H^2_x(\R^3)})^{3/2}$$
for $i=1,2,3,4,5$.  In a similar spirit, one has
$$ X_0 = \int_{\R^3} (\partial_i \partial_j \nabla \Delta^{-1} \psi) \Delta u_{1,i} \Delta u_{1,j} + O_{u_0,\psi}(1 + \|u_1\|_{H^2_x(\R^3)}).$$

To demonstrate \eqref{abs}, it thus suffices to exhibit a sequence $u_1^{(n)}: \R^3 \to \R^3$ of smooth divergence-free compactly supported vector fields supported outside of $B(0,R)$ such that
$$ |\int_{\R^3} (\partial_i \partial_j \nabla \Delta^{-1} \psi) \Delta u_{1,i}^{(n)} \Delta u_{1,j}^{(n)} | \gtrsim_{R,\psi} \|u_1^{(n)}\|_{H^2_x(\R^3)}^2,$$
with
$$\|u_1^{(n)}\|_{H^1_x(\R^3)} \to 0$$
and
$$\|u_1^{(n)}\|_{H^2_x(\R^3)} \to \infty.$$

We construct $u_1^{(n)}$ explicitly as the ``wave packet''
$$ u_1^{(n)}(x) := n^{-5/2} \nabla \times \Psi^{(n)}( x_0 )$$
where $e_1,e_2,e_3$ is the standard basis, $x_0 \in \R^3$ is a point (independent of $n$) outside of $B(0,R+1)$ to be chosen later, and
$$ \Psi^{(n)}(x) = \chi(x) \sin(n \xi \cdot x) \eta$$
where $\xi \in \R^3$ is a non-zero frequency (independent of $n$) to be chosen later, $\eta \in \R^3$ is a non-zero direction, and $\chi: \R^3 \to \R$ is a smooth bump function supported on $B(0,1)$ to be chosen later.  Note from construction that $u_1^{(n)}$ is smooth, divergence-free, and supported on $B(x_0,1)$, and thus vanishing on $B(0,R)$ for $R_0 > R+1$.  One can compute that
$$ \| u_1^{(n)} \|_{H^1_x(\R^3)} \ll_\chi n^{-1/2}$$
and
$$ \| u_1^{(n)} \|_{H^2_x(\R^3)} \gg_\chi n^{1/2}$$
as long as $\chi$ is not identically zero.  To conclude the theorem, it thus suffices to show that
$$ |\int_{\R^3} (\partial_i \partial_j \nabla \Delta^{-1} \psi) \Delta u_{1,i}^{(n)} \Delta u_{1,j}^{(n)} | \gg_{R_0,\psi,\chi} n$$
if $R_0$ and $n$ are large enough.   

Observe that 
$$ u_1^{(n)}(x) := n^{-3/2} \sin(n \xi \cdot (x-x_0)) \chi(x-x_0) (\xi \times \eta) + O(n^{-5/2})$$
and similarly
$$ \Delta u_1^{(n)}(x) := - n^{1/2} |\xi|^2 \sin(n \xi \cdot (x-x_0)) \chi(x-x_0) (\xi \times \eta) + O(n^{-1/2})$$
and so by choosing $\chi$ appropriately, and using the Riemann-Lebesgue lemma, it suffices to find $x_0, \xi, \eta \in \R^3$ such that
$$ (\partial_i \partial_j \nabla \Delta^{-1} \psi) (\xi \times \eta)_i (\xi \times \eta)_j ( x_0 ) \neq 0.$$
But as $\psi$ has mean one, we see that $\nabla^3 \Delta^{-1} \psi(x_0)$ is not identically zero for $x_0$ large enough, and the claim follows.

\end{document}